\documentclass[letterpaper]{amsart}
\usepackage[T1]{fontenc}
\usepackage[utf8]{inputenc}
\usepackage{lmodern}
\usepackage{amsmath,amssymb,amsthm,mathtools,mathrsfs}
\usepackage{geometry}
\usepackage{booktabs,longtable,array,enumitem,tabularx}
\usepackage{microtype}
\usepackage{xcolor}
\usepackage{graphicx}
\usepackage{float}
\usepackage{tikz}
\usetikzlibrary{arrows.meta,positioning,calc,shapes.geometric}
\usepackage{hyperref}
\usepackage{cleveref}
\makeatletter
\renewcommand{\paragraph}{\@startsection{paragraph}{4}{\z@}%
  {2.0ex \@plus .7ex \@minus .2ex}%
  {0.55ex}%
  {\normalfont\normalsize\bfseries}}
\makeatother
\hypersetup{colorlinks=true,linkcolor=blue,citecolor=blue,urlcolor=blue,
  pdftitle={Farey Symbols for the Picard Group PSL2(Z[i])},
  pdfauthor={Devendra Tiwari and Helena Verrill}}
\newtheorem{theorem}{Theorem}[section]
\newtheorem{proposition}[theorem]{Proposition}
\newtheorem{lemma}[theorem]{Lemma}
\newtheorem{corollary}[theorem]{Corollary}
\theoremstyle{definition}
\newtheorem{definition}[theorem]{Definition}
\newtheorem{example}[theorem]{Example}
\newtheorem{problem}[theorem]{Problem}
\theoremstyle{remark}
\newtheorem{remark}[theorem]{Remark}
\newcommand{\HH}{\mathbb H}
\newcommand{\PP}{\mathbf P}
\newcommand{\ZZ}{\mathbf Z}
\newcommand{\QQ}{\mathbf Q}
\newcommand{\OO}{\mathcal O}
\newcommand{\PSL}{\mathrm{PSL}}
\newcommand{\PGL}{\mathrm{PGL}}

\newcommand{\cD}{\mathcal D}
\DeclareMathOperator{\Isom}{Isom}
\DeclareMathOperator{\Stab}{Stab}

\title{Farey Symbols for the Picard Group $\PSL_2(\ZZ[i])$}
\author{Devendra Tiwari}
\address{Statistics and Mathematics Department\\
Indian Statistical Institute\\
Bangalore 560 059\\
INDIA}
\email{devendra9.dev@gmail.com}
\author{Helena Verrill}
\address{Mathematics Institute\\
University of Warwick\\
Coventry CV4 7AL\\
ENGLAND}
\email{H.A.Verrill@warwick.ac.uk}
\subjclass[2020]{Primary 11F06; Secondary 20H05, 20H10, 11F67}
\keywords{Bianchi groups, Picard group, Farey symbols, hyperbolic 3-space, fundamental polyhedra, modular symbols, Steinberg module}
\date{}

\begin{document}
\begin{abstract}
We develop Picard Farey symbols for finite-index subgroups of
\(G=\PSL_2(\ZZ[i])\) from the Gaussian Farey tessellation of \(\HH^3\) by
ideal octahedra.  The ambient geometry naturally involves two different
index-two extensions of \(G\): the orientation-preserving group
\(\PGL_2(\ZZ[i])\) and a reflection extension.  This distinction explains the
two \(G\)-orbits of Farey faces, which merge into one \(\PGL_2\)-orbit, and
the classical extended modular tessellation appears as a totally geodesic
slice.  A Picard Farey symbol records finitely many Gaussian-rational
octahedral occurrences, their local stabilizers, and cooriented face-pairing
transformations.  We prove that such a symbol reconstructs the finite
Picard-cell action and hence the subgroup; moreover, when subgroup membership
is decidable, the symbol can be found by a terminating finite double-coset
procedure.  For torsion-free subgroups we obtain whole-octahedron fundamental domains.
The marked symbol also supplies finite orbit relations for the established
Gaussian Steinberg edge--face presentation, with its corresponding
Bianchi modular-symbol relations.  When the
natural modular-to-Picard coset map is bijective, a side-connected Kulkarni
special polygon lifts, using the same modular representatives, to a Picard
fundamental polyhedron whose intersection with the invariant real plane is
exactly that polygon; the level-$(2+i)$/level-$5$ family is worked out
explicitly.  Comparison examples show that octahedral boundary minimization,
absolute geometric face minimization, and preservation of such a modular
section are distinct problems.  Canonical reduction and effective Hecke
reduction remain open.
\end{abstract}
\maketitle

\section{Introduction}\label{sec:intro}

\subsection{Kulkarni's Farey-symbol picture}\label{subsec:intro-kulkarni}
Kulkarni's theory of special polygons for finite-index subgroups of
\(\PSL_2(\ZZ)\) is the geometric point of departure for this paper.  The
arithmetic boundary points, the hyperbolic tessellation and the subgroup
presentation are encoded simultaneously: a special polygon has ideal vertices
in \(\PP^1(\QQ)\), consecutive vertices are joined by unimodular Farey edges,
and the side identifications are recorded by a Farey symbol.  In the ambient
modular group
\[
   \PSL_2(\ZZ)\cong C_2*C_3,
\]
and the two finite stabilizer types already occur in the smallest reflection
triangle.

Put
\[
   \rho=\frac12+\frac{\sqrt3}{2}i,
   \qquad
   \Delta=(i,\rho,\infty)\subset\HH^2 .
\]
The angles of \(\Delta\) are \(\pi/2,\pi/3,0\).  Thus its three vertices have
ambient stabilizers of types
\[
        C_2,\qquad C_3,\qquad \ZZ,
\]
respectively.  Its finite side \([i,\rho]\) is Kulkarni's \(f\)-edge.  The
translates of the finite \(f\)-edges form a tree, and for a finite-index
subgroup \(\Phi\) the quotient is a finite connected graph.  A maximal tree
in that graph determines the cuts used in Kulkarni's Sections~2--3: terminal
even rays and the appropriate terminal odd rays are opened, and the developed
surface becomes a convex special polygon.  The resulting boundary intervals
are the even, odd and freely paired intervals of the Farey symbol.  In this
way the maximal-tree argument, the local angle calculation and the side
pairings all enter the subgroup presentation \cite{Kulkarni1991}.

The first orientation-preserving object is the double
\[
   F=\Delta\cup s\Delta=(0,p,\infty),
\]
which is a fundamental polygon for \(M=\PSL_2(\ZZ)\).  Thus the passage
from the reflection tessellation to the modular group is already visible by
doubling one basic tile.  The subgroup \(\Gamma_0(2)\) then gives a useful
worked comparison: it has index three in \(M\), and the ideal triangle
\((0,1,\infty)\) may be developed as \(F\cup XF\cup X^2F\), hence as
six copies of \(\Delta\).  Its two vertical sides form a free pair under
\(T\), while the side \((0,1)\) contains the order-two elliptic point and
is self-paired by an element of \(\Gamma_0(2)\); there is no order-three
elliptic point for this subgroup.  At any subgroup vertex the stabilizer is
the intersection with the ambient \(C_2\), \(C_3\), or parabolic
stabilizer.

\begin{figure}[H]
\centering
\includegraphics[width=.96\textwidth]{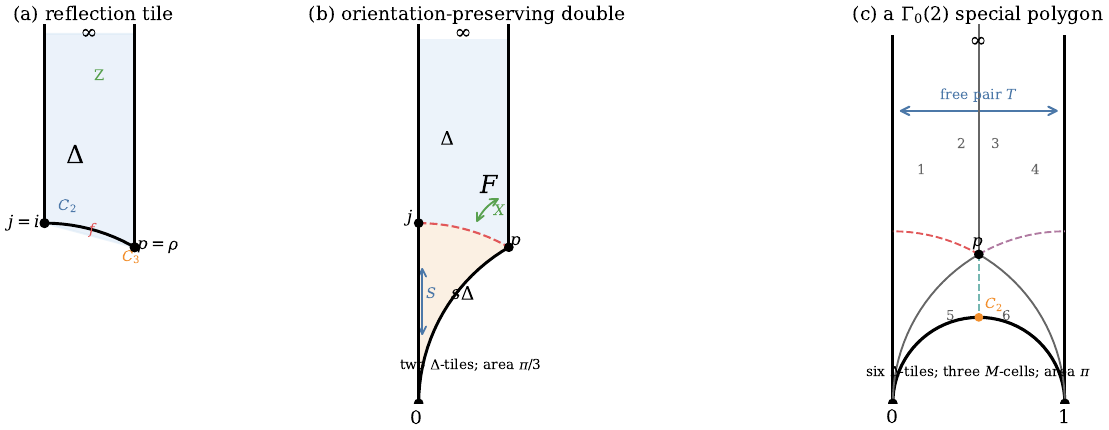}
\caption{The modular reflection triangle \(\Delta\), its orientation-preserving double \(F\), and a \(\Gamma_0(2)\) special polygon assembled from six copies of \(\Delta\).  The labels \(C_2,C_3,\ZZ\) in (a) are stabilizers in \(\PSL_2(\ZZ)\).  Dashed arcs are tile seams; the marked identifications in (b) and (c) belong to the orientation-preserving polygons.}
\label{fig:intro-kulkarni}
\end{figure}

A useful modern reformulation due to Belabas--Bernardi--Perrin-Riou and
Bernardi--Perrin-Riou packages the same information as a cyclic necklace of
rational arcs with an involution and elliptic markings
\cite{BelabasBernardiPerrinRiou2020,BernardiPerrinRiou2020}.  This viewpoint
emphasizes that a Farey symbol is not merely a generating set: it is an
arithmetic-geometric encoding of the quotient.  Our aim is to preserve this
interaction when \(\PP^1(\QQ)\) is replaced by \(\PP^1(\QQ(i))\).

\subsection{From the modular surface to the Picard group}\label{subsec:intro-picard}
We now pass to
\[
      G=\PSL_2(\ZZ[i])
\]
acting on upper half-space
\[
      \HH^3=\{(z,t):z=x+iy\in\mathbf C,\ t>0\}.
\]
The ideal boundary is \(\mathbf C\cup\{\infty\}\), and the parabolic fixed
points are \(\PP^1(\QQ(i))\).  We use the standard matrices
\[
S=\begin{pmatrix}0&-1\\1&0\end{pmatrix},\qquad
R=\begin{pmatrix}0&i\\i&0\end{pmatrix},\qquad
T=\begin{pmatrix}1&1\\0&1\end{pmatrix},\qquad
U=\begin{pmatrix}1&i\\0&1\end{pmatrix}.
\]
A classical Ford-type fundamental region is
\[
 \mathcal F=
 \bigl\{(x+iy,t):x^2+y^2+t^2\ge1,
       -\tfrac12\le x\le\tfrac12,\ 0\le y\le\tfrac12\bigr\}.
\tag{1.1}
\]
Its vertical walls record the two cusp translations while its curved boundary
lies on the unit hemisphere.

\begin{figure}[H]
\centering
\includegraphics[width=.47\textwidth]{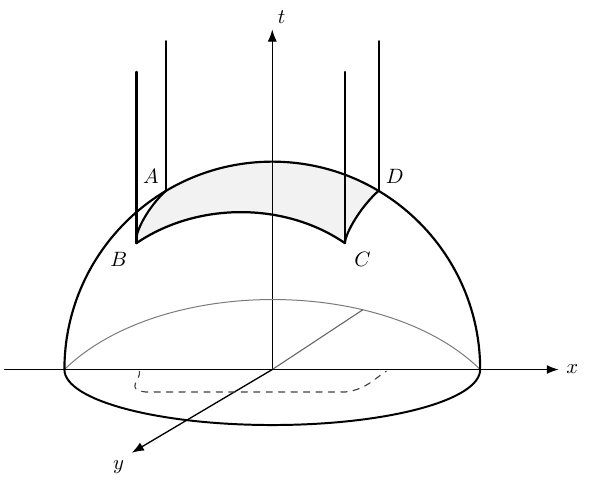}
\caption{The classical Picard fundamental region \(\mathcal F\) in upper
half-space.  The unit hemisphere and the four vertical walls give the
Poincar\'e-polyhedron geometry behind the standard presentation of \(G\).}
\label{fig:intro-classical-F}
\end{figure}

The associated presentation may be written
\[
\begin{split}
G=\langle S,T,U,R\mid &\ TU=UT,\quad S^2=R^2=(RS)^2=1,\\
 & (URS)^2=(TS)^3=(UR)^3=1,\\
 & (RTS)^2=1\rangle .
\end{split}
\tag{1.2}
\]
Thus \(\mathcal F\) presents \(G\) in the usual Poincar\'e manner: side
pairings provide generators and ridge cycles provide relations
\cite{Swan1971}.

For the Farey geometry it is more convenient to use the standard Picard cell
\(P_0\).  We henceforth write its finite vertices as
\[
 p=\left(\frac12,0,\frac{\sqrt3}{2}\right),\qquad
 q=\left(\frac12,\frac12,\frac1{\sqrt2}\right),\qquad
 p'=\left(0,\frac12,\frac{\sqrt3}{2}\right),
\tag{1.3}
\]
and its ideal vertices as \(0,\infty\).  The point
\[
        j=(0,0,1)
\tag{1.4}
\]
lies on the ideal edge \((0,\infty)\) and appears when the cell is subdivided
by the unit hemisphere.  In particular \(j\) is not an additional vertex of
the unrefined Cremona polyhedron.

The finite-face geometry makes the local stabilizers visible:
\[
  G_j\simeq C_2\times C_2,\qquad
  G_p\simeq S_3,\qquad
  G_q\simeq A_4,\qquad
  G_{p'}\simeq S_3,
\tag{1.5}
\]
while the edges \(jp,pq,qp',p'j\) have stabilizers
\[
          C_2,\ C_3,\ C_3,\ C_2.
\tag{1.6}
\]
The open finite quadrilateral has trivial stabilizer.  Most importantly,
\[
        K=\Stab_G(\OO)=G_q\simeq A_4,\qquad |K|=12,
\tag{1.7}
\]
a fact that later identifies octahedral occurrences with the \(q\)-type
vertices of the finite-face quotient.

\begin{figure}[H]
\centering
\includegraphics[width=.48\textwidth]{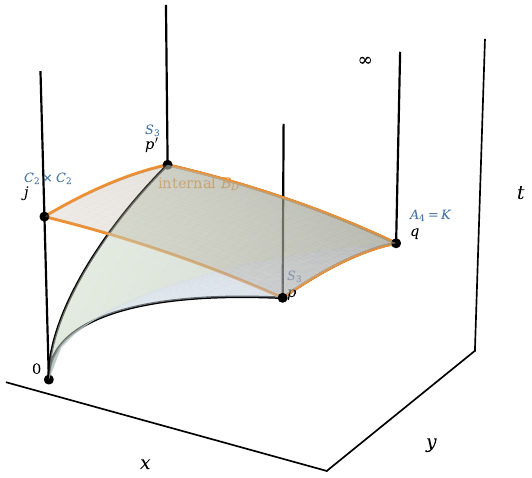}
\caption{The standard Picard cell \(P_0\), with the finite face
\(B_D=[j,p,q,p']\) highlighted.  The labels beside \(j,p,q,p'\) give the
corresponding vertex stabilizers; in particular \(G_q=K\simeq A_4\).}
\label{fig:intro-P0}
\end{figure}

Put \(X=TS\) and \(Y=UR\).  A presentation adapted to this cell is
\[
\begin{split}
G=\langle S,R,X,Y\mid\;&S^2=R^2=(RS)^2=X^3=(RX)^2=Y^3\\
 &=(X^{-1}Y)^2=(SY)^2=1\rangle .
\end{split}
\tag{1.8}
\]
The two presentations therefore emphasize different parts of the same group:
\(\mathcal F\) is natural for a direct Poincar\'e presentation, while
\(P_0\) makes the finite stabilizers and the later octahedral decomposition
transparent.

This geometry also gives a useful visual entry to Fine's classical amalgam
\[
G\cong
\bigl(S_3*_{C_3}A_4\bigr)
*_{\PSL_2(\ZZ)}
\bigl(S_3*_{C_2}(C_2\times C_2)\bigr),
\qquad \PSL_2(\ZZ)\cong C_2*C_3.
\tag{1.9}
\]
The picture does not by itself prove the amalgam.  Rather, the same finite
stabilizer groups that occur in Fine's Bass--Serre description are already
visible on the Picard cell, and the embedded real plane contains the modular
\(C_2*C_3\) geometry from Subsection~\ref{subsec:intro-kulkarni}
\cite{Fine1976,Fine1980}.

\subsection{The regions \(Q,D,P_0\), the finite face \(B_D\), and the classical retract}\label{subsec:intro-retract}
The comparison with Kulkarni becomes sharper after subdividing \(P_0\).  Put
\[
Q=\left\{0\le y\le x\le\frac12,\ x^2+y^2+t^2\ge1\right\},\qquad
D=\left\{0\le x,y\le\frac12,\ x^2+y^2+t^2\ge1\right\}.
\tag{1.10}
\]
If \(d\) is reflection in \(x=y\) and \(s\) reflection in the unit
hemisphere, then
\[
      D=Q\cup dQ,\qquad P_0=D\cup sD.
\tag{1.11}
\]
The exact twelve-cell subdivision of the Gaussian Farey octahedron gives
\[
      \OO=\bigcup_{k\in K}kP_0,
      \qquad
      \OO=12P_0=24D=48Q.
\tag{1.12}
\]
Figure~\ref{fig:intro-hierarchy} repeats the basic geometry at four scales so
that these multiplicities can be read geometrically rather than as a formal
chain of equalities.

\begin{figure}[H]
\centering
\includegraphics[width=.88\textwidth]{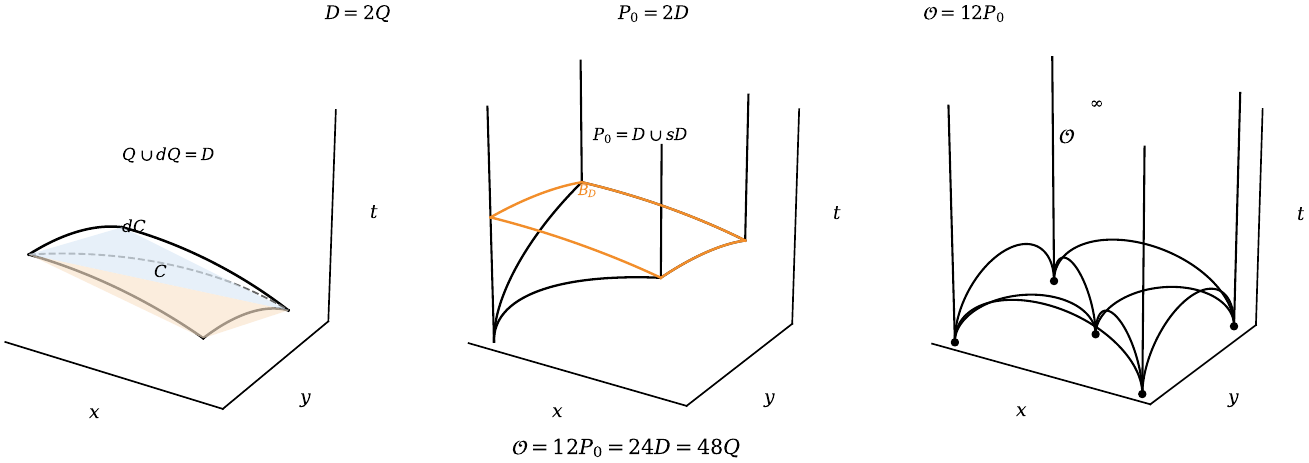}
\caption{The nested geometric scales \(Q\), \(D\), \(P_0\), and the Gaussian
Farey octahedron \(\OO\).  The successive multiplicities are
\(2,2,12\), hence \(\OO=12P_0=24D=48Q\).}
\label{fig:intro-hierarchy}
\end{figure}

The finite face on the unit hemisphere is
\[
      B_D=[j,p,q,p'],
\tag{1.13}
\]
and the diagonal \(jq\) divides it into the two triangles
\[
      C=[j,p,q],\qquad dC=[j,q,p'].
\tag{1.14}
\]
For \(\PGL_2(\ZZ[i])\), Mendoza's quotient uses the triangular cell \(C\);
for \(G=\PSL_2(\ZZ[i])\), the natural cell is the quadrilateral
\(B_D=C\cup dC\).  The translated complex
\[
      \mathscr X=\bigcup_{g\in G}gB_D
\tag{1.15}
\]
is classical: it is the Gaussian minimal-incidence retract of Mendoza, in
class number one it is the Fl\"oge retract, and it agrees with the translated
visible Ford boundary used in Swan's setting
\cite{Mendoza1980,Floege1983,RahmFuchs2011,Swan1971}.  We use this
identification as background; the paper does not claim a new Picard spine.

This is also the precise point at which Kulkarni's one-dimensional argument
ceases to transfer literally.  In the modular surface the translates of one
finite \(f\)-edge form a tree.  In the Picard group the corresponding finite
object is already the two-cell \(B_D\), and its translates carry genuine
attaching relations.  A maximal tree in the quotient one-skeleton is still
useful for paths and presentations, but it cannot replace those two-cells.
The attempt to mimic Kulkarni's Sections~2--3 therefore leads naturally to
the already classical Mendoza--Fl\"oge finite-face complex rather than to a
new graph-like spine.

Finally, the real plane
\[
        \Pi=\{y=0\}
\tag{1.16}
\]
is totally geodesic and cuts the Picard cells in the modular ones:
\[
        \Delta=D\cap\Pi,\qquad
        F=P_0\cap\Pi=\Delta\cup s\Delta.
\tag{1.17}
\]
Thus the two-dimensional geometry of Subsection~\ref{subsec:intro-kulkarni}
is literally a slice of the three-dimensional geometry just described.

\subsection{Gaussian Farey octahedra and Lee's geometry}\label{subsec:intro-octa}
A reduced Gaussian rational \(\alpha/\gamma\) is represented by a primitive
column in \(\ZZ[i]^2\), and two such vertices are Farey neighbours when
\[
  \frac{\alpha}{\gamma}\sim\frac{\beta}{\delta}
  \quad\Longleftrightarrow\quad
  N(\alpha\delta-\beta\gamma)=1.
\tag{1.18}
\]
Schmidt's Farey triangles and quadrangles and Hockman's Gaussian Farey
construction provide the arithmetic background
\cite{Schmidt1967,Hockman2019,Hockman2020a,Hockman2020b}.  In the positive
normalization used here,
\[
    \OO=(\infty,0,1,i,1+i,w),\qquad w=\frac{1+i}{2},
\tag{1.19}
\]
and \(K=\Stab_G(\OO)\simeq A_4\).

\begin{figure}[H]
\centering
\includegraphics[width=.58\textwidth]{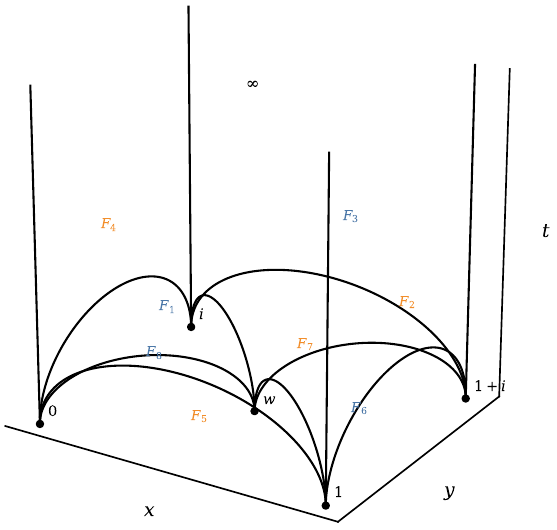}
\caption{The positive Gaussian Farey octahedron \(\OO\), with its six ideal vertices and the fixed face numbering \(F_1,\ldots,F_8\).  Different finite-index subgroups may impose different pairings on this same geometric octahedron.}
\label{fig:intro-octahedron}
\end{figure}

For a finite-index subgroup \(H\le G\), an octahedron may occur with local
isotropy.  If
\[
       x=HgK\in H\backslash G/K,
       \qquad
       J_x=g^{-1}Hg\cap K,
\tag{1.20}
\]
then the Picard cells over that occurrence form the fibre \(J_x\backslash K\)
and
\[
       [G:H]=\sum_{x\in H\backslash G/K}[K:J_x].
\tag{1.21}
\]
Thus an undecorated octahedron cannot determine the subgroup: one must also
know the local group and the ordered face identifications.  Lee's torsion-free
index-twelve examples make this especially concrete.  The same regular ideal
octahedron supports different side pairings, and hence different subgroups
\cite{Lee1984,BrunnerFrameLeeWielenberg1984}.  This observation motivates the
marked Picard Farey symbol used below: a finite collection of octahedral
occurrences, their local groups \(J_x\), and their ordered cooriented face
pairings.

\subsection{Result A: reconstruction and finite construction}\label{subsec:intro-A}
The first theorem says that these marked octahedral data recover the subgroup
itself, not merely a quotient polyhedron.

\begin{theorem}[Result A: reconstruction and finite construction]\label{thm:intro-A}
A marked Picard Farey symbol determines the pointed finite right \(G\)-set
and hence the corresponding finite-index subgroup.  Conversely every marked
finite-index subgroup gives such a symbol.  If membership in a finite-index
subgroup \(H\) can be decided for exact matrices, then \(H\backslash G/K\),
its local groups and all normalized face pairings can be found after finitely
many membership tests; no prior index or fundamental polyhedron is required.
\end{theorem}

For an outward-cooriented paired face \(F\to F'\), the normalized local map
\(q\) satisfies
\[
       qF=\overline{F'},\qquad
       qJ_{x,F}q^{-1}=J_{y,F'}.
\tag{1.22}
\]
The proof expands each occurrence into the exact twelve-cell fibre
\(J_x\backslash K\).  The face data then induce the permutations of the
Picard generators \(S,R,X,Y\); the verified transition identities and the
Picard relations produce a transitive pointed right \(G\)-set, whose
stabilizer is the required subgroup.  The same transition data can also be
read as the monodromy of the finite-face covering.  Conversely, the eight
fixed neighbouring-octahedron maps permit a breadth-first exploration of
\(H\backslash G/K\), and equality with a previously found occurrence reduces
to at most twelve exact membership tests in \(H\).

The ambient polyhedral and coset-development methods are classical
\cite{Swan1971,Lee1984,Page2015,ChanLangLimTan1993,LangLimTan1995}.  What is
specific here is that the marked Gaussian-octahedral quotient is sufficient
to reconstruct the Picard-cell action and can itself be obtained from an
effective membership test.  The finite-face retract remains classical, and
the covering interpretation is a second organization of the same verified
transition data.

The first genuinely nontrivial two-occurrence example is \(\Gamma_0(3)\), with local groups of orders two and three and fibres of sizes six and four.  It is developed after the construction theorem in Section~\ref{sec:picard-farey}, where the local isotropy and inter-occurrence face classes can be read against the exact Appendix data.

\subsection{Result B: torsion-free whole-octahedron domains}\label{subsec:intro-B}
Torsion-freeness removes the local octahedral isotropy and turns the counting
formula (1.21) into a uniform geometric statement.

\begin{theorem}[Result B: torsion-free Farey-octahedral domains]\label{thm:intro-B}
If \(H\le G\) is torsion-free of index \(n\), then
\[
       12\mid n,
       \qquad
       |H\backslash G/K|=\frac n{12}.
\]
There is a connected fundamental polyhedron made from exactly \(n/12\)
complete Gaussian Farey octahedra.  Every complete-octahedron fundamental
domain uses that same number of octahedra.  A spanning-tree development has
at most
\[
       \frac n4+1
\]
paired exposed triangular faces, whose pairing transformations generate
\(H\).
\end{theorem}

Indeed \(J_x=1\) for every occurrence, so each occurrence contains all twelve
Picard cells.  A spanning tree in the finite octahedral adjacency graph lifts
to a connected choice of one octahedron from each orbit; counting the faces
removed along its tree edges gives the boundary estimate.  The index-twelve
group \(\Gamma^{\mathrm{Pic}}_1(2+i)\) is the smallest clean example: one
whole octahedron is already a fundamental region, represented by the same
geometry as the left-hand part of Figure~\ref{fig:intro-octahedron}.

Lee's one-octahedron examples and other low-index polyhedra are classical
\cite{Lee1984,BrunnerFrameLeeWielenberg1984}.  The new point is the uniform
\(n/12\) consequence and the corresponding construction from the marked
quotient.  The statement is deliberately restricted to complete-octahedron
developments; it does not assert unrestricted geometric-face or presentation
minimality.

\subsection{Result C: Gaussian ideal edges and Bianchi modular symbols}\label{subsec:intro-C}
The same marked octahedra also carry the ideal edge--face data needed for the
Gaussian Steinberg module.

\begin{theorem}[Result C: finite Gaussian Steinberg presentation]\label{thm:intro-C}
Let \(F=\QQ(i)\) and
\[
  \Delta_0(F)=\ker\!\left(\ZZ[\PP^1(F)]\xrightarrow{\deg}\ZZ\right).
\]
For \(\mathrm{GL}_2/F\), the Steinberg module is canonically \(\Delta_0(F)\) and has the
integral Gaussian edge--triangle presentation
\[
 C_2^{\mathrm{id}}\longrightarrow C_1^{\mathrm{id}}
 \longrightarrow \Delta_0(F)\longrightarrow0.
\]
A marked Picard Farey symbol determines the finite \(H\)-orbit edge--face
presentation directly from its ideal one- and two-skeleton, without expanding
every occurrence into twelve Picard cells.  Hence, for a left \(H\)-module
\(V\),
\[
 \operatorname{Symb}_H(V)=
 \operatorname{Hom}_{\ZZ[H]}(\Delta_0(F),V)
\]
is the kernel of an explicit finite relation map.
\end{theorem}

The proof takes orbit representatives of the Gaussian ideal edges and
triangular faces already present in the marked octahedra, retaining the
stabilizers, orientation characters and incidence transports.  The resulting
finite boundary map is the quotient of the established integral Gaussian
Steinberg presentation.  These linear edge--face relations are different
from Poincar\'e ridge-cycle relations and from a group presentation of
\(H\); keeping those three relation systems separate is essential.

The foundational modular-symbol and Steinberg objects are classical, through
Cremona's Bianchi modular symbols, Gunnells' Vorono\"i viewpoint and the
integral generalized-Bykovskii presentation of Kupers--Miller--Patzt--Wilson
\cite{CremonaThesis,Cremona1984,Gunnells1999,KMPW}.  The contribution here is
the direct finite-orbit extraction from the marked octahedral data.  For the
complete \(\Gamma_1^{\mathrm{Pic}}(2+i)\) example, the three edge orbits,
four face orbits and the resulting \(3\times4\) group-ring matrix are written
out later; a separate spatial diagram would add less information than that
exact table.

\subsection{Result D: cusp lattices, level, congruence and Fine's quotient}\label{subsec:intro-D}
The marked symbol retains peripheral arithmetic that is invisible in a bare
finite incidence graph.

\begin{theorem}[Result D: cusp arithmetic and finite quotients]\label{thm:intro-D}
At each cusp \(C\), the marked symbol determines the full additive
translation lattice \(L_C\subset\ZZ[i]\).  The largest Gaussian ideal
contained in it is
\[
        \mathfrak c_C=L_C\cap i^{-1}L_C,
\]
and the standard ideal level is
\[
        \mathfrak l(H)=\bigcap_C\mathfrak c_C.
\]
Together with the reconstructed finite right \(G\)-set this gives a finite
congruence test.  The same finite action determines the quotient graph of
groups for Fine's amalgam, including the stabilizers of the modular edge
orbits.
\end{theorem}

One follows the cusp identifications around a horospherical link and reads off
the peripheral translations.  Their additive subgroup is \(L_C\), and the
intersection above extracts the largest Gaussian ideal contained in it.
Mason--Schweizer's level theorem then gives the standard level, while their
congruence criterion reduces principal-congruence containment to a finite
factorization test in the reconstructed permutation action.  Fine's
Bass--Serre quotient is obtained from the same finite \(G\)-set by taking
orbits under the two vertex groups and the embedded modular group
\cite{MasonSchweizer2012,Fine1976,Fine1980}.

For \(\Gamma_1^{\mathrm{Pic}}(2+i)\), the two normalized translation lattices
are
\[
       L_\infty=\ZZ[i],\qquad
       L_0=(2+i)\ZZ[i]=\langle 2+i,-1+2i\rangle_{\ZZ}.
\tag{1.23}
\]
Their covolumes in these coordinates are \(1\) and \(5\), illustrating the
arithmetic information carried by the cusp data.

\subsection{Result E: modular intersection and an adapted geometric lift}\label{subsec:intro-E}
The modular slice (1.17) also allows a classical special polygon to be lifted
back into a Picard fundamental polyhedron.

\begin{theorem}[Result E: modular intersection and \(M\)-adapted lifting]\label{thm:intro-E}
Let \(M=\PSL_2(\ZZ)\subset G\).  If \(\mathfrak a\subset\ZZ[i]\) is nonzero
and \(\mathfrak a\cap\ZZ=n(\mathfrak a)\ZZ\), then
\[
 \Gamma_\bullet^{\mathrm{Pic}}(\mathfrak a)\cap M
   =\Gamma_\bullet^{\mathrm{cl}}(n(\mathfrak a)),
 \qquad \bullet=0,1,\varnothing .
\]
At a split prime of residue degree one, the corresponding Picard and modular
coset actions are the same finite \(\PSL_2(\mathbf F_p)\)-action.  More
generally, if
\[
       (H\cap M)\backslash M\longrightarrow H\backslash G
\]
is bijective and a side-connected Kulkarni polygon has the form
\[
       Q_M=\bigcup_{m\in\mathcal R}mF,
\]
then the same representatives give a connected Picard domain
\[
       P=\bigcup_{m\in\mathcal R}mP_0,
       \qquad P\cap\Pi=Q_M.
\]
For \((2+i)\mid5\), the level-five polygons for
\(\Gamma_0(5),\Gamma_1(5),\Gamma(5)\) lift to domains with
\(6,12,60\) Picard cells, respectively.
\end{theorem}

The arithmetic intersection statement follows directly by imposing the
Picard congruence conditions on integral matrices.  At a degree-one split
prime, reduction identifies the two coset actions with the same projective
linear action over \(\mathbf F_p\).  For the geometric assertion, the
coset-bijection says that a modular transversal is simultaneously a Picard
transversal, and side-connectivity of the \(F\)-tiles lifts to connectivity
of the corresponding \(P_0\)-cells.  This is precisely where the identity
\(F=P_0\cap\Pi\) from (1.17) becomes useful.

Kulkarni provides the special polygon; the three-dimensional contribution is
the adapted lift with a prescribed totally geodesic section.  A more
compressed Picard domain for the same subgroup need not preserve that
section.  Figure~\ref{fig:intro-resultE} shows the simplest exact example:
the six modular images of \(F\) form the \(\Gamma_0(5)\) polygon and the same
six representatives lift to \(\Gamma_0^{\mathrm{Pic}}(2+i)\).

\begin{figure}[H]
\centering
\includegraphics[width=.76\textwidth]{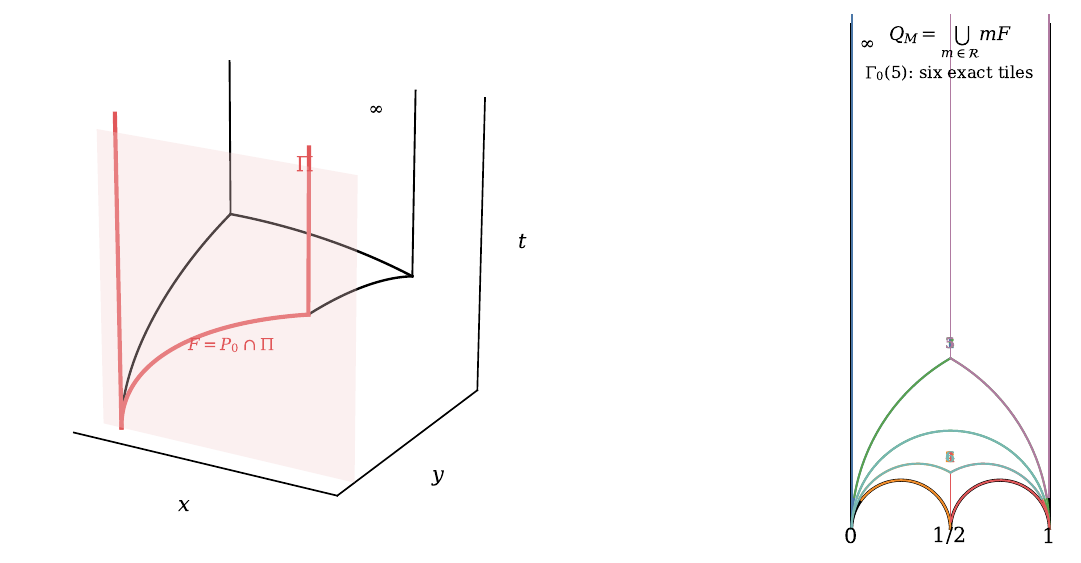}
\caption{The modular section \(F=P_0\cap\Pi\) and the six-tile \(\Gamma_0(5)\) polygon; the same six representatives give the corresponding \(M\)-adapted Picard cell union.}
\label{fig:intro-resultE}
\end{figure}

\subsection{Organization of the paper}\label{subsec:intro-org}
The organization has no separate preliminary section and places
background where it is first used.  Section~2 records the exact geometry of
\(P_0,D,Q,B_D\), the ambient presentations, the \(\PGL/\PSL\) distinction and
the classical Mendoza--Fl\"oge--Swan finite-face complex.  Section~3 develops
Gaussian Farey vertices, the standard octahedron, its neighbour maps and the
exact twelve-cell subdivision.  Section~4 collects the polyhedral subgroup
tools---Poincar\'e's theorem, Lee's connected coset construction, the
no-inversions/complex-of-groups conventions and the comparison with
Kulkarni---needed for the subgroup constructions that follow.

Result~A is proved in Section~5: first the marked symbol is defined, then the
reconstruction theorem and finite-face covering are established, and finally
the effective construction from subgroup membership is given.  Section~6
develops the level-\((2+i)\) family as a running set of concrete examples for
the local fibres and face pairings.  Section~7 proves Result~D and the
arithmetic part of Result~E, including cusp lattices, ideal level,
congruence recognition, modular intersection, the split-prime comparison,
the \(M\)-adapted lift and Fine's quotient.  The geometric base identities
\(\Delta=D\cap\Pi\) and \(F=P_0\cap\Pi\) have already been established in
Section~2.

Result~B is proved in Section~8 by passing from the finite octahedral quotient
to whole-octahedron developments in the torsion-free case.  Result~C is proved
in Section~9, where the Gaussian ideal edge--face complex is recalled, the
finite \(H\)-orbit presentation is extracted from a marked symbol, and the
complete \(\Gamma_1^{\mathrm{Pic}}(2+i)\) matrix is worked out.  Section~10
collects the remaining questions on geometric reduction, complexity,
arithmetic cusp reduction, modular sections, other Euclidean Bianchi fields
and Hecke operators.  The appendix contains the exact low-index domains and
pairing tables, including \(\Gamma_0(1+i)\), the two index-six square pyramids,
the one-octahedron index-twelve examples, the level-\((2+i)\) family and the
\(6+4\) occurrence data for \(\Gamma_0(3)\).

This order keeps three layers distinct: the classical Picard finite-face
retract, the subgroup-dependent marked octahedral data, and the Gaussian
Steinberg quotient.  The first is established geometry, the second is the
organizing object of the paper, and the third is one of its principal
applications.

\section{Picard geometry and the classical finite-face complex}\label{sec:picard-geometry}
This section fixes the exact normalization behind the geometric pictures in the Introduction.  It records the ambient extensions, the reflection regions and Picard cell, the stabilizer-labelled finite face and its classical retract, and the modular plane used later for Result~E.

\subsection{Ambient groups and conventions}\label{subsec:bianchi}\label{subsec:extended-picard}
Let \(K=\QQ(\sqrt{-d})\) be an imaginary quadratic field with ring of integers \(\mathcal O_K\).  The Bianchi group
\[
   G_K=\PSL_2(\mathcal O_K)
\]
acts on the upper-half-space model
\[
   \HH^3=\{(z,t):z\in\mathbf C,\ t>0\}
\]
by the usual extension of M\"obius transformations.  The action is properly discontinuous and has finite covolume.  Ideal vertices lie in \(\partial\HH^3=\mathbf C\cup\{\infty\}\), and the parabolic fixed points are \(\PP^1(K)\).  We work mainly with
\[
   G=\PSL_2(\ZZ[i]).
\]

We use
\[
S=\begin{pmatrix}0&-1\\1&0\end{pmatrix},\quad
R=\begin{pmatrix}0&i\\i&0\end{pmatrix},\quad
T=\begin{pmatrix}1&1\\0&1\end{pmatrix},\quad
U=\begin{pmatrix}1&i\\0&1\end{pmatrix},
\]
\[
X=TS=\begin{pmatrix}1&-1\\1&0\end{pmatrix},\qquad
Y=UR=\begin{pmatrix}-1&i\\i&0\end{pmatrix}.
\]
Let \(\mathcal U=\ZZ[i]^\times=\{\pm1,\pm i\}\).  We write
\[
G=\PSL_2(\ZZ[i]),\qquad
\widetilde G=\PGL_2(\ZZ[i])=\mathrm{GL}_2(\ZZ[i])/\mathcal U,
\]
where scalar units act centrally.  The determinant induces a well-defined square-class character
\begin{equation*}
\delta:\widetilde G\longrightarrow \mathcal U/\mathcal U^2\simeq C_2,
\qquad \mathcal U^2=\{1,-1\}.
\end{equation*}
Let \(c\) denote complex conjugation, acting on the boundary by \(c(z)=\bar z\), and put
\[
G^{\mathrm{ref}}=G\rtimes\langle c\rangle,
\qquad
\widehat G=\widetilde G\rtimes\langle c\rangle.
\]
An element of \(\widehat G\) is holomorphic or antiholomorphic; denote this character by \(\epsilon:\widehat G\to C_2\).  Since conjugation preserves the square class in \(\mathcal U/\mathcal U^2\), the character \(\delta\) extends to \(\widehat G\) by taking the determinant class of its linear part.

\begin{proposition}[The extended Picard square]\label{prop:extended-square}
With the notation above,
\[
G=\ker\delta\cap\ker\epsilon,\qquad
\widetilde G=\ker\epsilon,\qquad
G^{\mathrm{ref}}=\ker\delta,
\]
and
\[
\widehat G/G\simeq C_2\times C_2.
\]
Thus \(\widetilde G\) and \(G^{\mathrm{ref}}\) are distinct index-two extensions of \(G\), with intersection \(G\).  Hockman's extended Picard group is the orientation-preserving group \(\widetilde G\), not the reflection extension \(G^{\mathrm{ref}}\).
\end{proposition}
\begin{proof}
If a projective Gaussian matrix is multiplied by a scalar unit \(u\), its determinant is multiplied by \(u^2\), so \(\delta\) is well defined.  A class lies in \(\ker\delta\) exactly when a representative has determinant in \(U^2=\{\pm1\}\).  If the determinant is \(1\) it is already represented by an element of \(\mathrm{SL}_2(\ZZ[i])\); if the determinant is \(-1\), multiplication by the scalar \(i\) changes the determinant to \(1\).  Hence \(\ker(\delta|_{\widetilde G})=G\).  The character \(\epsilon\) records whether the action is holomorphic or antiholomorphic, so \(\ker\epsilon=\widetilde G\).  Since complex conjugation sends \(i\) to \(-i\), which represents the same nontrivial class in \(U/U^2\), \(\delta\) is unchanged under the semidirect action and extends to \(\widehat G\).  The displayed kernel identities follow.  The two independent characters \((\delta,\epsilon)\) therefore identify \(\widehat G/G\) with \(C_2\times C_2\).  Hockman's definition uses Gaussian M\"obius transformations of unit determinant class and is precisely \(\widetilde G\) \cite[\S3, p.~151]{Hockman2019}.
\end{proof}

\subsection{Reflection regions and the Picard cell}\label{subsec:coxeter-chambers}\label{subsec:prelim-picard}
Use upper-half-space coordinates \((x,y,t)\), \(t>0\), and put
\[
j=(0,0,1),\quad
p=\left(\tfrac12,0,\tfrac{\sqrt3}{2}\right),\quad
q=\left(\tfrac12,\tfrac12,\tfrac1{\sqrt2}\right),\quad
p'=\left(0,\tfrac12,\tfrac{\sqrt3}{2}\right).
\]
Consider the antiholomorphic reflections
\[
\begin{array}{lll}
s(z)=1/\bar z,&a(z)=-\bar z,&b(z)=\bar z,\\
u(z)=1-\bar z,&v_0(z)=\bar z+i,&d(z)=i\bar z.
\end{array}
\]
Their mirrors are respectively the unit hemisphere, \(x=0\), \(y=0\), \(x=1/2\), \(y=1/2\), and \(x=y\).

\begin{proposition}[The Coxeter tetrahedron and the Picard cell]\label{prop:reflection-chambers}
The tetrahedron
\[
Q=\left\{0\le y\le x\le\tfrac12,
\ x^2+y^2+t^2\ge1\right\}
\]
with vertices \(j,p,q,\infty\) is a fundamental Coxeter tetrahedron for \(\widehat G\).  In the mirror order \((s,u,d,b)\) its Coxeter type is \([3,4,4]\).  The region
\[
D=\left\{0\le x\le\tfrac12,\ 0\le y\le\tfrac12,
\ x^2+y^2+t^2\ge1\right\}
\]
is a five-faced square pyramid with ideal apex \(\infty\), and
\[
D=Q\cup dQ,\qquad P_0=D\cup sD.
\]
Consequently, after Proposition~\ref{prop:twelve-cells} is applied,
\[
\mathcal O=12P_0=24D=48Q
\]
as a union of interiors-disjoint copies of the indicated regions.
\end{proposition}
\begin{proof}
The successive mirrors of \(Q\) are the unit hemisphere, \(x=1/2\), \(x=y\), and \(y=0\).  Conformality of the upper-half-space model gives successive dihedral angles \(\pi/3,\pi/4,\pi/4\); the nonconsecutive mirrors meet orthogonally.  Thus the Coxeter matrix is the one of \([3,4,4]\).  By the reflection form of the Poincar\'e polyhedron theorem, these mirrors develop \(Q\) to the corresponding Coxeter tessellation \cite{EpsteinPetronio1994}.  Moreover \(db:z\mapsto iz\), and the remaining face reflections generate translations and inversion through products (for example \(ua:z\mapsto z+1\)); hence the reflection group contains \(\widetilde G\) and complex conjugation, and therefore is \(\widehat G\).  Conversely all four reflections lie in \(\widehat G\).

The diagonal reflection \(d\) interchanges the two halves of the square \(0\le x,y\le1/2\), so \(D=Q\cup dQ\).  Reflection \(s\) in the unit hemisphere sends the outer half of the standard Picard cell to its inner half; the defining inequalities give exactly \(P_0=D\cup sD\).  These unions have disjoint interiors.  The last equality follows from the twelve-cell decomposition proved in Proposition~\ref{prop:twelve-cells}.
\end{proof}

The same normalization gives the explicit Picard cell
\[
P_0=\left\{0\le x,y\le\tfrac12:\ x^2+y^2+t^2\ge2x,\quad x^2+y^2+t^2\ge2y\right\}.
\]
Its ideal vertices are \(0,\infty\), its finite vertices are \(p,q,p'\), and the point \(j\) lies on the ideal edge \((0,\infty)\).

The finite-face dictionary is
\[
\begin{array}{c|c|c}
\text{region}&\text{finite face on }x^2+y^2+t^2=1&\text{vertices}\\ \hline
Q&C&j,p,q\\
dQ&dC&j,q,p'\\
D&B_D=C\cup dC&j,p,q,p'
\end{array}
\]
The diagonal $jq$ lies in the mirror $x=y$.  In the triangular cell structure it separates $C$ from $dC$; it is absent from the quadrilateral cell structure for $G$.  The underlying translated finite-face set is unchanged.  Mendoza's $\PGL_2(\ZZ[i])$ quotient has the single triangular cell $C$, while our $\PSL_2(\ZZ[i])$ quotient has the quadrilateral $B_D$; their local stabilizers must not be identified.  Figure~\ref{fig:picard-square-v8e} gives the $G$-stabilizers.  This distinction concerns the cell structure on the same classical retract.

\begin{remark}[Two different index-two extensions]\label{rem:D-two-groups}
The five-faced region \(D\) is a fundamental region for the reflection group \(G^{\mathrm{ref}}\), whereas it is a side-paired fundamental polyhedron for the orientation-preserving group \(\widetilde G\).  Thus the same geometric region occurs for two different index-two extensions of \(G\), with different boundary identifications.  The distinction is therefore essential whenever the phrase ``extended Picard group'' is used.
\end{remark}

\begin{figure}[htbp]
\centering
\begin{tikzpicture}[x={(8cm,-2.5cm)},y={(5cm,4cm)},z={(0cm,10cm)},line join=round]
 \fill[orange!20] (0,0,1)
 -- plot[domain=0:.5,samples=45,variable=\u] ({\u},0,{sqrt(1-\u*\u)})
 -- plot[domain=0:.5,samples=45,variable=\u] (.5,{\u},{sqrt(.75-\u*\u)})
 -- plot[domain=.5:0,samples=45,variable=\u] ({\u},{\u},{sqrt(1-2*\u*\u)})--cycle;
 \fill[blue!12] (0,0,1)
 -- plot[domain=0:.5,samples=45,variable=\u] ({\u},{\u},{sqrt(1-2*\u*\u)})
 -- plot[domain=.5:0,samples=45,variable=\u] ({\u},.5,{sqrt(.75-\u*\u)})
 -- plot[domain=.5:0,samples=45,variable=\u] (0,{\u},{sqrt(1-\u*\u)})--cycle;
 \draw[thick] plot[domain=0:.5,samples=45,variable=\u] ({\u},0,{sqrt(1-\u*\u)})
 plot[domain=0:.5,samples=45,variable=\u] (.5,{\u},{sqrt(.75-\u*\u)})
 plot[domain=.5:0,samples=45,variable=\u] ({\u},.5,{sqrt(.75-\u*\u)})
 plot[domain=.5:0,samples=45,variable=\u] (0,{\u},{sqrt(1-\u*\u)});
 \draw[dashed] plot[domain=0:.5,samples=45,variable=\u] ({\u},{\u},{sqrt(1-2*\u*\u)});
 \fill (0,0,1) circle (1.5pt) node[left=4pt] {$j$};
 \fill (.5,0,{sqrt(.75)}) circle (1.5pt) node[below=4pt] {$p$};
 \fill (.5,.5,{sqrt(.5)}) circle (1.5pt) node[right=4pt] {$q$};
 \fill (0,.5,{sqrt(.75)}) circle (1.5pt) node[above=4pt] {$p'$};
 \node at (.33,.13,{sqrt(1-.33*.33-.13*.13)}) {$C$};
 \node at (.13,.33,{sqrt(1-.33*.33-.13*.13)}) {$dC$};
\end{tikzpicture}
\caption{The finite spherical face $B_D$ of $D$, split by the geodesic diagonal $jq$ into $C=[j,p,q]$ and $dC=[j,q,p']$.  The four sides are geodesic arcs on the unit hemisphere.  The region $D$ extends vertically from this face towards $\infty$; its two parts above $C$ and $dC$ are $Q$ and $dQ$.  Reflection in the hemisphere gives $P_0=D\cup sD$.}
\label{fig:reflection-chambers}
\end{figure}

\subsection{The finite face and the classical complex}\label{subsec:finite-face-geometry}
The finite spherical face of the reflection region $D$ is the geodesic quadrilateral
\[
B_D=[j,p,q,p'],\qquad C=[j,p,q],\qquad
\mathscr X=\bigcup_{g\in G}gB_D.
\]
Here and below $\mathscr X$ denotes the finite-face complex, while $X$ remains the matrix generator above.  The region $D$ and its reflection double $P_0$ are described in Subsection~\ref{subsec:coxeter-chambers}.  The quadrilateral $B_D$ is a fundamental two-cell for the $G$-action on $\mathscr X$; its two-cell stabilizer is trivial.  The boundary has edge stabilizers
\[
\langle R\rangle,\quad\langle X\rangle,\quad
\langle Y\rangle,\quad\langle S\rangle
\]
on $jp,pq,qp',p'j$, respectively.  Their orders are $2,3,3,2$, and the vertex groups at $j,p,q,p'$ are $C_2\times C_2,S_3,A_4,S_3$; see Figure~\ref{fig:picard-square-v8e}.

The same underlying complex has Mendoza's triangular fundamental cell $C$ for $\PGL_2(\ZZ[i])$.  For $\PSL_2(\ZZ[i])$ the natural cell used here is $B_D=C\cup dC$, where $d(z)=i\bar z$.  The diagonal $jq$ subdivides this quadrilateral but is not an edge in its quadrilateral cell structure.  The geometric identification and its classical provenance are recorded below.  The corresponding ambient presentation is
\[
G=\langle S,R,X,Y\mid
S^2=R^2=(RS)^2=X^3=(RX)^2=Y^3=(X^{-1}Y)^2=(SY)^2=1\rangle.
\tag{2.1}\label{eq:picard-presentation}
\]
\begin{figure}[htbp]
\centering
\begin{tikzpicture}[x=3.0cm,y=2.3cm]
 \fill[blue!7] (0,0) rectangle (1,1);
 \draw[thick] (0,1)--node[above=5pt] {$C_2$}(1,1)
 --node[right=5pt] {$C_3$}(1,0)
 --node[below=5pt] {$C_3$}(0,0)
 --node[left=5pt] {$C_2$}(0,1);
 \foreach \P in {(0,1),(1,1),(1,0),(0,0)}\fill \P circle (1.6pt);
 \node[anchor=south east,align=center] at (-.08,1.04) {$j$\\$C_2\times C_2$};
 \node[anchor=south west,align=center] at (1.08,1.04) {$p$\\$S_3$};
 \node[anchor=north west,align=center] at (1.08,-.04) {$q$\\$A_4$};
 \node[anchor=north east,align=center] at (-.08,-.04) {$p'$\\$S_3$};
 \node at (.5,.5) {$B_D$};
\end{tikzpicture}
\caption{The stabilizer-labelled quadrilateral for $G\backslash\mathscr X$.  The four vertex groups and four edge groups are shown in cyclic order; the open two-cell has trivial stabilizer.  This is an incidence drawing of the spherical face $B_D$, not a Euclidean model of the Picard polyhedron.}
\label{fig:picard-square-v8e}
\end{figure}

This finite-face complex is classical.  As a subset of $\HH^3$ it is the Gaussian Mendoza minimal-incidence retract: Mendoza defines that set in \cite[\S2.1.6, p.~19]{Mendoza1980}, proves equivariant retraction in \cite[\S2.3]{Mendoza1980}, and identifies it with the translated visible Ford boundary in class number one in \cite[\S3.1 and Proposition~3.1.1, pp.~43--44]{Mendoza1980}.  The Gaussian triangular cell and its $\PGL_2$ quotient are given in \cite[\S\S4.1.3--4.1.8, pp.~49--54]{Mendoza1980}.  The $\PSL_2$ quadrilateral $B_D$ joins two of these triangles.  Since $\ZZ[i]$ has class number one, the underlying set is also the Fl\"oge retract; here we use Fl\"oge's Theorem~6.6 in the formulation reproduced by Rahm--Fuchs \cite[Appendix, Theorem~28]{RahmFuchs2011}.  In Swan's Ford-polyhedron description it is the union of translates of the visible lower boundary, not the three-dimensional Ford region itself.\footnote{R.~G.~Swan, \emph{Generators and relations for certain special linear groups}, Bull. Amer. Math. Soc. \textbf{74} (1968), 576--581, \S\S2--4: the edge relations appear on p.~578 and the Gaussian presentation in Theorem~4, p.~579.  See also the full account \cite{Swan1971}.}  Presentation extraction from faces and edge cycles, and subsequent presentation simplification, are classical methods.

\subsection{The modular plane}\label{subsec:modular-slice}
Write $\Pi=\{y=0\}$ and identify it with the upper half-plane with coordinates $(x,t)$.  The reflection $s$ preserves $\Pi$, and the preceding regions give
\[
\Delta=D\cap\Pi,\qquad F=P_0\cap\Pi=\Delta\cup s\Delta.
\]
Thus the small triangle used by Kulkarni is a literal section of $D$, and its orientation-preserving double is a section of $P_0$.

\begin{proposition}[The classical extended Farey tessellation as a slice]\label{prop:modular-slice}
The totally geodesic plane \(y=0\) is preserved by the three mirrors \(x=0\), \(x=1/2\), and \(x^2+t^2=1\).  Their intersection with that plane is the classical \((2,3,\infty)\) reflection tessellation of \(\HH^2\).  The chamber with vertices
\[
j=(0,1),\qquad p=\left(\tfrac12,\tfrac{\sqrt3}{2}\right),\qquad \infty
\]
has angles \((\pi/2,\pi/3,0)\).  Exactly six such reflection chambers make the ideal Farey triangle with vertices \((0,1,\infty)\).
\end{proposition}
\begin{proof}
Inside \(y=0\), the three mirrors are the vertical lines \(x=0\), \(x=1/2\), and the unit semicircle.  Their pairwise angles are \(\pi/2\) at \(j\), \(\pi/3\) at \(p\), and \(0\) at the cusp.  Hence their reflection group is the usual \((2,3,\infty)\) triangle group, whose orientation-preserving subgroup is \(\PSL_2(\ZZ)\).  The area of the elementary triangle is
\[
\pi-\left(\frac\pi2+\frac\pi3\right)=\frac\pi6,
\]
whereas an ideal hyperbolic triangle has area \(\pi\).  Since the reflection tessellation restricts face-to-face to this plane, six elementary chambers fill \((0,1,\infty)\).
\end{proof}

\begin{remark}[Even and odd ridges]\label{rem:even-odd-ridges}
Along the finite ridges $jp,jp'$ of $D$ the unit hemisphere meets a vertical face at angle $\pi/2$; along $pq,p'q$ the angle is $\pi/3$.  Respectively four and six $D$-regions meet around an interior point of such a ridge, and its pointwise stabilizer in $G$ is $C_2$ or $C_3$.  These are ridge stabilizers, not stabilizers fixing a whole face pointwise.  The even and odd vertices of Kulkarni's two-dimensional chamber therefore correspond locally to relation-bearing ridges in dimension three.  His finite $f$-edge, by contrast, corresponds to the finite face $B_D$.  The local cyclic groups describe the rotations around ridges; they do not dispense with complete ridge cycles, finite vertex links or Euclidean cusp links, and give no three-dimensional admissibility theorem by themselves.
\end{remark}

\begin{remark}[From the geometric slice to the arithmetic shadow]\label{rem:modular-slice-arithmetic}
Proposition~\ref{prop:modular-slice} concerns the ambient reflection tessellation and does not by itself identify subgroup fundamental domains in the plane $y=0$.  Section~\ref{sec:invariants} proves a separate arithmetic statement: for every Gaussian ideal $\mathfrak a$, the standard Picard congruence subgroups meet the embedded modular group in the corresponding classical groups of level determined by $\mathfrak a\cap\ZZ$.  At a split prime of residue degree one, the associated finite modular and Picard coset actions also coincide.  An intrinsic criterion ensuring that an arbitrary chosen or compressed Picard polyhedron cuts the plane in a Kulkarni special polygon remains open; the $M$-adapted existence theorem in Subsection~\ref{subsec:M-adapted-shadow} addresses a deliberately chosen development.  Keeping these three levels distinct will be important below.
\end{remark}

Put
\[
M=\PSL_2(\ZZ),\qquad \Pi=\{y=0\}\subset\HH^3.
\]
The preceding proposition identifies the reflection tessellation induced on
$\Pi$.  For subgroup applications it is useful to identify the
orientation-preserving base cell inside this slice.

\begin{lemma}[The base modular section]\label{lem:base-modular-section}
Let $P_0$ be the standard Picard cell and put $F=P_0\cap\Pi$.  If
\[
\Delta=\left\{(x,t)\in\HH^2:
0\le x\le\tfrac12,\ x^2+t^2\ge1\right\},
\]
then
\[
F=\Delta\cup s\Delta,
\]
where $s$ is reflection in the unit semicircle.  Thus $F$ is a
fundamental polygon for $M$, obtained from two adjacent
$(2,3,\infty)$ reflection triangles, and
$\operatorname{area}(F)=\pi/3$.
\end{lemma}
\begin{proof}
The normalization of Proposition~\ref{prop:reflection-chambers} gives
$P_0=D\cup sD$, and $s$ preserves $\Pi$.  Hence
$P_0\cap\Pi=(D\cap\Pi)\cup s(D\cap\Pi)=\Delta\cup s\Delta$.
The vertices of $\Delta$ are
\[
j=(0,1),\qquad p=(1/2,\sqrt3/2),\qquad \infty,
\]
and its angles are $\pi/2,\pi/3,0$.  The reflections in its sides form
the classical $(2,3,\infty)$ reflection group and their
orientation-preserving subgroup is $M$.  The area is twice
$\pi-\pi/2-\pi/3=\pi/6$.
\end{proof}

\section{Gaussian Farey octahedra and the twelve-cell subdivision}\label{sec:gaussian-octahedra}
The Gaussian-rational boundary data are now fixed independently of subgroup markings.  We first normalize the standard octahedron and its faces, then record the universal neighbour maps and finally prove the exact twelve-cell subdivision used in reconstruction.

\subsection{Gaussian neighbours and the standard octahedron}\label{subsec:gaussian-rational}\label{subsec:extended-octahedron}
Write
\[
   \QQ(i)_\infty=\PP^1(\QQ(i)).
\]
A point is represented by a reduced pair \((\alpha,\gamma)\in\ZZ[i]^2\), unique up to multiplication by a unit.  Following Schmidt and Hockman, two reduced Gaussian rationals are Farey neighbours when
\[
   \frac{\alpha}{\gamma}\sim\frac{\beta}{\delta}
   \quad\Longleftrightarrow\quad
   N(\alpha\delta-\beta\gamma)=1.
\]
Equivalently the two column vectors form a unimodular pair over \(\ZZ[i]\).  Hockman's Gaussian Farey graph has vertex set \(\QQ(i)_\infty\) and these neighbour pairs as edges; its hyperbolic realization is the one-skeleton of the Farey-octahedral tessellation.  Schmidt's earlier Farey triangles and quadrangles already make the same arithmetic point in terms of unimodular matrices over imaginary quadratic integer rings \cite{Schmidt1967,Hockman2020a}.

Set
\[
   \OO=\{\infty,0,1,i,1+i,w\},\qquad w=\frac{1+i}{2},
\]
and let \(K=\Stab_G(\OO)\).  We use the positive normalization compatible with the cell \(P_0\) from Subsection~\ref{subsec:prelim-picard}.  For comparison with the older Lee coordinates, put
\[
 L=\begin{pmatrix}1&-1\\0&1\end{pmatrix}=T^{-1}.
\]
Then \(L\) acts by \(z\mapsto z-1\), sends the positive octahedron to
\[
 L\OO=\{\infty,0,i,-1,-1+i,(-1+i)/2\},
\]
and transports the other normalization by
\[
 K_{-}=LKL^{-1},\qquad P_{0,-}=LP_0.
\]
Thus the octahedron, its stabilizer and the distinguished Picard cell are always translated together; no calculation below mixes the two normalizations.

For later explicit work it is convenient to use two elements
\[
 a=\begin{pmatrix}i&1\\ i&1-i\end{pmatrix},
 \qquad
 c=\begin{pmatrix}-1&1\\-1&0\end{pmatrix},
\]
which generate \(K\).  One exact ordering of its twelve elements is
\[
1,a,a^2,c^2,c,c^2a^2,ac,ca,a^2c^2,ac^2,c^2a,aca.
\tag{3.1}\label{eq:K-order}
\]
The exact word ordering is not intrinsic, but fixing it once is useful for comparing the examples.

For the present paper it is enough to define an oriented Gaussian Farey octahedron to be a translate
\[
   \OO_g=g\OO,
   \qquad g\in G,
\]
recorded by the ordered six-tuple
\[
   \bigl(g\infty,g0,g1,gi,g(1+i),gw\bigr)\in\PP^1(\QQ(i))^6.
\tag{3.2}\label{eq:rational-octahedron}
\]
Every triangular face is therefore an ordered triple of Gaussian rationals.  This is the arithmetic information hidden by a notation such as ``an octahedral occurrence.''  We shall keep the rational coordinates visible in the definition of a Picard Farey symbol and use normalized occurrences only as a convenient compact form in proofs and calculations.

\subsection{Ordered faces and determinant types}
\begin{definition}[Determinant type of a Gaussian Farey triangle]\label{def:triangle-type}
Let \(T\) be a Gaussian Farey triangle.  By Hockman's transitivity theorem there is \(g\in\widetilde G\) carrying \((\infty,0,1)\) to the three vertices of \(T\).  Define
\[
\chi(T)=\delta(g)\in U/U^2\simeq C_2.
\]
The value is independent of the ordering chosen for the vertices of \(T\).
\end{definition}

\begin{proposition}[Two Picard face types]\label{prop:two-face-types}
Gaussian Farey triangles form exactly two \(G\)-orbits, distinguished by \(\chi\), and one \(\widetilde G\)-orbit.  For the standard positive octahedron with faces
\[
\begin{aligned}
F_1&=(0,1,\infty),&F_2&=(1,1+i,\infty),&F_3&=(1+i,i,\infty),&F_4&=(i,0,\infty),\\
F_5&=(0,w,1),&F_6&=(1,w,1+i),&F_7&=(1+i,w,i),&F_8&=(i,w,0),
\end{aligned}
\]
the two \(G\)-orbits are
\[
\{F_1,F_3,F_6,F_8\},\qquad
\{F_2,F_4,F_5,F_7\}.
\]
The symmetry
\[
\rho(z)=iz+1\in\widetilde G\setminus G
\]
stabilizes \(\mathcal O\) and exchanges the two classes.
\end{proposition}
\begin{proof}
Every permutation of the standard triple \((\infty,0,1)\) is realized projectively by a matrix of determinant \(\pm1\), and therefore by an element of \(G\).  Hence \(\chi(T)\) is independent of the ordering.  Hockman proves that every Gaussian Farey triangle is an image of \((\infty,0,1)\) under her extended Picard group \cite[Theorem~3.3, p.~153]{Hockman2019}.  The map \(z\mapsto iz\) carries \((\infty,0,1)\) to \((\infty,0,i)\) and has the nontrivial determinant square class, so the two determinant types both occur and cannot lie in the same \(G\)-orbit.  This gives exactly two \(G\)-orbits.  Directly, \(\rho\) fixes \(\infty\) and \(w=(1+i)/2\), cycles \(0,1,1+i,i\), and sends \(F_1\) to \(F_2\); the displayed two four-element orbits follow by the action of \(K=\Stab_G(\mathcal O)\).
\end{proof}

\begin{remark}[From the ambient cusp set to a finite Farey sphere]\label{rem:no-sequence}
The ambient arithmetic boundary is the infinite set
$\PP^1(\QQ(i))$, and complex continued-fraction algorithms may of course
produce sequences of Gaussian rationals.  What has no direct analogue of the
classical generalized Farey sequence is a canonical cyclic ordering of all
cusps relevant to a three-dimensional subgroup domain: the ambient boundary
is a sphere rather than a circle.  The six-point set $V_g$ is therefore a
\emph{local} basic cell, not the whole Picard analogue of
$\PP^1(\QQ)$.  A finite Picard Farey symbol is assembled from finitely many
such cells and their two-dimensional incidences.  This distinction is already
visible in Schmidt's passage from Farey intervals to Farey triangles and
quadrangles.
\end{remark}

\begin{proposition}[Extended octahedral stabilizers]\label{prop:extended-octahedral-stabilizers}
Let
\[
K=\Stab_G(\mathcal O),\qquad
\widetilde K=\Stab_{\widetilde G}(\mathcal O).
\]
Then
\[
|K|=12,\quad K\cong A_4,\qquad
|\widetilde K|=24,\quad \widetilde K\cong S_4.
\]
The two displayed four-face classes are the two \(K\)-orbits, and \(\widetilde K\) is transitive on all eight faces.  If the antiholomorphic reflection symmetries are also admitted, the full setwise stabilizer of \(\mathcal O\) in \(\widehat G\) has order \(48\).
\end{proposition}
\begin{proof}
The twelve matrices in \eqref{eq:K-order} act faithfully on the four pairs of opposite faces of \(\mathcal O\), giving the rotational tetrahedral group \(A_4\).  The additional symmetry \(\rho(z)=iz+1\) has determinant of nontrivial square class, normalizes the same octahedron and exchanges the two \(K\)-face orbits.  Hence \(\langle K,\rho\rangle\) is a subgroup of \(\widetilde K\) of order \(24\); its faithful action on the four opposite-face pairs is the full symmetric group \(S_4\).  No larger orientation-preserving symmetry group of an octahedron exists.  Adjoining a reflection doubles this order, giving the full order-\(48\) octahedral symmetry group in \(\widehat G\).
\end{proof}

\subsection{Neighbouring octahedra}\label{subsec:neighbour-octahedra}
\begin{proposition}[Universal neighbour maps]\label{prop:universal-neighbours}
For each triangular face \(F\subset\partial\mathcal O\) there is an element \(a_F\in G\) such that \(a_F\mathcal O\) is the unique Farey octahedron on the opposite side of \(F\).  The eight elements may be fixed once and for all and transported by \(G\) to every octahedral occurrence.
\end{proposition}
\begin{proof}
Let \(T_0=(\infty,0,1)\), let \(c(z)=\bar z\) be reflection in the hyperbolic plane with ideal boundary \(T_0\), and let \(u(z)=1-\bar z\), an antiholomorphic symmetry of the positive octahedron.  For a face \(F\), choose \(g_F\in\widetilde G\) carrying \(T_0\) to \(F\), and put
\[
c_F=g_Fcg_F^{-1}.
\]
Then \(c_F\) is reflection in the hyperbolic plane containing \(F\).  Both \(c_F\) and \(u\) are antiholomorphic and have trivial determinant square class, so
\[
a_F=c_Fu
\]
is holomorphic and lies in \(G\).  Since \(u\) preserves \(\mathcal O\), the octahedron \(a_F\mathcal O=c_F\mathcal O\) is its mirror across \(F\).  Hockman proves that every Farey triangle is the common face of exactly two Farey octahedra \cite[Theorem~4.2, pp.~154--155]{Hockman2019}; hence this is the unique neighbour across \(F\).
\end{proof}

\begin{proposition}[A dual four-cycle]\label{prop:dual-square}
Around the ideal Farey edge \(\{0,\infty\}\) lie the four octahedra
\[
\mathcal O,\qquad \mathcal O-1,\qquad
\mathcal O-1-i,\qquad \mathcal O-i.
\]
Consecutive octahedra in this cyclic list share a triangular face, while opposite ones share only the ideal edge \(\{0,\infty\}\).  Hence the dual Farey-octahedral graph contains a four-cycle and is not a tree.
\end{proposition}
\begin{proof}
Translate the six vertices of \(\mathcal O\) by \(0,-1,-1-i,-i\).  The consecutive intersections are respectively the triples
\[
\{0,i,\infty\},\quad
\{-1,0,\infty\},\quad
\{-i,0,\infty\},\quad
\{0,1,\infty\},
\]
each a Farey face.  The two opposite pairs have intersection exactly \(\{0,\infty\}\).  Thus their adjacency graph around this edge is a square.
\end{proof}

\subsection{The exact twelve-cell subdivision}
\begin{proposition}[The exact twelve-cell subdivision]\label{prop:twelve-cells}
Let
\[
 \left\{\infty,0,1,i,1+i,w=\frac{1+i}{2}\right\}
 \subset \partial\mathbf H^3
\]
be the ideal-vertex set of the positively normalized Gaussian Farey
octahedron $\mathcal O\subset\mathbf H^3$, defined by its ideal convex hull, let
$K=\operatorname{Stab}_G(\mathcal O)$, and let $P_0$ be the standard
Picard cell.  For the ordering
\[
 1,a,a^2,c^2,c,c^2a^2,ac,ca,a^2c^2,ac^2,c^2a,aca
\]
of $K$, the cells $P_k=kP_0$ are distinct, have pairwise disjoint
interiors, and satisfy
\[
 \mathcal O=\bigcup_{k\in K}P_k.\tag{3.3}\label{eq:twelve-cells}
\]
The action of $K$ on these cells is simply transitive.  In the fixed
no-inversions subdivision, the $96$ elementary facet occurrences comprise
$48$ internal occurrences, paired into $24$ Picard-face pairs, and $48$
boundary triangles.  Each of the eight ideal faces of $\mathcal O$ contains
exactly six of the boundary triangles.
\end{proposition}

\begin{proof}
We give the finite incidence argument, including the exact containment and
no-gap checks.  Put
\begin{align*}
 F_1&=(0,1,\infty),&F_2&=(1,1+i,\infty),
 &F_3&=(1+i,i,\infty),&F_4&=(i,0,\infty),\\
 F_5&=(0,w,1),&F_6&=(1,w,1+i),
 &F_7&=(1+i,w,i),&F_8&=(i,w,0).
\end{align*}
Write $f_r$ for the finite centre of $F_r$, $e_{uv}$ for the finite
subdivision point of the ideal edge $\{u,v\}$, and
\[
 C=\left(\frac{1+i}{2},\frac{1}{\sqrt2}\right).
\]
Here a finite point is written as $(z,t)$.  The coordinates needed below are
\begin{align*}
 &(f_1,f_2,f_3,f_4)=
 \left(\left(\tfrac12,\tfrac{\sqrt3}{2}\right),
 \left(1+\tfrac i2,\tfrac{\sqrt3}{2}\right),
 \left(\tfrac12+i,\tfrac{\sqrt3}{2}\right),
 \left(\tfrac i2,\tfrac{\sqrt3}{2}\right)\right),\\
 &(f_5,f_6,f_7,f_8)=
 \left(\left(\tfrac12+\tfrac i4,\tfrac{\sqrt3}{4}\right),
 \left(\tfrac34+\tfrac i2,\tfrac{\sqrt3}{4}\right),
 \left(\tfrac12+\tfrac{3i}{4},\tfrac{\sqrt3}{4}\right),
 \left(\tfrac14+\tfrac i2,\tfrac{\sqrt3}{4}\right)\right).
\end{align*}
For an edge with one endpoint $\infty$, $e_{u\infty}=(u,1)$; for an
edge of the unit square, $e_{uv}=((u+v)/2,1/2)$; and
\[
 e_{0w}=\left(\frac{1+i}{3},\frac13\right),\quad
 e_{1w}=\left(\frac{2+i}{3},\frac13\right),\quad
 e_{iw}=\left(\frac{1+2i}{3},\frac13\right),\quad
 e_{(1+i)w}=\left(\frac{2+2i}{3},\frac13\right).
\]
The exact placements are listed in Table~\ref{tab:twelve-placement}; in every
row $p_2=C$.

\begin{table}[htbp]
\centering
\small
\setlength{\tabcolsep}{4.5pt}
\begin{tabular}{@{}c c c c c c@{}}
\toprule
$k$ & ideal edge $(v_0,v_\infty)$ & $p_1$ & $p_2$ & $p_3$ & $j$\\
\midrule
$1$       & $(0,\infty)$ & $f_1$ & $C$ & $f_4$ & $e_{0\infty}$\\
$a$       & $(w,1)$      & $f_6$ & $C$ & $f_5$ & $e_{1w}$\\
$a^2$     & $(i,1+i)$    & $f_3$ & $C$ & $f_7$ & $e_{i(1+i)}$\\
$c^2$     & $(1,0)$      & $f_1$ & $C$ & $f_5$ & $e_{01}$\\
$c$       & $(\infty,1)$ & $f_1$ & $C$ & $f_2$ & $e_{1\infty}$\\
$c^2a^2$  & $(w,i)$      & $f_8$ & $C$ & $f_7$ & $e_{iw}$\\
$ac$      & $(1,1+i)$    & $f_6$ & $C$ & $f_2$ & $e_{1(1+i)}$\\
$ca$      & $(i,0)$      & $f_8$ & $C$ & $f_4$ & $e_{0i}$\\
$a^2c^2$  & $(\infty,i)$ & $f_3$ & $C$ & $f_4$ & $e_{i\infty}$\\
$ac^2$    & $(1+i,w)$    & $f_6$ & $C$ & $f_7$ & $e_{(1+i)w}$\\
$c^2a$    & $(1+i,\infty)$ & $f_3$ & $C$ & $f_2$ & $e_{(1+i)\infty}$\\
$aca$     & $(0,w)$      & $f_8$ & $C$ & $f_5$ & $e_{0w}$\\
\bottomrule
\end{tabular}
\caption{Exact placements of the twelve Picard cells.}
\label{tab:twelve-placement}
\end{table}

The eight closed half-spaces defining $\mathcal O$ are $q_r\geq0$, where
\begin{align*}
 q_1&=y,&q_2&=1-x,&q_3&=1-y,&q_4&=x,\\
 q_5&=(x-\tfrac12)^2+y^2+t^2-\tfrac14,
 &q_6&=(x-1)^2+(y-\tfrac12)^2+t^2-\tfrac14,\\
 q_7&=(x-\tfrac12)^2+(y-1)^2+t^2-\tfrac14,
 &q_8&=x^2+(y-\tfrac12)^2+t^2-\tfrac14.
\end{align*}
Exact substitution of the finite vertices in Table~\ref{tab:twelve-placement}
gives
\[
 \{q_r(v):v\text{ finite}\}=
 \begin{cases}
 \{0,\frac14,\frac13,\frac12,\frac23,\frac34,1\},&1\le r\le4,\\
 \{0,\frac14,\frac13,\frac12,1,\frac32,2\},&5\le r\le8.
 \end{cases}
\]
Thus every vertex, and hence every convex cell $P_k$, lies in $\mathcal O$.

The complete facet incidence is encoded compactly in
Table~\ref{tab:twelve-incidence}.  The four middle columns give the adjacent
cell across the elementary facets $X_0,X_\infty,Y_0,Y_\infty$.  The last two
columns give the octahedral face containing the two boundary triangles
$R_0,R_\infty$, respectively $S_0,S_\infty$.

\begin{table}[htbp]
\centering
\small
\setlength{\tabcolsep}{4.1pt}
\begin{tabular}{@{}c cccc cc@{}}
\toprule
$k$ & $X_0$ & $X_\infty$ & $Y_0$ & $Y_\infty$ & $R_0,R_\infty$ & $S_0,S_\infty$\\
\midrule
$1$      &$c^2$    &$c$       &$ca$       &$a^2c^2$ &$F_1$ &$F_4$\\
$a$      &$ac^2$   &$ac$      &$aca$      &$c^2$    &$F_6$ &$F_5$\\
$a^2$    &$a^2c^2$ &$c^2a$    &$c^2a^2$  &$ac^2$   &$F_3$ &$F_7$\\
$c^2$    &$c$      &$1$       &$a$        &$aca$    &$F_1$ &$F_5$\\
$c$      &$1$      &$c^2$     &$c^2a$     &$ac$     &$F_1$ &$F_2$\\
$c^2a^2$ &$aca$    &$ca$      &$ac^2$     &$a^2$    &$F_8$ &$F_7$\\
$ac$     &$a$      &$ac^2$    &$c$        &$c^2a$   &$F_6$ &$F_2$\\
$ca$     &$c^2a^2$ &$aca$     &$a^2c^2$   &$1$      &$F_8$ &$F_4$\\
$a^2c^2$ &$c^2a$  &$a^2$     &$1$        &$ca$     &$F_3$ &$F_4$\\
$ac^2$   &$ac$     &$a$       &$a^2$      &$c^2a^2$ &$F_6$ &$F_7$\\
$c^2a$   &$a^2$    &$a^2c^2$  &$ac$       &$c$      &$F_3$ &$F_2$\\
$aca$    &$ca$     &$c^2a^2$  &$c^2$      &$a$      &$F_8$ &$F_5$\\
\bottomrule
\end{tabular}
\caption{Exact sector-facet incidence.  Each entry in the middle four columns
has the reciprocal entry with the opposite local type.}
\label{tab:twelve-incidence}
\end{table}

The table has $12\cdot4=48$ internal facet occurrences.  Direct comparison
of their three exact vertices shows that every occurrence has exactly one
reciprocal mate and that the two cells lie on opposite sides; hence there are
$24$ internal pairs.  The remaining $12\cdot4=48$ occurrences are boundary
triangles.  Each $F_r$ occurs in exactly three rows and contributes the two
triangles of the indicated type, so each $F_r$ contains exactly six triangles.
For every $F_r$ these triangles have seven vertices, twelve edges and six
faces; the six boundary edges form one cycle, every interior edge has degree
two, and every vertex link is a path or a cycle as appropriate.  They therefore
form a triangulated disc whose boundary is the subdivided boundary of $F_r$.
The eight discs use the same named subdivision points on common octahedral
edges.  Their union is consequently the entire $\partial\mathcal O$; directly,
its vertex links are cycles and its $f$-vector is $(26,72,48)$.

Let $U=\bigcup_{k\in K}P_k$.  The standard Picard cells form a face-to-face
tessellation, so distinct translates have disjoint interiors.  The matrices
listed above are twelve distinct elements of $\PSL_2(\mathbf Z[i])$, and a
top-dimensional Picard cell has trivial setwise stabilizer; hence the twelve
cells are distinct.  We have proved $U\subseteq\mathcal O$.  By the reciprocal
incidence check, every codimension-one facet of $U$ not lying on
$\partial\mathcal O$ is incident to two cells on opposite sides, whereas the
unpaired facets form all of $\partial\mathcal O$.  Thus $U$ has no relative
polyhedral boundary in the connected interior of $\mathcal O$.  A proper
finite union of closed convex polyhedra inside $\mathcal O$ would have a
codimension-one relative boundary separating one of its cells from a
component of the complement.  Therefore $U=\mathcal O$.

Finally, left multiplication sends $P_k$ to $P_{hk}$ for $h\in K$.  This
action is transitive, and it is free because the twelve cells are distinct.
Hence it is simply transitive.
\end{proof}

\section{Polyhedral subgroup constructions}\label{sec:polyhedral}
This section collects the technical tools needed to pass from the ambient cells to subgroup polyhedra and quotient presentations.  The order is deliberately geometric: Poincar\'e conditions and complete flags come first, followed by Lee's connected coset development, quotient complexes, and finally the precise limit of the modular tree analogy.

\subsection{Poincar\'e polyhedra, flags and no inversions}
A fundamental polyhedron for a discrete subgroup \(H\leq\Isom^+(\HH^3)\) is a finite polyhedral region \(P\) such that the \(H\)-translates cover \(\HH^3\) and distinct translates have disjoint interiors.  Boundary faces occur in side pairs.  Following a face-pairing map around an edge gives an edge cycle and a cycle transformation; the sum of the corresponding dihedral angles is controlled by the order of that transformation.  At an ideal vertex, a sufficiently small horosphere gives a Euclidean polygonal link, and the induced side identifications determine the cusp cross-section.  We use the standard forms of Poincar\'e's theorem in \cite{EpsteinPetronio1994,JespersKieferDelRio2015}.

Two distinctions will be used repeatedly.  First, a three-dimensional fundamental polyhedron and a two-dimensional equivariant retract encode related but different information: collapsing the latter does not automatically produce a new Poincar\'e polyhedron.  Second, a proposed merger of adjacent boundary faces is legitimate only when a single side-pairing transformation extends across their union and the resulting polyhedron satisfies the relevant Poincar\'e hypotheses.  These points become important when comparing the polyhedral picture with maximal-tree presentations.

For an equivariant cell structure we shall always pass, when necessary, to a compatible subdivision on which the action has no inversions.  The barycentric subdivision has this property: its vertices are barycentres of cells of distinct dimensions in a nested chain, so an element preserving a subdivided simplex setwise preserves each vertex and hence fixes the simplex pointwise.  In the Picard calculations below we use the explicit fixed subdivision of Proposition~\ref{prop:twelve-cells}; its complete-flag incidences are chosen compatibly with the same no-inversions convention.  This distinction between an abstract barycentric refinement and the concrete twelve-cell subdivision will be maintained throughout.

A \emph{complete oriented flag} in a three-dimensional cellulation is a nested chain
\[
 v\subset e\subset F\subset P,
\]
consisting of a vertex, edge, two-face and three-cell, together with an orientation of the face and the coorientation pointing into the chosen incident three-cell.  For an ideal triangular face the cyclic order of its vertices records the face orientation.  This is the information used below to distinguish the two incident sides of a face; in mixed finite/ideal flags the lower-dimensional cells remove any ambiguity not visible from ideal vertices alone.

\begin{lemma}[Complete-flag rigidity]\label{lem:flag-rigidity}
If \(g_1,g_2\in G_K\) carry the same complete oriented flag of a no-inversions cellulation to the same target flag, then \(g_1=g_2\).
\end{lemma}
\begin{proof}
The element \(g_2^{-1}g_1\) preserves each cell of the source flag and preserves the chosen side of the top-dimensional face.  The no-inversions convention therefore makes it fix the top-dimensional cell pointwise.  An orientation-preserving isometry of \(\HH^3\) which is the identity on a nonempty open subset is the identity.  Hence \(g_1=g_2\).
\end{proof}

\subsection{Lee's finite-index construction}
Lee's geometric subgroup method starts with a Poincar\'e fundamental polyhedron \(P_0\) for a discrete group \(G\) and the finite set of left cosets \(H\backslash G\), with its right \(G\)-action \cite{Lee1984}.  Representatives can be chosen so that
\[
   P=\bigcup_i g_iP_0
\]
is connected.  If an exposed face of \(g_iP_0\) is carried by an ambient side map \(q\) to a face of \(g_iqP_0\), and \(Hg_iq=Hg_j\), then
\[
   g_iq=h g_j\qquad(h\in H),
\]
and \(h^{-1}\) pairs the corresponding exposed faces.  This provides a geometric counterpart to Reidemeister--Schreier rewriting.

\begin{proposition}[Lee's finite-index polyhedron construction]\label{prop:lee}
Let \(G\) act discretely on \(\HH^3\) with a Poincar\'e fundamental polyhedron \(P_0\), and let \(H\leq G\) have finite index.  A connected choice of representatives \(g_i\) for \(H\backslash G\) gives a fundamental polyhedron \(\bigcup_i g_iP_0\) for \(H\).  The right \(G\)-action on \(H\backslash G\) determines every exposed face identification.
\end{proposition}
\begin{proof}
Choose representatives inductively along the finite coset graph so that each new cell shares a face with the union already chosen.  If an exposed face of \(g_iP_0\) is carried by an ambient side map \(q\) to the corresponding face of \(g_iqP_0\), choose \(j\) with \(Hg_iq=Hg_j\).  There is a unique \(h\in H\) such that \(g_iq=hg_j\); hence \(h^{-1}\) pairs the two exposed faces.  Every translate of \(P_0\) lies in an \(H\)-translate of the chosen union, so these translates cover \(\HH^3\).  If two interiors met, two distinct \(G\)-translates of the interior of \(P_0\) would meet, which is impossible.  Thus the union is a fundamental polyhedron.  This is Lee's construction \cite{Lee1984}.
\end{proof}

\begin{remark}[Lee's Picard example]\label{rem:lee-example}
For the Picard group, Lee carries out Proposition~\ref{prop:lee} explicitly.  For a torsion-free subgroup of index twelve he chooses twelve representatives $a_1,\ldots,a_{12}$ for $H\backslash G$ and forms
\[
Q_{\mathrm{Lee}}=\bigcup_{r=1}^{12}a_rP_0.
\]
He shows that $Q_{\mathrm{Lee}}$ is an ideal octahedron with congruent triangular faces, gives a three-dimensional sketch of the twelve-cell union, and then a spread-out face diagram from which the subgroup pairings are read \cite[pp.~184--190, Figs.~4--10]{Lee1984}.  Figure~\ref{fig:intro-octahedron} records the common ideal octahedron; Appendix~\ref{app:octahedra} records the subgroup-specific face pairings.  Thus Proposition~\ref{prop:lee} should be read as the general mechanism behind a concrete geometry already visible in Lee's Picard examples.
\end{remark}

\begin{remark}[Two uses of Lee's construction]\label{rem:lee-two-transversals}
Two later applications emphasize different choices of transversal.  In
Proposition~\ref{prop:M-adapted-lifting}, representatives are deliberately
chosen in the embedded modular group so that the planar section of the
three-dimensional union is prescribed.  In Appendix~\ref{app:gamma0pi}, the
same finite-index mechanism gives a seven-face domain for
$\Gamma_0(1+i)$, while a different four-face tetrahedral domain is also
proved for the same subgroup.  Equality of the resulting side-pairing
groups and even an explicit dictionary between their generators does not by
itself exhibit a sequence of legal geometric cuts and face mergers.  This is
why reduction of fundamental polyhedra is kept separate from subgroup
reconstruction.
\end{remark}

\subsection{Quotient complexes and presentations}
The two-dimensional quotient information may be organized as a complex of groups.  After removing inversions, cells carry their stabilizers; inclusions of stabilizers along incidences are conjugated by transport elements when different chosen lifts are used.  Choosing a maximal tree in the quotient one-skeleton kills the tree transports but does not remove the attaching relations of the quotient two-cells.  This is the fundamental reason the Bianchi situation differs from the Bass--Serre tree for \(C_2*C_3\).

\begin{proposition}[Maximal-tree presentation]\label{prop:max-tree}
Let \(H\) act without inversions on a simply connected two-dimensional cell complex \(X\) with finite quotient.  Choose one lift of every quotient cell and a maximal tree in the quotient one-skeleton.  Then \(H\) has the standard complex-of-groups presentation: generators are the local vertex groups and the transport elements on non-tree edges; relations are the local group relations, the edge inclusions, and the attaching relations of the quotient two-cells.
\end{proposition}
\begin{proof}
Choose lifts coherently along the maximal tree.  Every tree transport is then the identity.  Moving across a non-tree edge requires one transport element, and the local vertex groups account for stabilizers of the chosen lifts.  The two descriptions of an edge stabilizer inside its endpoint groups give the edge-inclusion relations.  Finally, traversing the attaching map of a lifted two-cell returns to its initial lift, giving the two-cell relation.  Since the complex of groups here is induced by an actual action of \(H\) on the simply connected complex \(X\), it is developable by construction, with development \(X\).  The maximal-tree presentation is the standard one for a developable complex of groups; see Bridson--Haefliger \cite[Chapter~III.C.3, especially Proposition~3.7]{BridsonHaefliger1999}.  Hence these relations generate all relations among the deck transformations.
\end{proof}

\begin{proposition}[Torsion from cell stabilizers]\label{prop:torsion}
Let \(H\) act cellularly and without inversions on an \(H\)-equivariant cellulation of \(\HH^3\).  Then \(H\) is torsion-free if and only if every cell stabilizer is trivial.  For a finite quotient it is enough to test one representative of each cell orbit.
\end{proposition}
\begin{proof}
A nontrivial finite cell stabilizer contains torsion.  Conversely, a finite subgroup of \(\Isom(\HH^3)\) fixes a point of the complete CAT(0) space \(\HH^3\) \cite[Prop.~II.2.7]{BridsonHaefliger1999}.  Let \(\sigma\) be the unique open cell containing such a fixed point.  The finite subgroup preserves \(\sigma\), and the no-inversions convention gives a nontrivial stabilizer of \(\sigma\).  This proves the equivalence.
\end{proof}

\subsection{Merging and the limit of the tree analogy}\label{subsec:tree-obstruction}
Kulkarni's maximal-tree argument is the correct two-dimensional model, but in the Picard setting a maximal tree is a presentation device rather than a replacement for two-cell and cusp data.

\begin{remark}[Merging adjacent paired faces]\label{prop:face-merger}
Suppose adjacent boundary faces \(F_1,F_2\) of a polyhedral fundamental domain have a union \(F=F_1\cup F_2\) which is a single geometric face, and suppose that a single element \(g\in G\) restricts to the prescribed side pairing on both \(F_1\) and \(F_2\) and maps \(F\) to the corresponding paired union.  Then the common edge is artificial for this side identification, and one may replace \(F_1,F_2\) by \(F\) without changing the underlying group action.  If both the original and merged-face domains satisfy the hypotheses of the chosen form of the Poincar\'e polyhedron theorem, the resulting presentations describe the same subgroup and are related by Tietze transformations.

We shall not need a general face-merger theorem below.  In particular, agreement of two side-pairing maps merely along the common geodesic is not by itself sufficient to justify a merger.
\end{remark}

\begin{corollary}[Bass--Serre specialization to the modular group]\label{cor:modular-specialization}
For \(\PSL_2(\ZZ)=C_2*C_3\), the no-inversions Farey tessellation collapses equivariantly to the Farey tree.  A maximal-tree presentation of a finite-index subgroup therefore gives its Kurosh decomposition
\[
   H\cong\Bigl(*C_2\Bigr)*\Bigl(*C_3\Bigr)*F_r.
\]
Kulkarni's special polygons geometrically realize the same decomposition through their even, odd and free side identifications.
\end{corollary}
\begin{proof}
The barycentrically subdivided Farey tessellation retracts to the Bass--Serre tree of \(C_2*C_3\).  Vertex stabilizers are trivial, \(C_2\) or \(C_3\), while edge stabilizers are trivial.  Proposition~\ref{prop:max-tree} therefore gives the displayed free product.  Kulkarni's correspondence between special polygons and Farey symbols records these finite factors and free side pairings on the boundary \cite{Kulkarni1991}.
\end{proof}

For the Picard group the corresponding small reflection region is $D$, whose section in $\Pi=\{y=0\}$ is Kulkarni's triangle.  Its finite face is $B_D=[j,p,q,p']$.  Thus the analogue of the finite side of the small triangle is a \emph{two-dimensional face}, not a selected graph of ridges.  The finite edges $jp,pq,qp',p'j$ already form the boundary of a quadrilateral.  Their translates are the one-skeleton of $\mathscr X=GB_D$; retaining the quadrilateral faces supplies the attaching relations that a graph alone omits.

A maximal tree in the one-skeleton of $H\backslash\mathscr X$ chooses coherent lifts and removes tree transport letters, as in Proposition~\ref{prop:max-tree}.  It does not remove local vertex groups, non-tree transports, or the square attaching relations.  Rank-two peripheral subgroups also remain.  This is a presentation-theoretic use of a maximal tree, not a proof that cutting the three-dimensional domain produces a special polyhedron or independent face generators.

\begin{proposition}[Parabolic obstruction to Rademacher independence]\label{prop:z2-obstruction}
Let \(H\leq G=\PSL_2(\ZZ[i])\) have finite index.  Then \(H\) contains a rank-two free abelian subgroup.  More precisely, if
\[
N=\left\{\begin{pmatrix}1&z\\0&1\end{pmatrix}:z\in\ZZ[i]\right\}/\{\pm I\}\cong \ZZ[i],
\]
then \(H\cap N\) has finite index in \(N\), and hence \(H\cap N\cong\ZZ^2\).  Consequently no finite-index Picard subgroup admits a decomposition as a free product of finite cyclic groups and a free group.
\end{proposition}
\begin{proof}
For every subgroup \(A\leq G\) one has
\[
[A:A\cap H]\leq [G:H].
\]
Taking \(A=N\) shows that \(H\cap N\) has finite index in the lattice \(N\cong\ZZ^2\), hence is itself isomorphic to \(\ZZ^2\).  On the other hand every abelian subgroup of a free product of cyclic groups is cyclic or conjugate into one of the cyclic factors.  Such a group contains no subgroup isomorphic to \(\ZZ^2\).  Thus the Rademacher--Kurosh conclusion available in the modular case cannot survive literally for finite-index Picard subgroups.
\end{proof}

\begin{remark}[What survives from the modular picture]\label{rem:local-not-tree}
The obstruction is not to using a spanning tree for selecting cells or writing a presentation.  It is to replacing the Picard finite-face complex by the modular finite-edge tree and retaining Kulkarni's conclusions.  Besides the quadrilateral boundaries above, four Farey octahedra around an ideal edge form a cycle in the dual adjacency graph.  These observations do not supply a unique closest-side rule of the kind used in the modular algorithms \cite{ChanLangLimTan1993,LangLimTan1995}.  A three-dimensional cutting argument must still control face incidence, ridge cycles, finite vertex links and rank-two cusp links.  No general admissibility or minimum-side assertion follows from the classical retract.
\end{remark}

\section{Picard Farey symbols and subgroup reconstruction}\label{sec:picard-farey}
This section packages the geometry of Sections~\ref{sec:picard-geometry} and~\ref{sec:gaussian-octahedra} into a finite object attached to a subgroup \(H\leq G\).  The six Gaussian-rational vertices and their Farey triangles record the arithmetic boundary geometry; the local groups record isotropy; and the face-pairing transformations record how the octahedral pieces are identified.  The larger group \(\widetilde G\) is used only to understand the two ambient face classes: the subgroup information and all pairing transformations remain in \(G\).  After defining the data and their consistency conditions, we prove that they recover the finite \(G\)-set \(H\backslash G\), and then show that a symbol can be found from an effective membership test for \(H\).  The distinction between construction and later questions of reduction or minimality is kept explicit.  In particular, the reconstructed finite $G$-set does not select a preferred geometric development: the $M$-adapted representatives of Proposition~\ref{prop:M-adapted-lifting} and an octahedrally compressed development may encode the same subgroup while having different intersections with the modular plane.  The Gaussian Farey tessellation and finite coset actions are established background; the specific result of this section is the marked octahedral package together with its reversible reconstruction and finite double-coset construction.

\subsection{Picard--Farey sets and the two face classes}\label{subsec:farey-spheres}
\begin{definition}[Picard--Farey set and Farey sphere]\label{def:picard-farey-set}
Put
\[
V_0=\left\{\infty,0,1,i,1+i,\frac{1+i}{2}\right\}\subset \PP^1(\QQ(i)).
\]
For $g\in G$, the six-point set
\[
V_g=gV_0\subset \PP^1(\QQ(i))
\]
is called a \emph{Picard--Farey set}.  Its twelve unimodular Farey edges and eight Farey triangles form the boundary of the ideal octahedron $g\mathcal O$; we call this triangulated two-sphere the associated \emph{Farey sphere}.  Thus the Picard--Farey set is the arithmetic ideal-boundary vertex set, while the Farey sphere remembers its two-dimensional incidence.
\end{definition}

\begin{proposition}[Gaussian denominator restrictions]\label{prop:gaussian-denominators}
Let
\[
 g=\begin{pmatrix}a&b\\ c&d\end{pmatrix}\in\mathrm{SL}_2(\ZZ[i]),
 \qquad
 u=\binom ac,\quad v=\binom bd.
\]
With the ordering of Definition~\ref{def:picard-farey-set}, homogeneous representatives for the six vertices of $g\mathcal O$ may be chosen as
\[
 u,\quad v,\quad u+v,\quad iu+v,\quad (1+i)u+v,\quad u+(1-i)v.
\]
Consequently their Gaussian denominators are
\[
 c,\quad d,\quad c+d,\quad ic+d,\quad (1+i)c+d,\quad c+(1-i)d.
\tag{5.1}\label{eq:gaussian-denominator-six}
\]
Moreover, if
\[
 g'=\begin{pmatrix}a'&b'\\ c&d\end{pmatrix}\in\mathrm{SL}_2(\ZZ[i])
\]
has the same lower row, then there is a unique $t\in\ZZ[i]$ such that
\[
 g'=\begin{pmatrix}1&t\\0&1\end{pmatrix}g.
\]
Thus the denominator pair $(c,d)$ determines the ordered Picard--Farey set up to simultaneous Gaussian translation of all six cusps.
\end{proposition}

\begin{proof}
The standard cusps $\infty,0,1,i,1+i,w$ have homogeneous representatives
\[
 (1,0)^t,\ (0,1)^t,\ (1,1)^t,\ (i,1)^t,\ (1+i,1)^t,\ (1,1-i)^t,
\]
where $1/(1-i)=(1+i)/2=w$.  Applying $g$ gives the six displayed vectors and hence \eqref{eq:gaussian-denominator-six}.

For the uniqueness assertion, $ad-bc=a'd-b'c=1$, so
\[
 (a'-a)d=(b'-b)c.
\]
Since $(c,d)=\ZZ[i]$ for a unimodular matrix and $\ZZ[i]$ is a principal ideal domain, there is $t\in\ZZ[i]$ with
\[
 a'-a=tc,\qquad b'-b=td.
\]
This is exactly the displayed left multiplication by the translation matrix.  Uniqueness of $t$ follows because $c$ and $d$ are not both zero.
\end{proof}

\begin{remark}[Comparison with Kulkarni's denominator restriction]\label{rem:kulkarni-56-gaussian}
Kulkarni proves that, after normalizing a generalized Farey sequence inside an interval $[k,k+1)$, its denominators determine the corresponding numerators uniquely \cite[\S5.6, pp.~1073--1074]{Kulkarni1991}.  Proposition~\ref{prop:gaussian-denominators} is the closest direct Gaussian analogue for one ordered Farey octahedron: the lower row determines all six denominator coordinates and determines the numerator row up to the parabolic translation group $\ZZ[i]$.  It is not a uniqueness theorem for an arbitrary unmarked Picard Farey symbol; choosing reduced representatives for the translation action is part of the arithmetic reduction problem in Section~\ref{sec:open}.
\end{remark}

For a subgroup $H$, an octahedral occurrence $x=HgK$ is represented by one such placed Farey sphere.  The Picard--Farey set alone does \emph{not} determine $H$: the subgroup information begins only when we add the local group $J_x$ and the ordered, cooriented identifications of its triangular faces.  In this sense the construction has three visible layers
\[
\text{Picard--Farey set}\;\longrightarrow\;\text{Farey sphere}\;\longrightarrow\;\text{marked Picard Farey data}.
\]

The face type of Definition~\ref{def:triangle-type} gives the Farey sphere an additional intrinsic two-colouring.  Every local group \(J_x\leq K\) preserves this colouring.  A valid face pairing is necessarily \(G\)-valued and hence can only identify faces of the same determinant type; the following lemma makes this visible before the full twelve-sector consistency test is applied.

\begin{lemma}[Determinant type and face pairing]\label{lem:ideal-face-transport}
Let \(F=(\alpha,\beta,\gamma)\) and \(F'=(\alpha',\beta',\gamma')\) be ordered Gaussian Farey triangles.  There is a unique projective transformation \(q\in\PGL_2(\QQ(i))\) carrying the ordered triple \(F\) to \(F'\), and in fact \(q\in\widetilde G\).  Moreover
\[
q\in G\quad\Longleftrightarrow\quad \chi(F)=\chi(F').
\]
Equivalently, after clearing denominators and dividing by a common Gaussian divisor, the resulting primitive matrix \(A_0\in M_2(\ZZ[i])\) represents an element of \(G\) exactly when
\[
\det A_0\in\{1,-1\}.
\]
When these conditions hold, \(q\) carries the six-triangle subdivision of \(F\) to the corresponding subdivision of \(F'\).  Specifying the opposite target coorientation then gives the source-to-target face-pairing convention used in Definition~\ref{def:farey-symbol}.
\end{lemma}
\begin{proof}
Three distinct points of \(\PP^1(\mathbf C)\) determine a unique projective M\"obius transformation.  Because both triples are Gaussian Farey triples, the corresponding primitive cusp columns form unimodular pairs, and the resulting projective class has an integral Gaussian representative of unit determinant; hence it lies in \(\widetilde G\).  Choose \(g_F,g_{F'}\in\widetilde G\) carrying the standard ordered triangle \((\infty,0,1)\) to \(F,F'\).  Up to a permutation of the standard triple, which is represented by an element of \(G\), one has
\[
q=g_{F'}g_F^{-1}.
\]
Therefore
\[
\delta(q)=\chi(F')-\chi(F)\in C_2,
\]
and \(q\in\ker\delta=G\) exactly when the two types agree.  For a primitive integral representative, the same statement is the displayed determinant test: determinant \(1\) is already special linear and determinant \(-1\) becomes determinant \(1\) after multiplication by the scalar \(i\).  Finally every element of \(G\) preserves the Picard-cell tessellation, so it carries the six elementary boundary triangles of \(F\) to those of \(F'\).  The coorientation condition selects the incident octahedron on the target side.
\end{proof}

\subsection{Octahedral occurrences and local fibres}\label{subsec:octahedral-fibres}
\begin{lemma}[Octahedral fibres]\label{lem:octahedral-fibres}
Let \(H\leq G\) have finite index and let \(x=HgK\in H\backslash G/K\).  Put
\[
   J_x=g^{-1}Hg\cap K.
\]
Then the fibre of \(H\backslash G\to H\backslash G/K\) above \(x\) is naturally
\[
   J_x\backslash K,
\]
and therefore
\[
   [G:H]=\sum_{x\in H\backslash G/K}[K:J_x].
\tag{5.2}\label{eq:index-sum}
\]
Changing the representative \(g\) conjugates \(J_x\) inside \(K\).
\end{lemma}
\begin{proof}
Send \(J_xk\) to \(Hgk\).  If \(k_1=j k_2\) with \(j\in J_x\), then \(gjg^{-1}\in H\) and the two left cosets agree.  Conversely
\[
   Hgk_1=Hgk_2
   \iff gk_1k_2^{-1}g^{-1}\in H
   \iff k_1k_2^{-1}\in g^{-1}Hg\cap K=J_x.
\]
Thus the map is bijective.  Summing the cardinalities of the fibres gives \eqref{eq:index-sum}.  Replacing \(g\) by \(hgk_0\), with \(h\in H\) and \(k_0\in K\), conjugates \(J_x\) by \(k_0\), so only its \(K\)-conjugacy class depends on the unmarked occurrence.
\end{proof}

\begin{definition}[Picard Farey data and symbols]\label{def:farey-symbol}
A \emph{marked Picard Farey datum} \(\Sigma\) consists of the following information.
\begin{enumerate}[label=(\roman*),leftmargin=2.2em]
\item A finite set \(X\).  For each \(x\in X\), an oriented Gaussian Farey octahedron
\[
\OO_x=g_x\OO
=\bigl(g_x\infty,g_x0,g_x1,g_xi,g_x(1+i),g_xw\bigr),
\qquad g_x\in G,
\]
specified by its ordered six Gaussian-rational vertices.
\item A local subgroup
\[
\widehat J_x\leq\Stab_G(\OO_x),
\qquad J_x=g_x^{-1}\widehat J_xg_x\leq K.
\]
\item For each \(\widehat J_x\)-orbit of outward-cooriented oriented triangular faces \(F\subset\partial\OO_x\), a paired outward-cooriented face \(F'\subset\partial\OO_y\) and a source-to-target element \(q:F\to F'\) in normalized coordinates.  If
\[
J_{x,F}=\{j\in J_x:jF=F\}
\]
denotes the setwise stabilizer of the cooriented oriented face, then
\[
qF=\overline{F'},\qquad
qJ_{x,F}q^{-1}=J_{y,F'}.
\tag{5.3}\label{eq:face-equivariance}
\]
Here \(\overline{F'}\) is the same geometric face with the opposite boundary orientation and coorientation.  A chosen based ordered triple may be used to specify \(q\), but it is not part of the definition of \(J_{x,F}\).
\item The pairing is involutive: the reverse pairing is represented by \(q^{-1}\).  Changing source and target representatives by \(j\in J_x\) and \(j'\in J_y\) replaces
\[
(F,F',q)\quad\text{by}\quad(jF,j'F',j'qj^{-1}).
\tag{5.4}\label{eq:face-representatives}
\]
The graph whose vertices are the occurrences \(X\) and whose edges are paired face orbits is required to be connected.
\item One occurrence and one Picard cell in its fibre are distinguished; this is the marking.
\end{enumerate}

Expand every occurrence through the fixed subdivision of Proposition~\ref{prop:twelve-cells}.  A marked datum is called a \emph{marked Picard Farey symbol} if the induced maps \(S,R,X,Y\) on
\[
\mathscr C_\Sigma=\bigsqcup_{x\in X}J_x\backslash K
\]
are total permutations, satisfy all eight relators in \eqref{eq:picard-presentation}, and act transitively.

Two marked Picard Farey symbols are \emph{marked equivalent} if their expanded pointed \(G\)-sets are isomorphic.  After forgetting the distinguished cell, two Picard Farey symbols are \emph{unmarked equivalent} if their expanded \(G\)-sets are isomorphic.
\end{definition}

The definition keeps the Gaussian rational coordinates visible while separating three pieces of information: the placement of each octahedron, its local isotropy, and its face-pairing transformations.

\begin{proposition}[Change of normalized placements]\label{prop:arithmetic-normalized}
Suppose the normalized placement of an occurrence is changed by
\[
g_x\longmapsto g_xk_x,\qquad k_x\in K.
\]
Then
\[
J_x\longmapsto J_x'=k_x^{-1}J_xk_x,
\qquad
q_{xy}\longmapsto q_{xy}'=k_y^{-1}q_{xy}k_x.
\tag{5.5}\label{eq:normalization-law}
\]
The corresponding fibres are canonically identified by
\[
J_xk\longmapsto J_x'(k_x^{-1}k).
\tag{5.6}\label{eq:fibre-normalization}
\]
Under this identification the reconstructed cell transitions agree.  Thus changing normalized placements changes only representatives, not the reconstructed pointed \(G\)-set.
\end{proposition}
\begin{proof}
The local-group formula follows directly from
\[
(g_xk_x)^{-1}H(g_xk_x)\cap K=k_x^{-1}J_xk_x.
\]
If \(q=g_y^{-1}hg_x\) is a normalized source-to-target transport, then
\[
q'=(g_yk_y)^{-1}h(g_xk_x)=k_y^{-1}qk_x.
\]
The map \eqref{eq:fibre-normalization} is well defined because \(k_x^{-1}jk_x\in J_x'\) for \(j\in J_x\), and it preserves the represented geometric cell.  If \(k,\ell\in K\) are source and target sector representatives and \(\tau\) is the fixed elementary transport, then
\[
\tau=\ell^{-1}qk,\qquad \ell=qk\tau^{-1}.
\tag{5.7}\label{eq:elementary-lookup}
\]
After the normalization change, \(k'=k_x^{-1}k\) and \(\ell'=k_y^{-1}\ell\) give
\[
(\ell')^{-1}q'k'=\ell^{-1}qk=\tau.
\]
Hence every reconstructed transition is unchanged under the canonical fibre identification.
\end{proof}

\begin{example}[The full Picard group]\label{ex:picard-symbol}
For \(H=G\) there is one octahedral occurrence and \(J=K\).  Hence
\[
   J\backslash K=K\backslash K
\]
has one element: the twelve Picard cells of the ambient octahedron become one cell in \(G\backslash\HH^3\).  The face maps are the ambient \(S,R,X,Y\) identifications.  This is the smallest possible Picard Farey symbol.
\end{example}

\begin{example}[The subgroup \(\Gamma_0(1+i)\)]\label{ex:gamma0pi-symbol}
For \(H=\Gamma_0(1+i)\), the Farey quotient still has one octahedral occurrence, but
\[
   J=\{1,ac^2,c^2a,aca\},\qquad |J|=4.
\]
Thus \([K:J]=3\), and one may take three representatives \(1,c,c^2\) for the Picard cells above the octahedron.  Internal adjacency accounts for three face pairs; the remaining nine face pairs are boundary identifications.  This example shows concretely why a Farey symbol must retain the local group: one octahedron need not mean twelve distinct Picard cells.
\end{example}

\begin{center}\small
\begin{tabular}{lccc}
\toprule
subgroup & local group & \(|J|\) & \(|J\backslash K|\)\\
\midrule
\(G\) & \(K\) & 12 & 1\\
\(\Gamma_0(1+i)\) & \(\{1,ac^2,c^2a,aca\}\) & 4 & 3\\
\bottomrule
\end{tabular}
\end{center}
\subsection{Reconstruction}\label{subsec:reconstruction}
\begin{theorem}[Picard Farey reconstruction]\label{thm:farey-reconstruction}
Let \(\Sigma\) be a marked Picard Farey datum.  Its local and face data determine at most one transition system on
\[
\mathscr C_\Sigma=\bigsqcup_{x\in X}J_x\backslash K.
\tag{5.8}\label{eq:expanded-cells}
\]
The datum is a marked Picard Farey symbol precisely when these transitions make \(\mathscr C_\Sigma\) a transitive finite right \(G\)-set with the prescribed action of \(S,R,X,Y\).  For a marked Picard Farey symbol, if \(c_0\) is the distinguished cell, then
\[
H_\Sigma=\Stab_G(c_0),
\qquad
[G:H_\Sigma]=|\mathscr C_\Sigma|=\sum_{x\in X}[K:J_x].
\]

Conversely, every marked finite-index subgroup \(H\leq G\) yields a marked Picard Farey symbol for which
\[
\bigsqcup_{x\in H\backslash G/K}J_x\backslash K
\longrightarrow H\backslash G,
\qquad
J_xk\longmapsto Hg_xk,
\tag{5.9}\label{eq:fibre-coset-iso}
\]
is an isomorphism of pointed right \(G\)-sets.  Consequently marked Picard Farey symbols modulo marked equivalence correspond to finite-index subgroups of \(G\), while unmarked Picard Farey symbols modulo unmarked equivalence correspond to conjugacy classes.  No uniqueness or canonical choice of a combinatorial symbol for a subgroup is asserted.
\end{theorem}
\begin{proof}
For an occurrence \(x\), replace the octahedron by the fibre \(J_x\backslash K\).  Proposition~\ref{prop:twelve-cells} supplies every adjacency lying in the interior of that octahedron.

Consider a boundary face pairing \((F,F',q)\) and a source sector represented by \(J_xk\).  The fixed twelve-cell table contains a unique elementary transport \(\tau\) of the relevant complete oriented flag, and the target sector is determined by
\[
\tau=\ell^{-1}qk,
\qquad
\ell=qk\tau^{-1}.
\]
If the source and target face representatives are changed by \(j\in J_x\) and \(j'\in J_y\), then
\[
q'=j'qj^{-1},\qquad k'=jk,\qquad \ell'=j'\ell
\]
give
\[
(\ell')^{-1}q'k'=\ell^{-1}qk=\tau.
\]
Thus the same elementary lookup is made and \(J_y\ell'=J_y\ell\).  If two source changes preserve the same cooriented face, their quotient lies in \(J_{x,F}\); equation~\eqref{eq:face-equivariance} carries that quotient to \(J_{y,F'}\).  This is exactly the remaining representative-independence condition.  Equality, rather than one-sided inclusion, also makes the reverse transport descend.  The inverse pairing is represented by \(q^{-1}\).

The construction therefore gives four partially defined maps \(S,R,X,Y\) on \(\mathscr C_\Sigma\).  By definition of a Picard Farey symbol, they are total permutations, satisfy the eight relations in \eqref{eq:picard-presentation}, and act transitively.  Hence they define a transitive finite right \(G\)-set.  The stabilizer of the distinguished point is \(H_\Sigma\), and the usual transitive-action correspondence gives the index formula.

Conversely, let \(H\leq G\) have finite index and choose representatives \(g_x\) for \(H\backslash G/K\).  Put \(J_x=g_x^{-1}Hg_x\cap K\) and record the actual outward-cooriented octahedral face-pairing transformations in the Gaussian Farey tessellation.  Their source-to-target representatives satisfy \eqref{eq:face-equivariance}, and Lemma~\ref{lem:octahedral-fibres} gives the bijection \eqref{eq:fibre-coset-iso}.  Under this bijection the reconstructed \(S,R,X,Y\)-maps are the original right \(G\)-action on \(H\backslash G\), so the resulting datum is a Picard Farey symbol.  The final equivalence statements are the standard pointed and unpointed transitive \(G\)-set correspondences.
\end{proof}

\subsection{The Picard finite-face covering}\label{subsec:finite-face-covering}
Put
\[
(v_0,v_1,v_2,v_3)=(j,p,q,p'),
\qquad e_i=v_iv_{i+1}\quad(i\bmod 4),
\]
and write
\[
\begin{array}{c|c|c}
 i&A_i=G_{e_i}&V_i=G_{v_i}\\ \hline
 0&\langle R\rangle&\langle S,R\rangle\\
 1&\langle X\rangle&\langle R,X\rangle\\
 2&\langle Y\rangle&\langle X,Y\rangle\\
 3&\langle S\rangle&\langle Y,S\rangle .
\end{array}
\]
Thus the edge groups have orders $2,3,3,2$, the vertex groups have types
$C_2^2,S_3,A_4,S_3$, the open two-cell has trivial stabilizer, and
\[
V_2=\langle X,Y\rangle=G_q=K.
\]
Let
\[
\mathscr B=G\backslash\!\backslash\mathscr X
\]
denote the quotient together with its local groups, incidence maps and twisting
data.

\begin{proposition}[The finite-face covering associated with a symbol]\label{prop:finite-face-covering}
Let $(\Sigma,c_0)$ be a valid marked Picard Farey symbol.  Its verified
transition permutations on
\[
\mathscr C_\Sigma=\coprod_xJ_x\backslash K
\]
determine a pointed connected finite covering of complexes of groups
\[
\mathscr Y_\Sigma\longrightarrow\mathscr B.
\]
The monodromy on the fibre is the pointed right $G$-set of
Theorem~\ref{thm:farey-reconstruction}.  If
\[
H=\{g\in G:c_0g=c_0\},
\]
then the covering is isomorphic over $\mathscr B$ to
\[
\mathscr X_H=H\backslash\!\backslash\mathscr X.
\]
Under the pointed equivariant identification
$\mathscr C_\Sigma\simeq H\backslash G$, the typed cells are
\[
\begin{aligned}
C_2(\mathscr X_H)&=H\backslash G,\\
C_1(\mathscr X_H)&=\coprod_{i=0}^3(H\backslash G)/A_i,\\
C_0(\mathscr X_H)&=\coprod_{i=0}^3(H\backslash G)/V_i.
\end{aligned}
\]
The edge class $[c]_{A_i}$ has endpoints $[c]_{V_i}$ and
$[c]_{V_{i+1}}$, and the square represented by $c$ has the corresponding
ordered attaching loop.  For a template cell $\sigma$ with stabilizer
$A=G_\sigma$ and $c=Hg$, the normalized local group is
\[
L_A(c)=\{a\in A:ca=c\}=A\cap g^{-1}Hg,
\]
whereas the stabilizer in $H$ of the chosen geometric lift $g\sigma$ is
\[
H_{g\sigma}=gL_A(c)g^{-1}=H\cap gAg^{-1}.
\]
The quotient retains the incidence injections and twisting elements induced
by the fixed template inclusions.

Moreover,
\[
\{\text{octahedral occurrences}\}
 =H\backslash G/K
 =\{\text{$q$-type vertices of }\mathscr X_H\}.
\]
If $x=Hg_xK$ and $J_x=K\cap g_x^{-1}Hg_x$, the square two-cells incident
to the corresponding $q$-vertex are naturally indexed by
$J_x\backslash K$, and the normalized vertex group in the chart $g_x$ is
$J_x$, up to the usual change of chart.
\end{proposition}

\begin{proof}
Begin with the transition permutations on $\mathscr C_\Sigma$, after the
elementary sector lookup and its representative-independence check have been
verified.  The eight relations in \eqref{eq:picard-presentation} restrict to
right actions of the four finite groups $V_i$, and their restrictions to the
common cyclic edge groups $A_i$ agree.  For every template group $A$, form
the action groupoid $\mathscr C_\Sigma//A$ and use the inclusion functors
coming from $A_i\leq V_i,V_{i+1}$.  Attach one square for each
$c\in\mathscr C_\Sigma$, with side classes $[c]_{A_i}$ and corner classes
$[c]_{V_i}$.

For the local covering condition, consider an incidence $A\leq B$, choose
$c\in\mathscr C_\Sigma$, and put $L=L_B(c)$.  The $A$-orbits incident to
the $B$-orbit of $c$ are indexed by $L\backslash B/A$, and the natural map
\[
\coprod_{LtA\in L\backslash B/A}
L/(L\cap tAt^{-1})\longrightarrow B/A,
\qquad
l(L\cap tAt^{-1})\longmapsto ltA
\]
is a bijection, since its summands are precisely the $L$-orbits on $B/A$.
These maps are compatible on incidence flags, and the square attaching loop
lifts because one complete square has been attached for every sector.  Hence
the local actions define a finite covering of $\mathscr B$.

The contractible development $\mathscr X$ identifies the fundamental group
of the base complex of groups with $G$.  Path lifting therefore gives a right
$G$-action on the fibre $\mathscr C_\Sigma$.  Its generator maps are the
original transitions, and transitivity makes the covering connected.  For the
stabilizer $H$ of the marked sheet, the map
\[
H\backslash G\longrightarrow\mathscr C_\Sigma,
\qquad Hg\longmapsto c_0g,
\]
is a pointed equivariant bijection.  The orbit formulas above then identify
the covering with $H\backslash\!\backslash\mathscr X$ and give the stated
normalized and geometric stabilizers.

Finally $V_2=K$.  Thus the $q$-type vertices are
$(H\backslash G)/K=H\backslash G/K$.  For $x=Hg_xK$, the incident squares
are the classes $Hg_xk$ with $k\in K$, which form $J_x\backslash K$; their
common normalized $q$-vertex group is $J_x$ in the chart $g_x$.
\end{proof}

\begin{remark}[Dependence of the covering proof]\label{rem:finite-face-dependence}
This proof begins only after the fixed twelve-sector lookup and the
representative-independence argument have supplied total transition
permutations satisfying all eight relations in \eqref{eq:picard-presentation}
and acting transitively.  It is therefore a second organization of the
reconstruction direction, not a validity test for arbitrary coarse face
data.  It neither replaces the twelve-sector table nor proves the construction
from an effective membership test.  It also gives no special-polyhedron
admissibility, independent-generator, or minimum-complexity theorem.
\end{remark}

\begin{remark}[Direct evaluation and marking]\label{rem:finite-face-directness}
The finite-face quotient can be evaluated from the sector labels and their
verified transitions without first choosing a matrix representative for every
Picard cell.  Its explicit two-dimensional quotient nevertheless contains
\[
|\mathscr C_\Sigma|=\sum_x[K:J_x]=[G:H]
\]
quadrilateral two-cells.  This is a compatibility between two finite
descriptions, not a further compression bound; the associated finite-face presentation is
the classical one.

Conversely, a fully marked covering over $\mathscr B$ determines its pointed
right $G$-set.  Its $K$-orbits and their stabilizers recover the octahedral
occurrences and the groups $J_x$, up to change of chart.  Choosing octahedral
representatives then recovers a marked Picard Farey symbol in the same marked
equivalence class.  The covering does not canonically select the particular
printed placements $g_x$, ordered Gaussian-rational faces, or representative
face-pairing matrices.  Thus the two descriptions organize the same abstract
pointed subgroup information while retaining different choices of geometric
representatives.
\end{remark}

\begin{corollary}[Classical Farey-symbol specialization]\label{cor:classical-farey-symbol}
In the modular case the finite-face covering degenerates to the finite graph
of groups obtained from the Farey tree.  Kulkarni's correspondence identifies
its marked special-polygon development, with even, odd and free side
identifications, with a marked Farey symbol.  This is the one-dimensional
specialization of the reconstruction formalism; it is not obtained by
intersecting an arbitrary Picard subgroup with the embedded modular group.
\end{corollary}
\begin{proof}
Corollary~\ref{cor:modular-specialization} identifies the quotient with a
finite graph of groups obtained from the Farey tree, and
Subsection~\ref{subsec:tree-obstruction} records Kulkarni's maximal-tree
cutting and side-pairing interpretation.  Kulkarni's correspondence between
special polygons and Farey symbols then gives the asserted marked
specialization \cite{Kulkarni1991}.
\end{proof}

\subsection{Construction from an effective membership test}\label{subsec:constructive-symbol}
The converse direction of Theorem~\ref{thm:farey-reconstruction} can be made effective without introducing a continued-fraction height or first constructing a full Picard-cell transversal.

\begin{theorem}[Construction from subgroup membership]\label{thm:constructive-symbol}
Let \(H\leq G\) have finite index, and suppose that membership in \(H\) is decidable for exact matrices in \(G\).  Then a marked Picard Farey symbol for \(H\) can be constructed using finitely many membership tests from the fixed Farey octahedron \(\mathcal O\), its stabilizer \(K\), and the universal neighbour maps \(a_F\) of Proposition~\ref{prop:universal-neighbours}.  Neither the index \([G:H]\) nor a previously known fundamental polyhedron for \(H\) need be supplied.
\end{theorem}
\begin{proof}
Maintain a list of representatives \(g_x\in G\) for discovered double cosets \(Hg_xK\in H\backslash G/K\), beginning with \(g_0=1\).  For a discovered occurrence \(Hg_xK\) and a face \(F\subset\partial\mathcal O\), the adjacent octahedron is \(g_xa_F\mathcal O\).  To decide whether it represents an already discovered occurrence \(Hg_yK\), use
\[
Hg_xa_FK=Hg_yK
\quad\Longleftrightarrow\quad
 g_xa_Fkg_y^{-1}\in H\ \text{for some }k\in K.
\]
Indeed, the right-hand condition is exactly the statement that \(g_xa_F\in Hg_yK\).  Since \(|K|=12\), each comparison requires only finitely many membership tests.

For every discovered occurrence compute
\[
J_x=\{k\in K:g_xkg_x^{-1}\in H\}.
\]
If the comparison with \(Hg_yK\) succeeds, retain a witness
\[
h=g_xa_Fkg_y^{-1}\in H.
\]
The element \(h\) carries the chosen target octahedron \(g_y\mathcal O\) to the actual octahedron across the source face, so the physical source-to-target face-pairing transformation is \(h^{-1}\).  In the normalized coordinates of the two occurrences the corresponding face-pairing transformation is
\[
q=g_y^{-1}h^{-1}g_x=k^{-1}a_F^{-1}.
\]
It therefore satisfies the fixed convention \(qF=\overline{F'}\), and the local face-stabilizer conjugacy follows from the fact that both sides are stabilizers of the same face in the actual \(H\)-quotient.  Record one representative for each undirected pairing and use its inverse for the reverse direction.

Whenever no existing double coset passes the test, add the new representative \(g_xa_F\).  The octahedral adjacency graph is connected, so repeated processing eventually reaches every element of \(H\backslash G/K\).  Since this double-coset set is finite, the procedure terminates.  The resulting occurrences, local groups and face-pairing transformations are the actual decorated quotient of the Farey-octahedral tessellation by \(H\); hence, by the converse direction of Theorem~\ref{thm:farey-reconstruction}, they form a marked Picard Farey symbol for \(H\).
\end{proof}

\begin{remark}[Construction versus reduction]\label{rem:construction-not-reduction}
A fixed ordering of faces and of the finite group \(K\) makes the procedure deterministic, but not intrinsically reduced or canonical.  Hockman's triangle norm is useful along selected geodesic branches, but it is invariant under Gaussian translations and does not provide a global unique-predecessor function on the octahedral adjacency complex.  Thus Theorem~\ref{thm:constructive-symbol} gives an effective existence procedure, while arithmetic reduction and geometric simplification or minimality remain separate questions.
\end{remark}

\begin{remark}[Relation with the modular and Hecke constructions]\label{rem:cllt-comparison}
Chan--Lang--Lim--Tan replace the trial-and-error construction of modular special polygons by a systematic development procedure \cite[\S3, pp.~20--22]{ChanLangLimTan1993}; Lang--Lim--Tan develop the corresponding construction for Hecke groups \cite[\S3, pp.~581--583]{LangLimTan1995}.  Their procedures use the rooted tree underlying the two-dimensional tessellation, together with a closest-side or depth structure.  Theorem~\ref{thm:constructive-symbol} has the analogous effective role for the Picard group, but the mechanism is different: the octahedral adjacency graph has cycles, so the construction proceeds by finite recognition of double cosets in $H\backslash G/K$.  It produces a Picard Farey symbol, not a reduced or minimum-boundary polyhedron, and it gives no analogue of independent generators.  These stronger geometric questions are separated in Section~\ref{sec:open}.
\end{remark}

\begin{proposition}[Exact Picard examples]\label{prop:six-examples}
The reconstruction of Theorem~\ref{thm:farey-reconstruction} agrees with independently obtained Picard-cell actions in the following cases:
\begin{enumerate}[label=(\roman*),leftmargin=2.2em]
\item \(G\), with one octahedral occurrence and \(J=K\);
\item \(\Gamma_0(1+i)\), with one occurrence and \(|J|=4\);
\item Lee's two torsion-free index-twelve groups and Hockman's torsion index-twelve group, each with one occurrence and \(J=1\);
\item \(\Gamma_0(3)\), with two occurrences \(A,B\), local groups of orders two and three, and fibre sizes six and four;
\item the level-\((2+i)\) family \(\Gamma_0(2+i),\Gamma_1(2+i),\Gamma(2+i)\), of indices \(6,12,60\).  Their octahedral orbit sizes are respectively
\[
6,\qquad 12,\qquad 12+12+12+12+12,
\]
and their local groups have orders \(2\), \(1\), and \(1,1,1,1,1\).
\end{enumerate}
For \(\Gamma_0(3)\), the reconstructed \(S,R,X,Y\)-permutations agree entry-by-entry with the direct congruence action on \(\PP^1(\mathbf F_9)\).  For the level-\((2+i)\) family all \(312=4(6+12+60)\) reconstructed transitions agree with the direct actions obtained from reduction modulo \(2+i\).
\end{proposition}
\begin{proof}
The first five examples compare the fixed twelve-cell subdivision with independently obtained coset actions.  In the full group, \(J=K\), so the fibre has one cell.  For \(\Gamma_0(1+i)\), the order-four local group gives three cells.  In the three index-twelve octahedral examples, \(J=1\), so all twelve cells occur; the three groups are distinguished by their ordered boundary identifications.

For \(\Gamma_0(3)\), the prime \(3\) is inert and the left-coset set \(\Gamma_0(3)\backslash G\), equipped with its right \(G\)-action, identifies with \(\PP^1(\mathbf F_9)\).  The \(K\)-action has two orbits of sizes six and four.  Their stabilizers are the displayed groups \(J_A,J_B\).  Expanding connected transversals gives ten cells and eighteen internal face pairs; the remaining twenty-two pairs are exposed boundary identifications.  The forty resulting facet pairs are involutive with opposite coorientation, and the four permutations coincide literally with reduction of \(S,R,X,Y\) on all ten points.  Appendix~\ref{app:gamma03} records the face-class table and the permutations.

For the level-\((2+i)\) family, reduction identifies \(\ZZ[i]/(2+i)\cong\mathbf F_5\) with \(i\mapsto3\).  The three subgroups correspond to the Borel, upper-unipotent and trivial subgroups of \(\PSL_2(\mathbf F_5)\).  Appendix~\ref{app:level5} records the octahedral representatives, local groups and complete face-incidence table.  Expanding these data through Proposition~\ref{prop:twelve-cells} gives \(6,12,60\) Picard cells, and direct finite-field evaluation checks all \(4(6+12+60)=312\) transitions.

The theorem-level ingredients in this proposition are Lemma~\ref{lem:octahedral-fibres}, Proposition~\ref{prop:twelve-cells} and Theorem~\ref{thm:farey-reconstruction}.  The assertions that the displayed examples agree with independently constructed congruence or coset actions are exact finite verifications; the corresponding tables are recorded in Appendix~\ref{app:examples}.
\end{proof}

The nontrivial two-occurrence example is useful enough to display here.  For
$\Gamma_0(3)$ the local groups have orders two and three, so the two fibres
contain six and four Picard sectors.  The two inter-occurrence face classes
shown below are representatives of the exact face data recorded in
Appendix~\ref{app:gamma03}.

\begin{figure}[htbp]
\centering
\includegraphics[width=.76\textwidth]{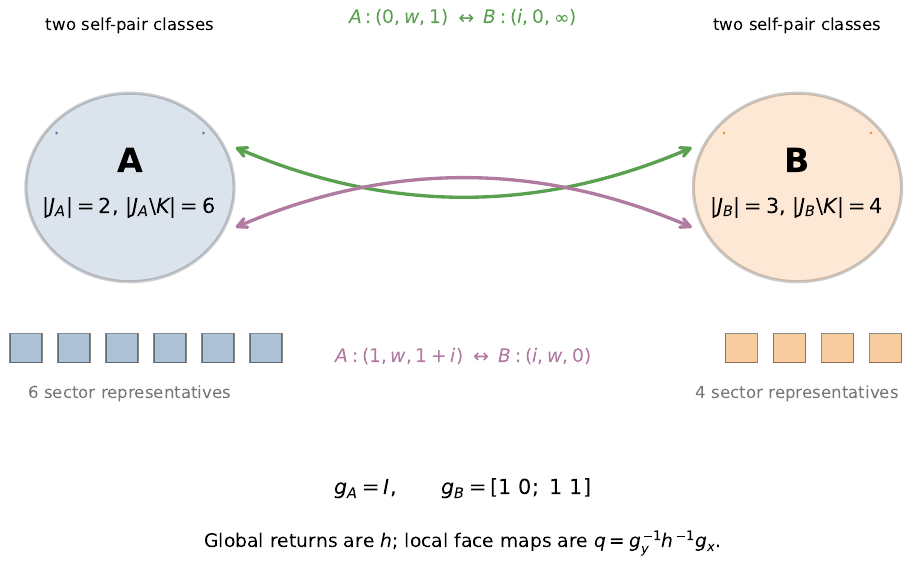}
\caption{The two octahedral occurrences for $\Gamma_0(3)$.  The numbers six
and four are the sector-fibre sizes $|J_A\backslash K|$ and
$|J_B\backslash K|$, not numbers of exposed faces or complete octahedra.
The exact normalized local face maps are listed in
Appendix~\ref{app:gamma03}.}
\label{fig:gamma03-two-octahedra}
\end{figure}

\section{\texorpdfstring{The level-$(2+i)$ family: a running example}{The level-(2+i) family: a running example}}\label{sec:level-five-family}
This section is the paper's principal worked example rather than a source of a general theorem.  Its role is to bridge the abstract definition and reconstruction of Section~\ref{sec:picard-farey} with the invariant-extraction results of Section~\ref{sec:invariants}.  By keeping one Gaussian prime fixed, we can compare three subgroups in exactly the same octahedral normalization and see separately the effects of local isotropy, torsion-freeness, normal covering, multiple octahedral occurrences, cusp lattices and modular shadow.  The calculations also provide independent finite-quotient checks of Theorem~A and supply the examples used later in Theorems~D and~E.
The reduction modulo $\mathfrak p$ and the resulting finite coset actions are standard finite-group calculations; their role here is to test the marked-octahedral reconstruction in a common normalization and to provide the concrete input for the cusp, level and modular-shadow comparisons below.

We examine the prime Gaussian ideal
\[
   \mathfrak p=(2+i),\qquad N\mathfrak p=5.
\]
Reduction identifies \(\ZZ[i]/\mathfrak p\simeq\mathbf F_5\), with \(i\mapsto3\), and \(|\PSL_2(\mathbf F_5)|=60\).  The three standard congruence subgroups give a small but representative family:
\[
\Gamma(\mathfrak p)\subset\Gamma_1(\mathfrak p)\subset\Gamma_0(\mathfrak p)\subset G.
\]

Throughout the level examples we use the following projective congruence
notation.  For a nonzero ideal $\mathfrak a\subset\ZZ[i]$, let
$\Gamma_0^{\mathrm{Pic}}(\mathfrak a)$,
$\Gamma_1^{\mathrm{Pic}}(\mathfrak a)$ and
$\Gamma^{\mathrm{Pic}}(\mathfrak a)$ denote the images in
$G=\PSL_2(\ZZ[i])$ of
\[
\begin{aligned}
 \widetilde\Gamma_0(\mathfrak a)
   &=\left\{\begin{pmatrix}a&b\\c&d\end{pmatrix}\in\mathrm{SL}_2(\ZZ[i]):
                    c\in\mathfrak a\right\},\\
 \widetilde\Gamma_1(\mathfrak a)
   &=\left\{\begin{pmatrix}a&b\\c&d\end{pmatrix}\in\mathrm{SL}_2(\ZZ[i]):
                    c\in\mathfrak a,\ a\equiv d\equiv1\pmod{\mathfrak a}
             \right\},\\
 \widetilde\Gamma(\mathfrak a)
   &=\left\{A\in\mathrm{SL}_2(\ZZ[i]):A\equiv I\pmod{\mathfrak a}\right\}.
\end{aligned}
\]
For the last two groups, a determinant-one representative belongs to the
projective image precisely when one common sign
$\epsilon\in\{\pm1\}$ makes $a,d\equiv\epsilon$ or
$A\equiv\epsilon I$, respectively.  The classical groups
$\Gamma_\bullet^{\mathrm{cl}}(n)\leq\PSL_2(\ZZ)$ use the same projective
convention.

They show, in one level, how local octahedral isotropy, torsion-free one-octahedron domains, normal principal subgroups, multiple octahedral occurrences and cusp lattices appear in the symbol.  Section~\ref{sec:invariants} gives these examples an additional arithmetic interpretation: their intersections with the embedded modular group are the classical level-$5$ groups, and the complete right actions of the modular generators agree with the corresponding classical coset actions.

\subsection{Reading the exact octahedral data}
Three finite descriptions will be used below.  The occurrence set
$H\backslash G/K$ records only the octahedral orbits.  A marked Picard Farey
symbol also records the normalized local groups $J_x$, the Gaussian-rational
placements and the ordered, cooriented face identifications.  From these
verified transitions Proposition~\ref{prop:finite-face-covering} determines
the classical finite-face quotient $\mathscr X_H$ with its stabilizers,
incidence maps and transports.

These descriptions serve different purposes.  The occurrence data are often
enough to see the size and local isotropy of an example; the marked
three-dimensional face identifications give Poincar\'e ridge and peripheral
data; the finite-face quotient supplies the classical complex-of-groups
presentation.  The Gaussian ideal edge--face relations of Result~C form a
separate relation system and do not require materializing all
$[G:H]$ finite-face squares.

\subsection{The same octahedron with and without local isotropy}
Write the eight ordered faces of the standard Farey octahedron as
\[
\begin{array}{llll}
A=(0,1,\infty),&B=(1,1+i,\infty),&C=(1+i,i,\infty),&D=(i,0,\infty),\\
E=(0,w,1),&F=(1,w,1+i),&G=(1+i,w,i),&H=(i,w,0),
\end{array}
\]
where \(w=(1+i)/2\).  For \(\Gamma_0(\mathfrak p)\) there is one octahedral occurrence and
\[
J=K\cap\Gamma_0(\mathfrak p)=\langle c^2a\rangle\cong C_2.
\]
Thus the fibre \(J\backslash K\) has six Picard cells.  The local involution identifies the eight faces in four classes
\[
\{A,C\},\quad\{B,D\},\quad\{E,G\},\quad\{F,H\}.
\tag{6.1}\label{eq:g0-faceclasses}
\]
This is the simplest example in which the local group is visibly part of the Farey symbol rather than an auxiliary correction: without \(J\), one would incorrectly expand the single octahedron to twelve cells.

For \(\Gamma_1(\mathfrak p)\) the underlying Gaussian rational octahedron is unchanged but \(J=1\).  All twelve Picard cells occur and the eight faces are individually visible.  The face pairing is
\[
A\leftrightarrow C,\qquad B\leftrightarrow D,\qquad
E\leftrightarrow G,\qquad F\leftrightarrow H.
\tag{6.2}\label{eq:g1-facepairs}
\]
The exact edge-cycle calculation has no shortened elliptic cycle, so this subgroup is torsion-free.  The contrast between \eqref{eq:g0-faceclasses} and \eqref{eq:g1-facepairs} shows that the decoration makes the elliptic structure visible before a presentation is written down; we reserve Kulkarni's terms ``even'' and ``odd'' for the literal modular specialization recorded in Corollary~\ref{cor:modular-specialization}.

The two equations above contain the exact distinction needed later: the first records local face orbits under $J$, while the second records genuine face pairings for the torsion-free subgroup.

\subsection{The principal subgroup and the double complete graph}
For the principal group \(\Gamma(\mathfrak p)\), all local groups are trivial and there are five octahedral occurrences.  Convenient representatives are
\[
I,\quad S,\quad R,\quad
\begin{pmatrix}1&0\\1&1\end{pmatrix},\quad
\begin{pmatrix}1&0\\-i&1\end{pmatrix}.
\]
Each contributes twelve Picard cells, giving the index \(60\).  The complete ordered face table has forty directed entries, hence twenty unoriented face adjacencies.  A useful simplification appears when only the incidence of octahedral occurrences is retained: every pair of distinct occurrences is joined by exactly two face classes.  The octahedral adjacency multigraph is therefore
\[
                         2K_5.
\tag{6.3}\label{eq:2k5}
\]
The full Farey symbol is not merely this graph---the Gaussian rational faces and their ordered identifications are still required---but \eqref{eq:2k5} is an effective first picture of a sixty-cell subgroup.

The complete adjacency table is recorded in Appendix~\ref{app:level5}; it is
that table, rather than a generic graph drawing, that fixes the twenty
unoriented face classes.  A concrete spanning-tree development of these five
octahedra is used in Section~\ref{sec:torsionfree} to illustrate
Theorem~\ref{thm:torsionfree-farey-domain}.

\subsection{Cusps and torsion in the family}
The three examples also illustrate how geometric information is read from the symbol.  The cusp numbers are respectively \(2,2,6\).  For both \(\Gamma_0(\mathfrak p)\) and \(\Gamma_1(\mathfrak p)\), representatives may be chosen as \(\infty,0\), with translation lattices
\[
\ZZ[i],\qquad (2+i)\ZZ[i].
\]
For the principal subgroup there are six cusp classes represented by
\[
\infty,0,i,1,-i,-1,
\]
and the translation lattice at every cusp is \((2+i)\ZZ[i]\).  In the torsion-free cases this immediately verifies the identity of Corollary~\ref{cor:cusp-lattice-sum}:
\[
2(1+5)=12
\quad\text{for }\Gamma_1(2+i),
\qquad
6\cdot2\cdot5=60
\quad\text{for }\Gamma(2+i).
\]
The edge-cycle calculation distinguishes the first member of the family: \(\Gamma_0(2+i)\) has order-two torsion, while \(\Gamma_1(2+i)\) and \(\Gamma(2+i)\) are torsion-free.  Proposition~\ref{prop:level-from-symbol} will show that the displayed cusp lattices recover the standard level $(2+i)$ for all three groups, so in these examples the same data record both the Euclidean cusp geometry and the ideal-valued congruence level.

The congruence condition is used here only to make the finite actions and
arithmetic invariants explicit.  The reconstruction and membership
constructions of Section~\ref{sec:picard-farey} apply to arbitrary
finite-index subgroups with the stated effective input; congruence is not a
hypothesis of the Farey-symbol theory.

Fine and Newman classify normal subgroups of the Picard group of index
less than sixty and exhibit normal noncongruence examples
\cite[pp.~769--786]{FineNewman1987}.  We do not use these groups as worked
examples here.

\section{Geometric and arithmetic information determined by a Picard Farey symbol}\label{sec:invariants}
The symbol determines several related finite objects, and they should not be
conflated.  Here the \emph{finite-face quotient} means the typed quotient of
$\mathscr X$ together with its local stabilizers, incidence injections and
transport or twisting data; its underlying square complex alone is not the
object used for presentations.  Proposition~\ref{prop:finite-face-covering}
constructs this quotient from the verified symbol.  The marked octahedral
face identifications retain the three-dimensional Poincar\'e data, while the
coarse ideal one- and two-skeleton supplies the separate Steinberg
edge--face presentation of Section~\ref{sec:modular-symbols}.

Index, covolume and cusp equivalence classes are already visible from the
octahedral occurrences, local groups and paired Gaussian-rational faces.
Peripheral translation elements determine the full cusp lattices; these in
turn determine the cusp amplitudes and standard ideal level.  The reconstructed
pointed right $G$-set decides congruence and determines both the modular-group
shadow and Fine's quotient graph of groups.  The classical finite-face
complex also gives the standard maximal-tree presentation described below,
but rank-two cusp groups rule out Rademacher independence.  The cusp-amplitude
formalism and the Picard amalgam are due to Mason--Schweizer and Fine,
respectively; the purpose here is to state exactly which of these established
structures are recovered from the finite symbol.

\subsection{Index and volume}
\begin{proposition}[Index and covolume]\label{prop:index-volume}
Let \(\Sigma\) be a Picard Farey symbol for \(H\leq G\), with normalized local groups \(J_x\leq K\).  Then
\[
[G:H]=\sum_{x\in X}[K:J_x],
\qquad
\operatorname{vol}(H\backslash\HH^3)
=\sum_{x\in X}\frac{\operatorname{vol}(\OO)}{|J_x|}
=\frac{[G:H]}{12}\operatorname{vol}(\OO).
\label{eq:index-volume}
\]
In particular \(\operatorname{vol}(G\backslash\HH^3)=\operatorname{vol}(\OO)/12\).
\end{proposition}
\begin{proof}
The index formula is Lemma~\ref{lem:octahedral-fibres}.  The quotient of the interior of an octahedral occurrence by its local group has orbifold volume \(\operatorname{vol}(\OO)/|J_x|\), and \(|K|=12\).
\end{proof}

If \(m=|X|\) and \(n=[G:H]\), then
\[
\left\lceil\frac n{12}\right\rceil\leq m\leq n.
\]
This is only a bound on the number of octahedral occurrences; the amount of face-pairing information may vary among symbols.

\subsection{Cusps from Gaussian rational vertices}
Let \(V(\Sigma)\) be the disjoint union of the six rational vertices of the octahedra occurring in \(\Sigma\), with the equivalence relation generated by the local groups and the paired faces.

\begin{proposition}[Cusps]\label{prop:cusps-from-symbol}
The equivalence classes in \(V(\Sigma)\) are naturally the cusp classes
\[
H\backslash\PP^1(\QQ(i)).
\]
Thus the cusp set is obtained directly from the Gaussian-rational vertices, local actions and octahedral face identifications; the twelve-cell expansion is not needed merely to determine the cusp classes.
\end{proposition}
\begin{proof}
Two ideal vertices represent the same cusp precisely when they are related by a sequence of local octahedral identifications and boundary face-pairing transformations.  The reconstruction theorem identifies the resulting relation with the action of \(H\) on \(\PP^1(\QQ(i))\).
\end{proof}

\begin{definition}[Cusp covering degree]\label{def:cusp-degree}
Let \(C\) be a cusp represented by \(g\infty\).  Its \emph{cusp covering degree} is
\[
w(C)=[G_\infty:g^{-1}Hg\cap G_\infty],
\qquad G_\infty=\Stab_G(\infty).
\label{eq:cusp-degree}
\]
For the Picard group one has explicitly
\[
G_\infty=
\left\{
\left[
\begin{pmatrix}u&b\\0&u^{-1}\end{pmatrix}
\right]:
 u\in\{\pm1,\pm i\},\ b\in\ZZ[i]
\right\}
\cong\ZZ[i]\rtimes C_2.
\]
The corresponding affine action is \(z\mapsto u^2z+ub\).  The translation subgroup is \(z\mapsto z+\lambda\), \(\lambda\in\ZZ[i]\), and the nontrivial element of the finite factor may be represented by \(\left[\begin{smallmatrix}i&0\\0&-i\end{smallmatrix}\right]\), acting by \(z\mapsto-z\) and hence by \(\lambda\mapsto-\lambda\) on the translation lattice.
\end{definition}

\begin{proposition}[Cusp covering degrees]\label{prop:cusp-degree-sum}
For every finite-index subgroup \(H\leq G\),
\[
\sum_{C\in H\backslash\PP^1(\QQ(i))}w(C)=[G:H].
\label{eq:cusp-sum}
\]
The cusp subgroup \(g^{-1}Hg\cap G_\infty\), and hence \(w(C)\), is determined by the elements of the cusp stabilizer in the reconstructed finite cell structure.
\end{proposition}
\begin{proof}
The cusps are the double cosets \(H\backslash G/G_\infty\).  The fibre above \(HgG_\infty\) in \(H\backslash G\to H\backslash G/G_\infty\) is
\[
(g^{-1}Hg\cap G_\infty)\backslash G_\infty,
\]
so summing its cardinalities gives the formula.  The peripheral elements obtained around the corresponding cusp link generate the cusp stabilizer.
\end{proof}

\begin{corollary}[Translation lattices in the torsion-free case]\label{cor:cusp-lattice-sum}
If \(H\) is torsion-free, its cusp stabilizer at \(C\) is a rank-two translation lattice \(L_C\subset\ZZ[i]\), and
\[
w(C)=2[\ZZ[i]:L_C],
\qquad
\sum_C[\ZZ[i]:L_C]=\frac{[G:H]}2.
\label{eq:lattice-sum}
\]
The symbol determines \(L_C\) up to conjugation in the ambient cusp group and hence determines the Euclidean similarity class of the torus cusp cross-section.
\end{corollary}
\begin{proof}
Torsion-freeness removes the order-two rotational part of \(G_\infty\), giving the factor two in the index.  To recover \(L_C\), collect the translation parts of all peripheral translation elements; these generate the full lattice.  A \(\ZZ\)-basis is then obtained by ordinary lattice reduction, for example Hermite or Smith normal form.  Two independent peripheral translations need not by themselves form a basis of \(L_C\).
\end{proof}

For the torsion-free subgroup $H=\Gamma_1^{\mathrm{Pic}}(2+i)$ there are two
Picard cusp classes, represented by $\infty$ and $0$.  In the standard chart
at $\infty$ and the $S$-normalized chart at $0$, the translation lattices are
\[
L_\infty=\ZZ[i]=\langle1,i\rangle_{\ZZ},\qquad
L_0=(2+i)\ZZ[i]=\langle2+i,-1+2i\rangle_{\ZZ}.
\]
Their indices in $\ZZ[i]$ are $1$ and $5$, so
$2(1+5)=12=[G:H]$, as required by Corollary~\ref{cor:cusp-lattice-sum}.

\begin{figure}[htbp]
\centering
\includegraphics[width=.80\textwidth]{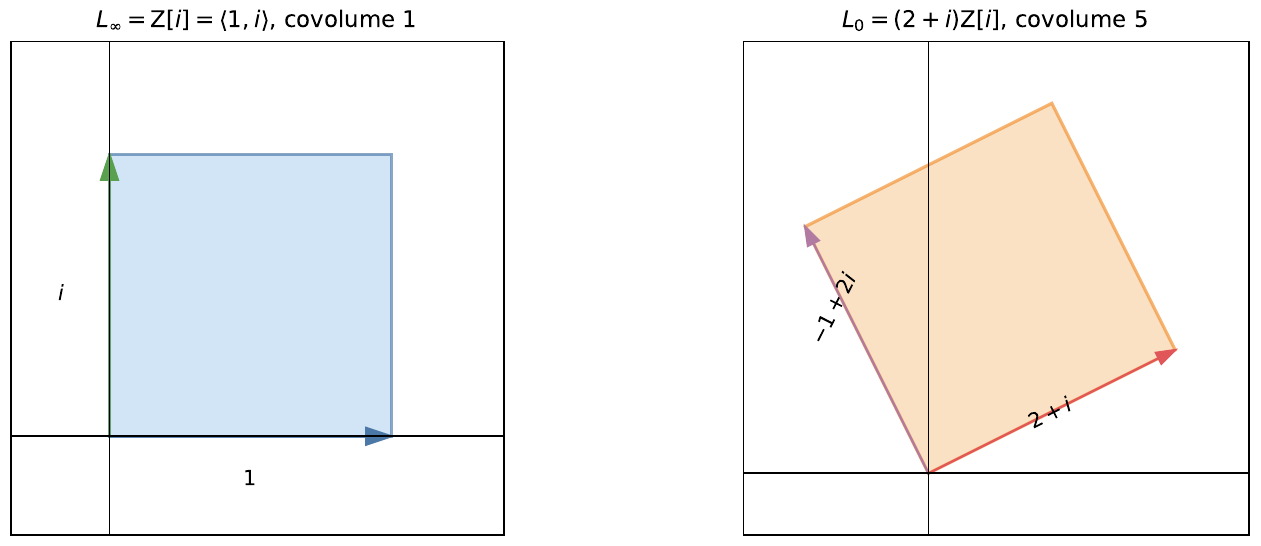}
\caption{Normalized cusp translation lattices for
$\Gamma_1^{\mathrm{Pic}}(2+i)$.  The two displayed bases have covolumes
$1$ and $5$ in the chosen Euclidean cusp coordinates.}
\label{fig:level5-cusp-links}
\end{figure}

\begin{example}[\(\Gamma_0(1+i)\)]\label{ex:gamma0pi-invariants}
The symbol has one octahedral occurrence with \(|J|=4\), hence index three.  Its rational vertices fall into two cusp classes.  The two translational lattices have indices one and two in \(\ZZ[i]\), and in this orbifold example both cusps retain the order-two rotational part.  Their cusp covering degrees are therefore one and two, whose sum is three.
\end{example}

\subsection{Ideal level from cusp lattices}\label{subsec:level-shadow}
We next pass from the cusp lattices already reconstructed above to the standard ideal-valued level.  Let $\widetilde H\leq\mathrm{SL}_2(\ZZ[i])$ be the full inverse image of $H$.  If $C$ is represented by $g\infty$, set
\[
 L_C=\left\{z\in\ZZ[i]:
 g\begin{pmatrix}1&z\\0&1\end{pmatrix}g^{-1}\in\widetilde H\right\}.
\]
This is the translation lattice already recovered from peripheral translation elements.  In Mason--Schweizer's terminology it is the quasi-amplitude \cite[\S3]{MasonSchweizer2012}.  The largest Gaussian ideal contained in $L_C$ is
\[
 \mathfrak c_C=L_C\cap i^{-1}L_C
   =\{z\in L_C:iz\in L_C\}.
\]
Indeed, the right-hand side is an additive subgroup closed under multiplication by $i$, hence a Gaussian ideal, and every ideal contained in $L_C$ lies in it.

\begin{proposition}[Level from a Picard Farey symbol]\label{prop:level-from-symbol}
Let $\Sigma$ be a marked Picard Farey symbol for $H\leq G$.  Then $\Sigma$ determines every cusp lattice $L_C$, every cusp amplitude $\mathfrak c_C$, the quasi-level $\bigcap_C L_C$, and the standard level
\[
 \mathfrak l(H)=\bigcap_C\mathfrak c_C.
\]
Equivalently, $\mathfrak l(H)$ is the largest Gaussian ideal contained in $\bigcap_C L_C$.
\end{proposition}
\begin{proof}
The symbol determines the cusp classes and the peripheral translation elements by Proposition~\ref{prop:cusps-from-symbol} and Proposition~\ref{prop:cusp-degree-sum}; their translation parts generate the full lattices $L_C$.  The displayed ideal-core operation gives $\mathfrak c_C$.  Mason--Schweizer's cusp-amplitude description of level then gives
\[
 \mathfrak l(H)=\bigcap_C\mathfrak c_C
\]
\cite[Definition in \S2 and Lemma~2.3]{MasonSchweizer2012}.
\end{proof}

\begin{remark}\label{rem:lattice-index-not-level}
Translation indices alone do not determine the amplitudes.  For example,
\[
 (2+i)\ZZ[i]
 \quad\text{and}\quad
 \ZZ+5i\ZZ
\]
both have index five in $\ZZ[i]$, but their largest contained ideals are $(2+i)$ and $(5)$, respectively.  Thus the full Gaussian cusp lattice, not merely its covolume, is essential.
\end{remark}

\subsection{Congruence recognition from the finite action}\label{subsec:congruence-recognition}
\begin{proposition}[Congruence recognition from the finite action]\label{prop:congruence-recognition}
Let $\Sigma$ reconstruct $H$ and let $\mathfrak l=\mathfrak l(H)\ne0$.  Then
\[
 H\text{ is congruence}
 \quad\Longleftrightarrow\quad
 \Gamma^{\mathrm{Pic}}(\mathfrak l)\leq H
 \quad\Longleftrightarrow\quad
 H\backslash G\text{ factors through reduction modulo }\mathfrak l.
\]
Consequently congruence is decidable from the symbol by a finite calculation.
\end{proposition}
\begin{proof}
The first equivalence is Mason--Schweizer's congruence criterion \cite[Corollary~1.4]{MasonSchweizer2012}, applied to the full inverse image of $H$ in $\mathrm{SL}_2(\ZZ[i])$.  Since $\Gamma^{\mathrm{Pic}}(\mathfrak l)$ is normal in $G$, its containment in $H$ is equivalent to its acting trivially on the transitive right $G$-set $H\backslash G$.  This is exactly factorization through the finite reduction image.  Using fixed generators of $G$, one can test the factorization by comparing their reduction images with their reconstructed permutations.
\end{proof}

\begin{remark}
Mason--Schweizer's Condition~L concerns the additional equality of quasi-level and level for certain congruence subgroups \cite[Theorem~4.6]{MasonSchweizer2012}.  It is not required for Proposition~\ref{prop:level-from-symbol} or for the congruence criterion above.
\end{remark}

\subsection{Modular intersection and split-prime finite actions}\label{subsec:modular-action}
For a Gaussian ideal $\mathfrak a$, let $n(\mathfrak a)>0$ be defined by
\[
 \mathfrak a\cap\ZZ=n(\mathfrak a)\ZZ.
\]
If
\[
 \mathfrak a=(1+i)^{e_2}
 \prod_{p\equiv1(4)}\mathfrak p_p^{e_p}\overline{\mathfrak p}_p^{f_p}
 \prod_{q\equiv3(4)}(q)^{g_q},
\]
then
\[
 n(\mathfrak a)=2^{\lceil e_2/2\rceil}
 \prod_{p\equiv1(4)}p^{\max(e_p,f_p)}
 \prod_{q\equiv3(4)}q^{g_q}.
\]

\begin{proposition}[Modular intersection theorem]\label{prop:modular-intersection}
For every nonzero ideal $\mathfrak a\subset\ZZ[i]$,
\[
\begin{aligned}
 \Gamma_0^{\mathrm{Pic}}(\mathfrak a)\cap\PSL_2(\ZZ)
   &=\Gamma_0^{\mathrm{cl}}(n(\mathfrak a)),\\
 \Gamma_1^{\mathrm{Pic}}(\mathfrak a)\cap\PSL_2(\ZZ)
   &=\Gamma_1^{\mathrm{cl}}(n(\mathfrak a)),\\
 \Gamma^{\mathrm{Pic}}(\mathfrak a)\cap\PSL_2(\ZZ)
   &=\Gamma^{\mathrm{cl}}(n(\mathfrak a)).
\end{aligned}
\]
These are equalities inside the fixed embedded modular group; they include the ramified and small-level cases.
\end{proposition}
\begin{proof}
Choose an integral determinant-one representative.  For every integer entry $x$,
\[
 x\in\mathfrak a
 \quad\Longleftrightarrow\quad
 x\in\mathfrak a\cap\ZZ=n(\mathfrak a)\ZZ.
\]
Apply this entrywise to $c$, to $a-\epsilon$ and $d-\epsilon$, or to every entry of $A-\epsilon I$, respectively.  The same global sign $\epsilon$ is used before and after restriction.  Characteristic two merely identifies the two signs on both sides simultaneously.
\end{proof}

Put $M=\PSL_2(\ZZ)$.  For $H_{\mathrm{Pic}}\leq G$ and $H_M=H_{\mathrm{Pic}}\cap M$, the natural map
\[
 H_M\backslash M\longrightarrow H_{\mathrm{Pic}}\backslash G,
 \qquad H_Mm\longmapsto H_{\mathrm{Pic}}m,
\]
is always injective and is surjective exactly when $G=H_{\mathrm{Pic}}M$.

\begin{proposition}[Split-prime finite-action shadow]\label{prop:split-prime-shadow}
Let $p\equiv1\pmod4$, let $\mathfrak p\mid p$ have norm $p$, and let
\[
 H_{\mathrm{Pic}}=\Gamma_\bullet^{\mathrm{Pic}}(\mathfrak p),
 \qquad
 H_M=\Gamma_\bullet^{\mathrm{cl}}(p),
\]
for $\bullet=0,1,\varnothing$.  Then the natural coset map $H_M\backslash M\to H_{\mathrm{Pic}}\backslash G$ is a bijection of right $M$-sets.  Under reduction both sides are the same quotient
\[
 B_\bullet\backslash\PSL_2(\mathbf F_p),
\]
where $B_0$ is a Borel subgroup, $B_1$ is its upper-unipotent subgroup, and $B_\varnothing=1$.  Hence
\[
 [G:\Gamma_0^{\mathrm{Pic}}(\mathfrak p)]=p+1,
 \qquad
 [G:\Gamma_1^{\mathrm{Pic}}(\mathfrak p)]=\frac{p^2-1}{2},
 \qquad
 [G:\Gamma^{\mathrm{Pic}}(\mathfrak p)]=\frac{p(p^2-1)}2.
\]
\end{proposition}
\begin{proof}
The natural map $\ZZ\to\ZZ[i]/\mathfrak p$ identifies $\ZZ/p\ZZ$ with $\ZZ[i]/\mathfrak p$.  The reductions of the standard integer matrices $S$ and $T$ generate $\PSL_2(\mathbf F_p)$, so the embedded modular group maps onto the full Picard reduction quotient.  The three congruence groups are inverse images of the same subgroups $B_\bullet$, and the general injective coset map is therefore surjective and equivariant.  The index formulas follow from $|\PSL_2(\mathbf F_p)|=p(p^2-1)/2$.
\end{proof}

\begin{example}[The level-$5$ shadow]\label{ex:level5-shadow}
For $\mathfrak p=(2+i)$, Proposition~\ref{prop:modular-intersection} gives
\[
 \Gamma_\bullet^{\mathrm{Pic}}(\mathfrak p)\cap M
 =\Gamma_\bullet^{\mathrm{cl}}(5).
\]
The three Picard indices are $6,12,60$, and Proposition~\ref{prop:split-prime-shadow} identifies their complete right $S,T$-actions with the classical level-$5$ actions.  An exact comparison of the two generator actions checks $12+24+120=156$ transitions.  The first two Picard groups have cusp lattices $\ZZ[i]$ and $(2+i)\ZZ[i]$; the principal group has six cusps, each with lattice $(2+i)\ZZ[i]$.  Proposition~\ref{prop:level-from-symbol} therefore gives standard level $(2+i)$ in all three cases.  As a ramified control, $\Gamma_0^{\mathrm{Pic}}(1+i)$ has cusp lattices $\ZZ[i]$ and $(1+i)\ZZ[i]$ and standard level $(1+i)$.
\end{example}

\subsection{\texorpdfstring{$M$-adapted lifting and cusp widths}{M-adapted lifting and cusp widths}}\label{subsec:M-adapted-shadow}
The general geometric lifting statement used in the split-prime examples is recorded here, after the modular intersection and finite-action comparison.

\begin{proposition}[$M$-adapted lifting]\label{prop:M-adapted-lifting}
Let $H\le G$ have finite index and suppose that the natural map
\[
(H\cap M)\backslash M\longrightarrow H\backslash G,
\qquad (H\cap M)m\longmapsto Hm,
\]
is bijective.  Let $R\subset M$ be a complete set of representatives for
$(H\cap M)\backslash M$.  If
\[
Q_M=\bigcup_{m\in R}mF
\]
is a fundamental polygon for $H\cap M$, then
\[
P=\bigcup_{m\in R}mP_0
\]
is a fundamental set for $H$ and
\[
P\cap\Pi=Q_M.
\]
If the $F$-tile dual graph of $Q_M$ is connected, then the $P_0$-cell dual
graph of $P$ is connected, so $P$ is the connected cell-union Poincar\'e
polyhedron obtained from the ambient Picard tessellation.  In particular,
this applies to a Kulkarni special polygon with its standard side-connected
$F$-tile decomposition.
\end{proposition}
\begin{proof}
Bijectivity makes $R$ a complete transversal for $H\backslash G$.  Since
$P_0$ is a $G$-fundamental cell, the usual coset-union argument proves that
the $H$-translates of $P$ cover $\HH^3$ with disjoint interiors.  Every
$m\in M$ preserves $\Pi$, and therefore
\[
P\cap\Pi
=\bigcup_{m\in R}m(P_0\cap\Pi)
=\bigcup_{m\in R}mF=Q_M.
\]
The three sides of $F$ are traces of the Picard faces represented by $S$
and $X^{\pm1}$, where $X=TS$.  Hence adjacent $F$-tiles lift to Picard cells
sharing a complete face.  A connected $F$-tile dual graph therefore embeds
in the Picard-cell dual graph.  The exposed Picard faces are paired by the
usual Schreier return elements, while ridge cycles and cusp links are
inherited from the ambient tessellation.  This is the connected cell-union
construction of Proposition~\ref{prop:lee}.
\end{proof}

The reflection regions explain the two-dimensional section without changing the hypothesis of the proposition:
\[
P=\bigcup_{m\in R}(mD\cup msD),\qquad
Q_M=\bigcup_{m\in R}(m\Delta\cup ms\Delta).
\]
Each $P_0$-cell contributes two $D$-regions, and each $F$-tile contributes two small reflection triangles of opposite parity.  Retaining only $mD$ for $m\in R$ would omit the other half of every Picard cell.  The lifting statement is therefore made for $P_0$ and its modular section $F$, not for a single reflection region per modular coset.

\begin{remark}[What the lifting theorem does not say]\label{rem:M-adapted-scope}
The proposition is an existence statement for a deliberately $M$-adapted
development.  It is not a criterion for an arbitrary choice of Picard
representatives $g\in G$, whose cells need not preserve $\Pi$, and it is not
invariant under octahedral compression.  The contrast is already visible for
$\Gamma_1^{\rm Pic}(2+i)$: its standard one-octahedron domain meets $\Pi$ in
a single ideal Farey face, whereas its $M$-adapted twelve-cell development
has the level-$5$ Kulkarni polygon as section.
\end{remark}

\begin{corollary}[Split-prime $M$-adapted geometric shadow]\label{cor:split-prime-geometric-shadow}
Let $p$ split in $\ZZ[i]$, let $\mathfrak p\mid p$ have norm $p$, and let
\[
H=\Gamma_\bullet^{\rm Pic}(\mathfrak p),\qquad
H_M=H\cap M=\Gamma_\bullet^{\rm cl}(p),
\]
for $\bullet=0,1,\varnothing$.  If $Q_M$ is a side-connected Kulkarni special
polygon for $H_M$, written as a union $Q_M=\bigcup_{m\in R}mF$ of the base
modular cells, then
\[
P=\bigcup_{m\in R}mP_0
\]
is a connected Picard fundamental polyhedron for $H$ and
$P\cap\Pi=Q_M$.
\end{corollary}
\begin{proof}
Proposition~\ref{prop:split-prime-shadow} makes the natural map
$H_M\backslash M\to H\backslash G$ bijective.  Apply
Proposition~\ref{prop:M-adapted-lifting}.
\end{proof}

\begin{example}[Explicit level-$(2+i)$/level-$5$ lifts]\label{ex:level5-geometric-shadow}
For $\mathfrak p=(2+i)$ the Farey sequences
\[
0,\tfrac12,1,
\qquad
0,\tfrac13,\tfrac25,\tfrac12,1,
\]
and
\[
0,\tfrac13,\tfrac25,\tfrac12,\tfrac35,\tfrac23,1,
\tfrac32,\tfrac85,\tfrac53,2,\tfrac52,\tfrac{13}5,
\tfrac83,3,\tfrac72,\tfrac{18}5,\tfrac{11}3,4,\tfrac92,5
\]
define the Kulkarni polygons for $\Gamma_0(5)$, $\Gamma_1(5)$ and
$\Gamma(5)$.  Exact enumeration gives side-connected decompositions into
$6,12,60$ copies of $F$, respectively.  Using the same modular
representatives in $\HH^3$ gives connected Picard fundamental polyhedra for
$\Gamma_0^{\rm Pic}(2+i)$, $\Gamma_1^{\rm Pic}(2+i)$ and
$\Gamma^{\rm Pic}(2+i)$ whose sections by $\Pi$ are exactly those three
special polygons.  The polygon areas are $2\pi,4\pi,20\pi$, equal to the
cell counts times $\operatorname{area}(F)=\pi/3$.
\end{example}

\begin{proposition}[Widths from Gaussian cusp lattices]\label{prop:shadow-widths}
Let $c=m\infty\in\PP^1(\QQ)$ with $m\in M$.  The Kulkarni cusp width of
$c$ for $H\cap M$ is the positive generator $w$ of
\[
\{n\in\ZZ:mT^nm^{-1}\in H\}=w\ZZ.
\]
Equivalently, if
\[
\Lambda_{H,c}=\left\{\lambda\in\ZZ[i]:
 m\begin{pmatrix}1&\lambda\\0&1\end{pmatrix}m^{-1}\in H\right\},
\]
then
\[
\Lambda_{H,c}\cap\ZZ=w\ZZ.
\]
Different $(H\cap M)$-orbits of rational cusps may lie in the same $H$-orbit
of Gaussian cusps.
\end{proposition}
\begin{proof}
The stabilizer of $\infty$ in $M$ is $\langle T\rangle$, and conjugation by
$m$ gives the first formula.  The second is the same assertion inside the
rank-two Gaussian translation lattice.  The final statement follows because
an element of $H$ identifying two rational cusps need not lie in $M$.
\end{proof}

For the level-$5$ examples, exact enumeration gives $2,4,12$ modular cusps
with width sums $6,12,60$.  The corresponding Picard cusp counts are
$2,2,6$: for $\Gamma_1$ the modular cusps merge in two pairs, and for the
principal group they merge in six pairs.  Thus equality of the finite coset
actions does not identify the two cusp-orbit sets.

\subsection{Fine\'s amalgam and its finite quotient}\label{subsec:fine-shadow}
The Picard group also has Fine's classical amalgam decomposition
\[
G=G_1*_{M}G_2,\qquad M=\PSL_2(\ZZ),
\]
with \(G_1=S_3*_{C_3}A_4\) and \(G_2=S_3*_{C_2}(C_2\times C_2)\).  The following comparison is background for the quotient extracted from the reconstructed finite action.

\begin{remark}[Fine's structure theorem for principal congruence subgroups]\label{rem:fine-congruence-structure}
Fine's 1980 analysis gives a more precise group-theoretic comparison with the modular case.  Write
\[
G=G_1*_{M}G_2,
\qquad M=\PSL_2(\ZZ),
\]
with $G_1=S_3*_{C_3}A_4$ and $G_2=S_3*_{C_2}(C_2\times C_2)$ as displayed immediately above.  If $\Gamma(\alpha)$ is a principal congruence subgroup and $(\alpha)$ is neither $(1+i)$ nor $(2)$, Fine proves that $\Gamma(\alpha)$ is an HNN group with finitely generated free part and with base a finite tree product of finitely generated free groups; the amalgamated and associated subgroups are conjugates of subgroups of the embedded modular group $M$ \cite[Theorem~1 and its proof, pp.~1475--1477]{Fine1980}.  When $n\in\ZZ$, he also notes that $\Gamma(n)\cap M$ is the ordinary principal congruence subgroup of level $n$ in $\PSL_2(\ZZ)$ \cite[p.~1476]{Fine1980}.

This result is not used in the proof of the Picard Farey reconstruction theorem.  It provides a group-theoretic comparison, rather than an independent-generator conclusion.  Once a Picard Farey symbol has reconstructed the finite right $G$-set $H\backslash G$, it also determines the finite double-coset sets $H\backslash G/G_1$, $H\backslash G/G_2$ and $H\backslash G/M$, together with their stabilizers.  The associated Bass--Serre graph of groups therefore yields a structured presentation by vertex-group generators and stable letters.  For the generic principal congruence subgroups covered by Fine's theorem, the vertex groups are free of finite rank, but the modular edge groups and the rank-two cusp geometry prevent the resulting generators from being independent in Rademacher's sense.  Proposition~\ref{prop:modular-intersection} above strengthens Fine's integer-level observation to every Gaussian ideal, while Proposition~\ref{prop:fine-shadow} below shows that the Picard Farey symbol determines the entire finite quotient graph of groups.
\end{remark}

The same finite action also has a one-dimensional graph-of-groups shadow.  With Fine's amalgam
\[
 G=G_1*_{M}G_2,
 \qquad M=\PSL_2(\ZZ),
\]
fixed as above, the quotient is determined directly from the reconstructed right $G$-set.

\begin{proposition}[Finite Bass--Serre shadow]\label{prop:fine-shadow}
The reconstructed right $G$-set $X=H\backslash G$ determines the finite quotient of the Bass--Serre tree of Fine's amalgam.  Its two vertex sets and edge set are
\[
 X/G_1=H\backslash G/G_1,
 \qquad X/G_2=H\backslash G/G_2,
 \qquad X/M=H\backslash G/M.
\]
If a double coset is represented by $HgG_i$ or $HgM$, its normalized stabilizer is respectively $g^{-1}Hg\cap G_i$ or $g^{-1}Hg\cap M$.  The incidence maps and conjugating elements are obtained from representatives in $X$, and a maximal tree gives the standard Bass--Serre presentation of $H$.
\end{proposition}
\begin{proof}
For any subgroup $A\leq G$, the $A$-orbits in the reconstructed right $G$-set $X$ are exactly $H\backslash G/A$, and the stabilizer of $Hg$ under the right $A$-action is $g^{-1}Hg\cap A$.  Apply this for $G_1,G_2,M$ and retain the inclusions $M\hookrightarrow G_i$.
\end{proof}

\begin{remark}\label{rem:fine-shadow-comparison}
Proposition~\ref{prop:fine-shadow} and Proposition~\ref{prop:generators-from-symbol} use the same finite $G$-set but different $G$-complexes.  The Bass--Serre quotient retains the two amalgam vertex types, their modular edge groups and the stable-letter data, while it forgets octahedral faces, ridge relations and cusp links.  For the three groups at $(2+i)$ there is one Bass--Serre edge orbit, whose edge group is respectively $\Gamma_0^{\mathrm{cl}}(5)$, $\Gamma_1^{\mathrm{cl}}(5)$ or $\Gamma^{\mathrm{cl}}(5)$.  Fine's generic principal-level theorem supplies a sharper structural description of the vertex groups \cite[Theorem~1 and pp.~1475--1477]{Fine1980}.
\end{remark}

\begin{remark}[Relation with Fine's principal congruence decomposition]\label{rem:fine-kulkarni6}
For the principal congruence groups covered by Fine's theorem, the same
pointed finite $G$-set also determines the quotient of the Bass--Serre tree of
\[
G=G_1*_{\PSL_2(\ZZ)}G_2.
\]
Its two vertex sets and edge set are
\[
H\backslash G/G_1,\qquad H\backslash G/G_2,\qquad
H\backslash G/\PSL_2(\ZZ),
\]
with stabilizers obtained by intersection with the corresponding conjugates.
A maximal tree gives the usual Bass--Serre generators and stable letters, and
Fine's theorem supplies the sharper structure of the vertex groups away from
the exceptional principal levels $(1+i)$ and $(2)$
\cite[Theorem~1 and proof]{Fine1980}.  Thus the symbol determines both the
standard finite-face quotient and the Bass--Serre quotient to which the
classical presentation procedures apply; this is not an analogue of
Kulkarni's independent-generator theorem.
\end{remark}

\subsection{The classical finite-face presentation determined by a symbol}\label{subsec:generators-from-symbol}
There are two standard geometric presentation languages in the background.
A three-dimensional Poincar\'e presentation uses boundary face-pairing
generators and ridge-cycle relations.  The classical finite-face complex
$\mathscr X$ instead gives a complex-of-groups presentation with local vertex
groups, edge inclusions, non-tree transport letters and square attaching
relations.  When both constructions apply they present the same subgroup,
but their individual ridge and square relations are not canonically paired
one by one.  The contribution here is that the marked octahedral symbol
determines the finite input to the latter classical construction.

\begin{proposition}[Finite-face presentation from a Picard Farey symbol]\label{prop:generators-from-symbol}
Let $\Sigma$ be a marked Picard Farey symbol for a finite-index subgroup
$H\leq G$.  Proposition~\ref{prop:finite-face-covering} determines the full
stabilizer-labelled quotient $\mathscr X_H$.  Choose a maximal tree $T$ in
its quotient one-skeleton and compatible lifts of its vertices.  Then the
standard complex-of-groups presentation of $H$ has
\begin{enumerate}[label=(\roman*),leftmargin=2.2em]
\item the nontrivial local vertex stabilizers as local generators;
\item one transport generator for each oriented quotient edge outside $T$;
\item the local group relations, the conjugacy relations identifying the two
images of every edge stabilizer, and the attaching relations of the quotient
two-cells.
\end{enumerate}
All groups, inclusions and transport elements are determined by $\Sigma$ via
the finite-face covering.
\end{proposition}
\begin{proof}
Proposition~\ref{prop:finite-face-covering} supplies the developable finite
complex of groups $\mathscr X_H$ with its local groups, incidences and
transports.  Applying Proposition~\ref{prop:max-tree} to this classical
finite-face quotient gives exactly the displayed presentation.  This is the
classical finite-face presentation applied to the quotient supplied by the symbol.
\end{proof}

\begin{corollary}[Torsion-free case]\label{cor:torsionfree-generators}
If $H$ is torsion-free, all finite local stabilizers in
Proposition~\ref{prop:generators-from-symbol} are trivial.  After choosing a
maximal tree, $H$ is generated by the non-tree transport elements, with
relations supplied by the square attaching maps.  In particular no
independence claim is made: the quotient remains two-dimensional and the
rank-two peripheral structure forced by the cusps is not removed.
\end{corollary}

\subsection{Singular strata, presentations and homology}
The octahedral groups $J_x$ do not by themselves contain all torsion
information.  The finite-face covering determines the finite stabilizers and
typed quotient of $\mathscr X$, while Poincar\'e ridge cycles and peripheral
translation lattices come from the marked octahedral face identifications and
the reconstructed three-dimensional quotient.

\begin{proposition}[Further information from the finite reconstruction]\label{prop:further-invariants}
A Picard Farey symbol determines:
\begin{enumerate}[label=(\roman*),leftmargin=2.2em]
\item the finite stabilizer-labelled finite-face quotient and its singular
vertex and edge strata;
\item the Poincar\'e ridge cycles and the peripheral translation lattices;
\item a finite presentation of $H$, either from the classical finite-face
complex of groups or, for a verified Poincar\'e polyhedron, from face pairings
and ridge cycles, together with the peripheral subgroups;
\item the abelianization $H^{\mathrm{ab}}$;
\item the ordinary cellular homology of the underlying finite quotient cell
complex.
\end{enumerate}
The last item is quotient-space homology.  For a group with torsion it is not,
in general, the same as $H_*(H,\ZZ)$; group or equivariant homology must retain
the cell stabilizers, for example through an equivariant cellular spectral
sequence or a suitable resolution.  When $H$ is torsion-free, a free
$H$-equivariant contractible cellular model or retract gives the usual
computation of group homology after the appropriate treatment of the cusps.
\end{proposition}
\begin{proof}
Proposition~\ref{prop:finite-face-covering} gives the stabilizer-labelled
finite-face quotient and its incidence data.  Theorem~\ref{thm:farey-reconstruction}
recovers the marked three-dimensional face identifications, from which the
Poincar\'e ridge cycles and peripheral translation elements are obtained.
The two standard presentation procedures then give presentations when their
respective hypotheses are satisfied; their relation lists need not coincide
term by term.  Abelianization follows algebraically.  The final distinction
is the usual one between the cellular chains of a quotient space and an
equivariant resolution when stabilizers are present.
\end{proof}

\paragraph{A one-octahedron topological test case.}
For Lee's Whitehead class, Appendix~\ref{app:octahedra} carries this extraction
through completely.  The marked octahedron gives three edge cycles, two torus
cusps with translation lattices
$\langle 1+i,4\rangle_{\ZZ}$ and $\langle 1,2i\rangle_{\ZZ}$, and
$H_1\cong\ZZ^2$; the primitive peripheral meridians form the homological
basis and the primitive longitudes vanish.  The identification of the
resulting manifold with the Whitehead-link complement is Lee's, rather than
extra structure built into the Picard Farey symbol.

The analogy with Kulkarni is therefore structural rather than direct.  Index, covolume and cusp classes are visible at the compact octahedral level.  Peripheral lattices and Poincar\'e ridge data use the marked three-dimensional face identifications, while the classical finite-face quotient supplies its own stabilizer and square-attaching data.  Homological calculations must use the chain theory appropriate to whether stabilizers are present.

\section{Torsion-free Farey octahedral domains}\label{sec:torsionfree}
For torsion-free subgroups the finite octahedral stabilizers disappear, so the quotient of the dual Farey-octahedron graph behaves as an ordinary graph covering.  This gives a useful connected fundamental domain made from complete octahedra.  The graph-covering and spanning-tree mechanism is standard; the Picard-specific consequence proved here is the exact $[G:H]/12$ octahedron count and the resulting exposed-face bound for the Gaussian Farey tessellation.  It is important to distinguish this statement from the stronger assertion that a chosen union, with chosen side pairings, satisfies every hypothesis of a particular Poincar\'e polyhedron theorem.

\begin{definition}[Farey octahedral fundamental domain]\label{def:farey-domain}
For \(H\leq G\), a \emph{Farey octahedral fundamental domain} is a connected finite union of complete Gaussian Farey octahedra whose open octahedra represent the \(H\)-orbits of open Farey octahedra exactly once.  Equivalently, its \(H\)-translates cover \(\HH^3\) and have disjoint interiors at the octahedral level.
\end{definition}

\begin{lemma}[Free action on the dual graph]\label{lem:dual-free}
Let \(H\leq G\) be torsion-free.  Then \(H\) acts freely on the vertices of the dual Farey-octahedron graph \(\cD\) and without inversions on its edges.  Consequently \(\cD\to H\backslash\cD\) is an ordinary graph covering.
\end{lemma}
\begin{proof}
A vertex stabilizer is conjugate to a subgroup of the finite octahedral stabilizer \(K\), so its intersection with torsion-free \(H\) is trivial.  If \(h\in H\) inverted an edge, then \(h^2\) would fix its endpoints; hence \(h^2=1\), and torsion-freeness gives \(h=1\).
\end{proof}

\begin{theorem}[Torsion-free Farey octahedral domains]\label{thm:torsionfree-farey-domain}
Let \(H\leq G=\PSL_2(\ZZ[i])\) be torsion-free of finite index \(n\).  Then
\[
12\mid n,
\qquad
N:=|H\backslash G/K|=\frac n{12}.
\label{eq:torsionfree-count}
\]
There exists a Farey octahedral fundamental domain \(P_T\) consisting of exactly \(N\) complete Farey octahedra.  Every fundamental domain that is a union of complete Farey octahedra contains exactly \(N\) octahedra.

One may choose \(P_T\) so that the number of exposed triangular faces satisfies
\[
f_2^{\mathrm{exp}}(P_T)\leq 6N+2.
\label{eq:face-bound}
\]
These exposed faces occur in distinct \(H\)-paired pairs, so the construction supplies at most
\[
3N+1=\frac n4+1
\label{eq:generator-bound}
\]
side-pairing transformations.  They generate \(H\).  No assertion is made that this generating set is minimal or independent.
\end{theorem}
\begin{proof}
For every \(g\in G\), torsion-freeness gives \(g^{-1}Hg\cap K=1\).  Lemma~\ref{lem:octahedral-fibres} therefore gives \(n=12N\).

Choose a spanning tree \(T\) of the finite connected graph \(H\backslash\cD\), choose one lift of a root, and lift the tree.  By Lemma~\ref{lem:dual-free} the lift is unique and contains one vertex above each quotient vertex.  The corresponding union \(P_T\) therefore contains exactly one octahedron from every \(H\)-orbit and is connected.  Its translates cover \(\HH^3\), and distinct translates have disjoint interiors, so it is a Farey octahedral fundamental domain.

Any fundamental domain that is a union of complete Farey octahedra must contain at least one representative of every orbit and cannot contain two representatives of the same orbit, because their interiors would be related by a nontrivial element of \(H\).  Hence the number \(N\) is forced within this class.

Before the tree gluings there are \(8N\) triangular face occurrences.  Each of the \(N-1\) tree edges makes one pair internal, so
\[
f_2^{\mathrm{exp}}(P_T)\leq8N-2(N-1)=6N+2.
\]
Additional adjacencies among the selected octahedra can only reduce this number.  The setwise stabilizer in \(G\) of an ideal Farey triangle is finite: its action on the three ideal vertices injects into \(S_3\), because an element fixing three distinct points of \(\PP^1(\mathbf C)\) is the identity.  Thus a nontrivial self-pairing of an exposed triangle would give torsion in \(H\).  The exposed triangles therefore occur in distinct pairs, giving at most \(3N+1\) side-pairing transformations.  The adjacency graph of the translates of \(P_T\) is connected, and the usual side-pairing argument shows that these transformations generate \(H\).
\end{proof}

\paragraph{The principal level-$(2+i)$ example.}
For $H=\Gamma(2+i)$, Section~\ref{sec:level-five-family} gives five
octahedral occurrences and coarse incidence graph $2K_5$.  One exact lifted
spanning tree uses the four face gluings
\[
O_0(D)\sim O_1(D),\qquad
O_0(A)\sim O_2(A),\qquad
O_0(C)\sim O_3(F),\qquad
O_0(B)\sim O_4(G),
\]
all of which occur in the exact table of Appendix~\ref{app:level5}.  Before
these four gluings the five octahedra have forty triangular face occurrences;
the tree gluings make eight of them internal and therefore leave thirty-two
exposed triangles, or sixteen paired sides.  Thus $N=5$ gives
$3N+1=16$, exactly the bound in Theorem~\ref{thm:torsionfree-farey-domain}.
Cycles among the selected octahedra can lower the exposed-face count further;
the separate $30$-triangle example in Appendix~\ref{app:level5} records that
phenomenon.

\begin{corollary}[Construction of a whole-octahedron domain from subgroup membership]\label{cor:effective-octahedral-domain}
Let $H\leq G$ be torsion-free of finite index, and suppose membership in $H$ can be decided for exact matrices.  Then the procedure of Theorem~\ref{thm:constructive-symbol} constructs the finite quotient dual graph $H\backslash\mathcal D$ and hence, after a choice of spanning tree, a Farey octahedral fundamental domain for $H$.  Neither the index $[G:H]$ nor a prior fundamental polyhedron is required as input; if the construction finds $N$ octahedral occurrences, then $[G:H]=12N$.

In particular this applies to the principal congruence subgroup $\Gamma(\alpha)$ whenever $(\alpha)$ is neither $(1+i)$ nor $(2)$: membership is decided by reduction modulo $(\alpha)$ and Fine's torsion lemma shows that $\Gamma(\alpha)$ is torsion-free \cite[pp.~1475--1476]{Fine1980}.
\end{corollary}
\begin{proof}
Theorem~\ref{thm:constructive-symbol} finds the actual finite set $H\backslash G/K$ and all adjacent-octahedron identifications.  Hence it constructs the quotient dual graph.  Torsion-freeness makes the vertex stabilizers trivial, so Lemma~\ref{lem:dual-free} and the spanning-tree proof of Theorem~\ref{thm:torsionfree-farey-domain} apply verbatim.  The index formula follows from $|K|=12$.
\end{proof}

\begin{remark}[The same symbol also records the arithmetic level]\label{rem:domain-level}
For a torsion-free subgroup constructed from a membership test, the octahedral development and the arithmetic label require no separate input.  Theorem~\ref{thm:constructive-symbol} supplies the symbol; Corollary~\ref{cor:effective-octahedral-domain} develops a whole-octahedron fundamental domain; and Proposition~\ref{prop:level-from-symbol} computes the standard ideal level and the congruence test from the cusp translation lattices $L_C$ and the reconstructed finite action.  This does not improve the face bound in Theorem~\ref{thm:torsionfree-farey-domain}; it records that the same finite object controls both the chosen three-dimensional domain and its congruence level.
\end{remark}

\begin{remark}[When the domain is a Poincar\'e polyhedron]\label{rem:poincare-domain}
The spanning-tree argument constructs a connected polyhedral fundamental domain.  To use a particular version of the Poincar\'e polyhedron theorem directly on \(P_T\), one must additionally verify its hypotheses for the chosen exposed sides and pairings: the edge-cycle conditions, local finiteness and the completeness conditions at ideal vertices or cusp links.  These properties are not consequences of the graph argument alone.
\end{remark}

\begin{corollary}[The first two torsion-free indices]\label{cor:lee-low-index}
For \(n=12\), the theorem gives one complete Farey octahedron and at most four side-pairing transformations; for \(n=24\), it gives two octahedra and at most seven.  The one-octahedron index-twelve domains are displayed in Lee's Section~3 (pp.~186--188), and the doubled index-twenty-four construction in his Section~4 (pp.~188--191) \cite{Lee1984}.
\end{corollary}

The theorem fixes the number of complete octahedra, not the boundary complexity among all choices of such domains and not the complexity of arbitrary hyperbolic fundamental polyhedra.  Those stronger minimization questions are left open in Section~\ref{sec:open}.  The standard octahedron is shown in Figure~\ref{fig:intro-octahedron}; the Lee--Whitehead example in Appendix~\ref{app:octahedra} is the $N=1$ case in which the same geometric octahedron, with its exact face pairings, also recovers the two peripheral lattices and $H_1\cong\ZZ^2$.

\begin{remark}[Three notions of geometric economy]\label{rem:three-geometric-objectives}
The whole-octahedron theorem controls only one class of developments.  The
examples in Appendix~\ref{app:examples} and the lifting theorem of
Section~\ref{subsec:M-adapted-shadow} distinguish three questions which need
not have the same answer:
\begin{enumerate}[label=(\roman*),leftmargin=2.2em]
\item minimize exposed triangular faces while retaining complete Farey
octahedra;
\item minimize geometric faces after allowing finer cuts and legitimate
mergers;
\item choose a development adapted to a prescribed totally geodesic modular
section.
\end{enumerate}
For $\Gamma(2+i)$, cycles improve the whole-octahedron boundary count from
$32$ to $30$; for $\Gamma_0(1+i)$ a tetrahedral domain reaches the absolute
four-face lower bound; and for $\Gamma_1(2+i)$ the one-octahedron domain is
more compressed than the twelve-cell $M$-adapted domain, but only the latter
has the full $\Gamma_1(5)$ special polygon as section.  These examples are
why no single notion of ``reduced domain'' is imposed in this paper.
\end{remark}

\section{Bianchi modular symbols and the Gaussian ideal complex}\label{sec:modular-symbols}

The integral Steinberg presentation used here is established.  This section
identifies its Gaussian ideal cells with the coarse Farey-octahedral cells used
above and shows that a Picard Farey symbol supplies the finite orbit data
required by that presentation.
This gives the compatibility statement summarized as Theorem~C and isolates, in Subsection~\ref{subsec:limitations}, the genuinely separate problem of reducing Hecke images back to a finite Gaussian edge basis.

\subsection{The Steinberg and cusp-divisor modules}\label{subsec:steinberg}

Put $F=\mathbf Q(i)$, $\mathcal O_F=\mathbf Z[i]$, and
$G=\operatorname{PSL}_2(\mathcal O_F)$.  The spherical Tits building
$\mathcal T_2(F)$ is the geometric realization of the poset of nonzero proper
$F$-subspaces of $F^2$.  Since $\dim_F F^2=2$, this is the discrete set of
lines in $F^2$, hence the discrete set $\mathbf P^1(F)$.  We use the standard
integral convention
\[
 \operatorname{St}_F:=\widetilde H_0(\mathcal T_2(F);\mathbf Z).
\]
Consequently there is a canonical $G$-equivariant identification
\begin{equation}\label{eq:st-delta}
 \operatorname{St}_F
 \cong \Delta_0(F):=\ker\!\left(
 \mathbf Z[\mathbf P^1(F)]\xrightarrow{\deg}\mathbf Z\right).
\end{equation}
Indeed, for any discrete set $S$, its reduced zeroth homology is the kernel of
the augmentation $\mathbf Z[S]\to\mathbf Z$.  This is the rank-one instance
of the usual definition of the Steinberg module as the top reduced homology of
the Tits building; see Kupers--Miller--Patzt--Wilson
\cite[\S1.2, Theorem~B]{KMPW}, especially pp.~10348--10349.

\subsection{The established Gaussian ideal edge--face complex}
\label{subsec:ideal-complex}

An ordered pair $(\alpha,\beta)$ of distinct cusps is a \emph{Gaussian
unimodular edge} if primitive lifts of the two lines form a basis of
$\mathcal O_F^2$.  Write $[\alpha,\beta]$ for its orientation and impose
$[\beta,\alpha]=-[\alpha,\beta]$.  An oriented Gaussian ideal triangle is a
cyclically ordered triple which, after unit changes of primitive lifts, is
represented by
\[
 Fv_1,\quad Fv_2,\quad F(v_1+v_2),
 \qquad (v_1,v_2)\text{ a basis of }\mathcal O_F^2.
\]
Let $C_1^{\mathrm{id}}$ and $C_2^{\mathrm{id}}$ be the corresponding
oriented cellular modules.  Reversing the orientation of a cell negates its
generator.

\begin{proposition}[Established Gaussian Steinberg presentation]
\label{prop:steinberg-presentation}
There is an exact sequence of left $G$-modules
\begin{equation}\label{eq:ideal-resolution}
 C_2^{\mathrm{id}}\xrightarrow{\partial_2}
 C_1^{\mathrm{id}}\xrightarrow{\partial_1}
 \operatorname{St}_F\longrightarrow0,
\end{equation}
where
\begin{align}
 \partial_1[\alpha,\beta]&=[\beta]-[\alpha],\label{eq:d1}\\
 \partial_2[\alpha,\beta,\gamma]
   &=[\alpha,\beta]+[\beta,\gamma]+[\gamma,\alpha].\label{eq:d2}
\end{align}
The ideal edges and triangles in \eqref{eq:ideal-resolution} are literally
the edges and triangular faces of the Gaussian Farey-octahedral tessellation
used in this paper; no subdivision or comparison map is required.
\end{proposition}

\begin{proof}
For an integral domain $R$, the generalized Bykovskii module is generated
by ordered bases $(v_1,v_2)$, modulo alternation, multiplication of a vector by
a unit, and the additive relation
\[
 [v_1,v_2]-[v_1+v_2,v_2]+[v_1+v_2,v_1]=0.
\]
For $R=\mathbf Z[i]$, Kupers--Miller--Patzt--Wilson
\cite[Theorem~B]{KMPW} prove integrally that this module is
$\operatorname{St}_F$; their Lemma~4.7 and Corollary~4.9
\cite[pp.~10371--10372]{KMPW} give the corresponding two-term partial
resolution.  Projecting primitive vectors to their $F$-lines identifies an
ordered basis with an oriented unimodular ideal edge.  The displayed additive
relation is exactly
\[
 [Fv_1,Fv_2]+[Fv_2,F(v_1+v_2)]+[F(v_1+v_2),Fv_1]=0,
\]
the boundary of the associated ideal triangle.  This proves exactness and the
formulas for the boundary maps.

It remains to identify the cells.  Cremona proves that the distinguished
geodesics $\{g0,g\infty\}$ form the one-skeleton of an ideal polyhedral
tessellation \cite[\S2.2, pp.~283--284]{Cremona1984}.  In the Gaussian case
his table identifies the basic polyhedron as an octahedron and its projective
stabilizer as a group of order twelve
\cite[\S2.3, p.~290 and Figure~2.3.1]{Cremona1984}.  Its vertices are
\[
 \infty,0,1,i,1+i,\frac{1+i}{2},
\]
after the positive normalization fixed in
Subsection~\ref{subsec:gaussian-rational}.  Cremona's generators for the
two triangular face families are $TS$ and
$X_{\mathrm{Cr}}=\left(\begin{smallmatrix}i&1\\1&0\end{smallmatrix}\right)$
\cite[pp.~290--291]{Cremona1984}.  In our notation $TS=X$, whereas
$iX_{\mathrm{Cr}}=Y=UR$.  Thus the two cyclic relation sums, expressed
in the present generators, are $1+X+X^2$ and $1+Y+Y^2$.  These are precisely the two
$G$-orbits into which the additive triangles above split.  Thus the
unimodular edges and additive triangles are exactly the coarse edges and
triangular faces of the manuscript's Gaussian Farey octahedra.  Gunnells'
Voronoi formulation gives the same identification: for
$K=\mathbf Q(i)$ the unique ideal three-polytope modulo $G$ is an octahedron
\cite[Example~2 in \S3.4, pp.~205--206]{Gunnells1999}.  Hence the comparison
map is the identity on cusps, oriented edges, oriented faces, boundary maps,
stabilizers, and orientation characters.
\end{proof}

For comparison with the older computational formulation, Cremona's
Theorem~2 \cite[p.~284]{Cremona1984} identifies his quotient of distinguished
edge symbols by the polyhedral relation ideal with Bianchi homology over
$\mathbf Q$; the Gaussian relations are written explicitly on p.~291.  The
integral exactness needed in Proposition~\ref{prop:steinberg-presentation} is
supplied by the later Steinberg-module theorem just cited.  Cremona's 1987
addendum \cite{Cremona1987} changes tables of elliptic curves, not the
construction in \S\S2.2--2.3.

\subsection{Finite orbit presentations and coefficients}
\label{subsec:finite-orbits}

Let $H\leq G$ have finite index.  For an underlying ideal cell $\bar\sigma$
choose an orientation $\sigma$, let
\[
 H_\sigma:=\{h\in H:h\bar\sigma=\bar\sigma\},
\]
and define its orientation character by
\[
 h\sigma=\epsilon_\sigma(h)\sigma,
 \qquad \epsilon_\sigma:H_\sigma\longrightarrow\{\pm1\}.
\]
Thus $H_\sigma$ is setwise, not pointwise, stabilizer; notation by the chosen
orientation does not exclude orientation-reversing elements.

\begin{proposition}[Finite $H$-orbit presentation]\label{prop:finite-orbits}
Choose representatives $\mathcal E_H$ and $\mathcal F_H$ for the $H$-orbits
of underlying ideal edges and faces, together with an orientation of each.
Then, as left $H$-modules,
\begin{align}
 C_1^{\mathrm{id}}&\cong
 \bigoplus_{e\in\mathcal E_H}
 \mathbf Z[H]\otimes_{\mathbf Z[H_e]}\mathbf Z_{\epsilon_e},\label{eq:C1ind}\\
 C_2^{\mathrm{id}}&\cong
 \bigoplus_{f\in\mathcal F_H}
 \mathbf Z[H]\otimes_{\mathbf Z[H_f]}\mathbf Z_{\epsilon_f}.
 \label{eq:C2ind}
\end{align}
There are finitely many summands.  Hence \eqref{eq:ideal-resolution} becomes
the finite induced-module presentation
\begin{equation}\label{eq:finite-induced}
 \bigoplus_{f\in\mathcal F_H}\operatorname{Ind}_{H_f}^{H}
       \mathbf Z_{\epsilon_f}
 \longrightarrow
 \bigoplus_{e\in\mathcal E_H}\operatorname{Ind}_{H_e}^{H}
       \mathbf Z_{\epsilon_e}
 \longrightarrow \operatorname{St}_F\longrightarrow0.
\end{equation}
Here ``finite'' means finitely many orbit summands; it does not assert that
$\operatorname{St}_F$ is finitely generated as an abelian group.

Let $V$ be a left $\mathbf Z[H]$-module and put
\[
 \operatorname{Symb}_H(V):=\operatorname{Hom}_{\mathbf Z[H]}
       (\operatorname{St}_F,V),\qquad
 V^{H_\sigma,\epsilon_\sigma}:=
 \{v\in V:hv=\epsilon_\sigma(h)v\ \text{for all }h\in H_\sigma\}.
\]
If
\begin{equation}\label{eq:boundary-coeffs}
 \partial_2 f_j=\sum_r\eta_{jr}h_{jr}e_{i(j,r)},
 \qquad \eta_{jr}\in\{\pm1\},\quad h_{jr}\in H,
\end{equation}
then applying $\operatorname{Hom}_H(-,V)$ gives
\begin{equation}\label{eq:delta}
 0\longrightarrow\operatorname{Symb}_H(V)\longrightarrow
 \bigoplus_{e\in\mathcal E_H}V^{H_e,\epsilon_e}
 \xrightarrow{\delta_H}
 \bigoplus_{f\in\mathcal F_H}V^{H_f,\epsilon_f},
 \qquad
 (\delta_Hv)_j=\sum_r\eta_{jr}\,h_{jr}v_{i(j,r)}.
\end{equation}
In particular,
\begin{equation}\label{eq:symb-kernel}
 \operatorname{Symb}_H(V)=\ker\delta_H.
\end{equation}
No inverse occurs in \eqref{eq:delta}.  If coefficients are originally a
right $H$-module with slash action, the corresponding left action is
$h\cdot v=v\mathbin{|}h^{-1}$.
\end{proposition}

\begin{proof}
The orbit of an oriented cell $\sigma$ is the quotient of the free left
$H$-set $H\times\{\sigma\}$ by $(hh_0,\sigma)\sim
(h,\epsilon_\sigma(h_0)\sigma)$ for $h_0\in H_\sigma$.  Its oriented cellular
module is therefore the indicated induced orientation module.  The number of
$H$-orbits is finite because $H$ has finite index and $G$ has one edge orbit
and two face orbits.  This proves \eqref{eq:C1ind}--\eqref{eq:finite-induced}.

Frobenius reciprocity sends an $H$-map on
$\mathbf Z[H]\otimes_{\mathbf Z[H_\sigma]}\mathbf Z_{\epsilon_\sigma}$ to
the value of that map at $1\otimes1$, which lies in
$V^{H_\sigma,\epsilon_\sigma}$.  Evaluating an equivariant cochain on
\eqref{eq:boundary-coeffs} gives the displayed formula for $\delta_H$.
Left exactness of $\operatorname{Hom}_H(-,V)$ now gives
\eqref{eq:delta} and \eqref{eq:symb-kernel}.
\end{proof}

\subsection{Picard Farey symbols and the Gaussian Steinberg presentation}
\label{subsec:picard-compatibility}

\begin{theorem}[Picard Farey symbols and the Gaussian Steinberg presentation]
\label{thm:picard-modsym}
Let $\Sigma$ be a marked Picard Farey symbol and let $H=H_\Sigma$ be
the finite-index subgroup reconstructed by Theorem~\ref{thm:farey-reconstruction}.  The data of
$\Sigma$---its Gaussian-rational octahedral occurrences, the local groups
$J_x$, and the ordered/cooriented source-to-target face-pairing transformations---determine,
without using the universal twelve-sector subdivision:
\begin{enumerate}
 \item the $H$-orbits of oriented ideal edges and faces;
 \item their setwise stabilizers and orientation characters;
 \item the transport elements and incidence signs in every face boundary;
 \item the induced-module presentation \eqref{eq:finite-induced}; and
 \item for every left $H$-module $V$, the relation map $\delta_H$ in
       \eqref{eq:delta} and hence $\operatorname{Symb}_H(V)$.
\end{enumerate}
The construction is canonical up to the usual change of orbit
representatives, which gives canonically isomorphic induced summands.
\end{theorem}

\begin{proof}
For each occurrence retain its twelve coarse edges and eight coarse oriented
faces.  First quotient these finite sets by the local action of $J_x$.
For a stored face pairing $(F,F',q)$, use the fixed convention
\[
 qF=\overline{F'}.
\]
Thus $q$ identifies the three source edges with the three target edges, and
the odd permutation between the two stored outward boundary orders supplies
the incidence signs.  Taking the equivalence closure of the local and paired
edge identifications gives the $H$-edge orbits; paired face classes give the
$H$-face orbits.

Attach to each elementary identification its known group label.  Products
along paths in the finite edge- or face-incidence groupoid transport a chosen
representative to every occurrence.  Labels of loops are precisely the
setwise cell stabilizers, and their parity on the ordered edge or face is the
orientation character.  Choosing a spanning forest in each orbit groupoid
therefore produces all $h_{jr}$ and $\eta_{jr}$ in
\eqref{eq:boundary-coeffs}.  Changing the forest conjugates stabilizers and
changes transports by the standard induced-module identifications, leaving
the presentation isomorphic.

This construction uses only the ideal edges and faces of the octahedral occurrences.  Proposition
\ref{prop:steinberg-presentation} identifies those cells themselves with the
Gaussian Steinberg complex, so the twelve Picard sectors inside an octahedron
carry no additional data relevant to $C_2^{\mathrm{id}}\to C_1^{\mathrm{id}}$.
Theorem~\ref{thm:farey-reconstruction} and its fixed equivariance convention are used only to know
that the finite labelled groupoid is the quotient by the reconstructed
subgroup $H$.  Proposition~\ref{prop:finite-orbits} completes the argument.
\end{proof}

\begin{corollary}[Construction of the finite Gaussian Steinberg presentation]\label{cor:effective-steinberg}
Let $H\leq G$ have finite index and suppose membership in $H$ can be decided for exact matrices.  Then one can construct, by finitely many membership tests, the finite $H$-orbit presentation
\[
 \bigoplus_{f\in\mathcal F_H}\operatorname{Ind}_{H_f}^{H}\ZZ_{\epsilon_f}
 \longrightarrow
 \bigoplus_{e\in\mathcal E_H}\operatorname{Ind}_{H_e}^{H}\ZZ_{\epsilon_e}
 \longrightarrow \operatorname{St}_{\QQ(i)}\longrightarrow0.
\]
For every left $\ZZ[H]$-module $V$ for which the required $H$-action can be evaluated, the same finite data determine the relation map $\delta_H$ and hence $\operatorname{Symb}_H(V)=\ker\delta_H$.
\end{corollary}
\begin{proof}
Construct a marked Picard Farey symbol by Theorem~\ref{thm:constructive-symbol}.  Theorem~\ref{thm:picard-modsym} then reads the edge and face orbits, stabilizers, orientation characters and incidence coefficients directly from its octahedral one- and two-skeleton.  Both stages are finite.
\end{proof}

\begin{remark}[Arithmetic label of the finite Steinberg presentation]\label{rem:steinberg-level-shadow}
The finite edge--face presentation and the ideal level are two consequences of the same Picard Farey symbol.  Proposition~\ref{prop:level-from-symbol} attaches the standard Gaussian level to the presentation through the cusp translation lattices $L_C$, while Proposition~\ref{prop:split-prime-shadow} identifies, at a split degree-one prime, the restriction of the reconstructed finite $G$-action to the embedded modular group with the corresponding classical level-$p$ coset action.  Thus the Gaussian Steinberg presentation for the level-$(2+i)$ examples carries the same finite modular shadow described in Section~\ref{sec:invariants}.  This compatibility concerns the arithmetic label of the finite presentation and the finite modular shadow; a geometric special-polygon section is a separate construction.
\end{remark}

In particular, the finite Steinberg presentation is obtained from the decorated octahedral two-skeleton itself; the twelve Picard sectors are needed for the subgroup reconstruction theorem but not for this passage.  Arbitrary non-unimodular cusp symbols are not reduced by this finite presentation alone.

\subsection{\texorpdfstring{The complete $\Gamma_1(2+i)$ presentation}{The complete Gamma1(2+i) presentation}}
\label{subsec:gamma1-example}

Let $H=\Gamma_1(2+i)$ in the projective convention fixed in Section~6 and put
$w=(1+i)/2$.  Its Picard Farey symbol consists of one octahedron, has $J=1$,
and is torsion-free.  Label its outward oriented faces
\[
\begin{array}{llll}
 A=(0,1,\infty),&B=(1,1+i,\infty),
 &C=(1+i,i,\infty),&D=(i,0,\infty),\\
 E=(0,w,1),&F=(1,w,1+i),
 &G=(1+i,w,i),&H_0=(i,w,0).
\end{array}
\]
The pairings are $A\leftrightarrow C$, $B\leftrightarrow D$,
$E\leftrightarrow G$, and $F\leftrightarrow H_0$.

The matrices used in the earlier face-pairing calculation were recorded as return matrices carrying the target lift back to the source lift.  The source-to-target face-pairing transformations used here are their inverses.
Choose the following representatives in $\operatorname{PSL}_2(\mathbf Z[i])$:
\begin{equation}\label{eq:abcd}
\begin{aligned}
 a&=\begin{pmatrix}1&i\\0&1\end{pmatrix},
&b&=\begin{pmatrix}1&-1\\0&1\end{pmatrix},\\
 c&=\begin{pmatrix}-2i&i\\-2-i&1+i\end{pmatrix},
&d&=\begin{pmatrix}-1&1+i\\-1+2i&2-i\end{pmatrix}.
\end{aligned}
\end{equation}
They carry $A,B,E,F$ respectively to $C,D,G,H_0$ with opposite outward
boundary orientation.  For example,
\[
 a(0,1,\infty)=(i,1+i,\infty)
   =\overline{(1+i,i,\infty)}.
\]
All four matrices reduce modulo $(2+i)$ to the upper-unipotent subgroup, so
they lie in $H$.

There are three unoriented edge orbits, represented with orientations by
\[
 e_1=[0,1],\qquad e_2=[1,\infty],\qquad e_3=[0,w].
\]
Their underlying edge sets are
\[
\begin{array}{c|l}
e_1&01,\ i(1+i),\ 1w,\ iw\\
e_2&0\infty,\ 1\infty,\ i\infty,\ (1+i)\infty\\
e_3&0w,\ (1+i)w,\ 0i,\ 1(1+i).
\end{array}
\]
There are four face orbits, represented by $f_A=A$, $f_B=B$, $f_E=E$,
and $f_F=F$, with paired members $C,D,G,H_0$.  Torsion-freeness and the
finiteness of ambient ideal-cell stabilizers imply
\[
 H_{e_r}=H_{f_j}=1,
 \qquad \epsilon_{e_r}=\epsilon_{f_j}=1.
\]

The complete boundary expressions are
\begin{align}
 \partial_2f_A&=e_1+(1-b)e_2,\label{eq:bdA}\\
 \partial_2f_B&=(a-1)e_2-d^{-1}e_3,\label{eq:bdB}\\
 \partial_2f_E&=-(1+c^{-1}a)e_1+e_3,\label{eq:bdE}\\
 \partial_2f_F&=c^{-1}a\,e_1+(c+d^{-1})e_3.\label{eq:bdF}
\end{align}
For instance, $b[1,\infty]=[0,\infty]$, so the last edge of $A$ is
$[\infty,0]=-b e_2$.  The other identities follow similarly from the four
ordered face maps.  Thus, with rows $(e_1,e_2,e_3)$ and columns
$(f_A,f_B,f_E,f_F)$, the exact group-ring boundary matrix is
\begin{equation}\label{eq:gamma1-matrix}
 M_{\partial_2}=
 \begin{pmatrix}
 1&0&-(1+c^{-1}a)&c^{-1}a\\
 1-b&a-1&0&0\\
 0&-d^{-1}&1&c+d^{-1}
 \end{pmatrix}\in M_{3\times4}(\mathbf Z[H]).
\end{equation}
Together with
\[
 \partial_1e_1=[1]-[0],\qquad
 \partial_1e_2=[\infty]-[1],\qquad
 \partial_1e_3=[w]-[0],
\]
this gives the complete presentation
\[
 \mathbf Z[H]^4\xrightarrow{M_{\partial_2}}
 \mathbf Z[H]^3\xrightarrow{\partial_1}
 \operatorname{St}_F\longrightarrow0.
\]
Direct substitution shows $\partial_1M_{\partial_2}=0$ column by column.
For exact reproduction one may use
\[
 c^{-1}a=\pm\begin{pmatrix}1+i&-1\\2+i&-1\end{pmatrix},
 \qquad
 d^{-1}=\begin{pmatrix}-2+i&1+i\\-1+2i&1\end{pmatrix}.
\]

For a left $H$-module $V$, the coefficient relation map
$\delta_H:V^3\to V^4$ is therefore
\begin{equation}\label{eq:gamma1-delta}
\begin{split}
 (\delta_Hv)_A&=v_1+v_2-bv_2,\\
 (\delta_Hv)_B&=av_2-v_2-d^{-1}v_3,\\
 (\delta_Hv)_E&=-v_1-c^{-1}av_1+v_3,\\
 (\delta_Hv)_F&=c^{-1}av_1+cv_3+d^{-1}v_3,
\end{split}
\qquad
 \operatorname{Symb}_H(V)=\ker\delta_H.
\end{equation}

This is the orbit-compressed form of Cremona's Gaussian Manin-symbol
presentation.  His twelve coset symbols for this index-twelve group are
first subjected to the two-term edge-stabilizer relations; the resulting
three edge orbits are $e_1,e_2,e_3$.  The two Gaussian three-term relation
families, represented by the two face types $A$ and $E$, split into the four
$H$-face orbits above.  Thus \eqref{eq:gamma1-matrix} is equivalent to the
specialization of Cremona's relation ideal, but is organized by the coarse
Picard Farey quotient rather than by all twelve cosets.

\subsection{Comparison of the finite encodings}\label{subsec:comparison}
The marked symbol feeds two different finite constructions:
\[
\begin{aligned}
\text{marked Picard Farey symbol}
&\longrightarrow \text{fully marked finite-face covering}\\[-2pt]
&\longleftrightarrow \text{pointed finite right }G\text{-set},\\[5pt]
\text{marked Picard Farey symbol}
&\longrightarrow \text{ideal edge--face orbit presentation}\\[-2pt]
&\longrightarrow \text{Steinberg presentation}
\longrightarrow \operatorname{Symb}_H(V).
\end{aligned}
\]
In the first branch, Proposition~\ref{prop:finite-face-covering} evaluates the
transition system on $[G:H]$ quadrilateral two-cells and retains local groups,
incidence injections and twisting data.  A fully ambient-marked covering and
a pointed finite right $G$-set determine one another.  Recovering particular
printed octahedral placements and representative face matrices requires
choices, although the marked-symbol equivalence class is determined.

The second branch is Result~C.  Its Gaussian ideal edge--face orbit matrix is
read directly from the labelled coarse ideal one- and two-skeleton, without
materializing the finite-face quotient.  Passing to the Steinberg
presentation keeps the linearized orbit, stabilizer, orientation and boundary
data; applying a coefficient module then retains the kernel of the resulting
cochain relation map.  These arrows therefore forget different choices and
retain different structures.  No strict comparison such as ``richer'' or
``stronger'' is intended between the fully marked finite-face covering and
the marked octahedral symbol.

\subsection{Hecke operators and limitations}\label{subsec:limitations}

Cremona's Euclidean continued fractions and Gunnells' Voronoi reduction
provide established methods for reducing symbols and computing Hecke action.
The present paper does not prove a new Bianchi modular-symbol theory, a new
continued-fraction algorithm, a new Hecke-reduction algorithm, computational
superiority over Cremona/Gunnells/Voronoi methods, or an overconvergent
control theorem.  The result proved here is the narrower compatibility
statement that the Picard Farey encoding already contains, at the octahedral
level, the finite orbit data for the established Gaussian Steinberg
edge--face presentation.  Here ``finite presentation'' means this finite
orbit/stabilizer presentation of the Steinberg module, not merely a finite
group presentation of $H$.  The remaining algorithmic question is whether an
arbitrary cusp divisor, and in particular the image of an edge generator
under a Hecke correspondence, can be reduced back to this finite edge basis
with controlled path length, coefficient growth and matrix sparsity.  The
$M$-adapted shadow theorem supplies useful planar subcomplexes in split-prime
examples but does not by itself provide such a reduction algorithm.

\section{Open problems}\label{sec:open}
The reconstruction, finite construction, standard ideal level, congruence
recognition and existence of an $M$-adapted geometric lift with the prescribed modular section are now
settled in the Picard case.  This section does not add another theorem package;
it isolates the questions that remain after the preceding results and the
explicit examples, with canonical reduction and a fully intrinsic
three-dimensional analogue of Kulkarni admissibility kept deliberately open.  The comparison examples also show that the
remaining geometric questions should not be compressed into a single
``minimum boundary'' problem.  Three different objectives occur naturally:
minimization inside the class of complete Farey-octahedron developments,
absolute minimization of geometric faces after allowing finer cuts and
legitimate mergers, and compatibility with a prescribed totally geodesic
modular section.  The examples below show that these objectives can select
different domains for the same subgroup.

\subsection{Admissibility and geometric reduction}\label{subsec:open-admissibility}
Kulkarni's special polygons are selected by local even/odd/free boundary rules
which imply a global admissibility theorem.  Section~\ref{sec:picard-geometry}
gives the Picard analogues of the local reflection geometry---the Coxeter
tetrahedron, the $\pi/2$ and $\pi/3$ ridge types, and the extended symmetry---
but the three-dimensional global conditions are stronger: complete ridge
cycles, finite vertex links and Euclidean cusp links must also be controlled.

The example $\Gamma_0(1+i)$ gives a first test case.  A connected union of
three standard Picard cells has a seven-face description (after the
single matrix correction recorded in Appendix~\ref{app:gamma0pi}), while the
same group has an ideal tetrahedral fundamental domain with four
self-paired faces.  Four is the absolute minimum number of faces of a
nondegenerate three-dimensional polyhedron.  The two side-pairing systems are
algebraically related, but no sequence of individually justified cuts,
translations and face mergers from the seven-face domain to the tetrahedron
is presently known.

\begin{problem}[Special Picard polyhedra and elementary moves]\label{prob:admissibility}
Find intrinsic boundary and link conditions on finite unions of reflection
pieces, Picard cells or Farey octahedra which imply that the resulting face
pairings give a fundamental polyhedron for the generated subgroup.  Determine
whether every finite-index Picard subgroup admits such a representative and
identify a finite family of elementary moves which preserves the subgroup and
the Poincar\'e conditions.  As a first test, exhibit---or obstruct---a legal
move sequence between the seven-face and four-face
$\Gamma_0(1+i)$ domains.
\end{problem}

The finite-face retract and ordinary presentation extraction are classical;
the question here concerns geometric cuts and face identifications that
preserve complete ridge cycles and links.  One may additionally ask, after
fixing torsion relations and deleting inverse and cyclic duplicates, whether
a chosen generating set $\mathcal S$ and relator set $\mathcal R$ have the
properties
\[
s\notin\langle\mathcal S\setminus\{s\}\rangle
\quad(s\in\mathcal S),
\qquad
r\notin\langle\!\langle\mathcal R\setminus\{r\}\rangle\!\rangle
\quad(r\in\mathcal R).
\]
Neither property follows from choosing a maximal tree in the quotient
one-skeleton.

This problem is deliberately geometric.  Equality of the associated subgroup
presentations is not enough: every intermediate cut or merger must be
geometrically valid and the resulting ridge and cusp cycles must remain
complete.

\subsection{Three notions of minimum complexity}\label{subsec:open-minimality}
For a torsion-free subgroup $H$ of index $n$, every whole-octahedron
fundamental domain contains
\[
N=\frac{n}{12}
\]
complete octahedra.  If a connected development has $e_{\rm int}$ internal
triangular face-pairs, then
\[
 b_\triangle=8N-2e_{\rm int}.
\]
The spanning-tree construction gives $e_{\rm int}\ge N-1$ and hence
$b_\triangle\le6N+2$, but cycles can improve this.  For
$\Gamma(2+i)$, $N=5$: the explicit development
\[
\mathcal O,\quad\mathcal O-1,\quad\mathcal O-i,\quad
\mathcal O-1-i,\quad\mathcal O+i
\]
has five internal face pairs and therefore $b_\triangle=30$, improving the
spanning-tree value $32$.  The exact search covers all $1640$ connected rooted five-octahedron
transversals in its specified complete-octahedron class: $1520$ have four
internal face pairs, $120$ have five, and none has six or more.  Thus $30$
is minimal within that searched class; unrestricted geometric-face
minimality remains open.

The $M$-adapted lifting theorem supplies a third objective.  The compressed
one-octahedron domain for $\Gamma_1(2+i)$ has eight exposed Farey triangles,
but it does not have the $\Gamma_1(5)$ special polygon as its real-plane
section.  The twelve-cell $M$-adapted development does.  Thus a domain with one octahedron need not furnish the prescribed modular
section.

\begin{problem}[Reduced domains and comparison of objectives]\label{prob:min-boundary}
Develop reduction notions for the following three categories and determine
their relation:
\begin{enumerate}[label=(\roman*),leftmargin=2.2em]
\item connected whole-octahedron developments, minimizing $b_\triangle$;
\item general Poincar\'e polyhedra, minimizing geometric faces after permitted
cuts and mergers;
\item $M$-adapted developments, minimizing boundary complexity subject to a
prescribed modular special-polygon section.
\end{enumerate}
For each category, ask for terminating local moves, existence and possible
uniqueness of reduced representatives, and group-theoretic lower bounds.
Determine in particular whether one can find representatives that are
simultaneously reduced in more than one of these senses.
\end{problem}

\subsection{Arithmetic reduction of Picard Farey symbols}\label{subsec:open-arithmetic-reduction}
Theorem~\ref{thm:constructive-symbol} constructs a Picard Farey symbol without
choosing a height, a continued fraction or a preferred ordering of Gaussian
rationals.  Construction and reduction should therefore remain separate.
Hockman's geodesic Gaussian continued fractions supply a geometry-compatible
sequence of Farey triangles, but their norm is invariant under Gaussian
translation and does not by itself give a global complexity on all symbols
\cite[Theorem~1.1]{Hockman2019}.

There are two further clues.  Kulkarni's Section~8 relates classical Farey
symbols to continued-fraction reduction, suggesting that the first Picard
analogue should be a reduction of cusp paths rather than an immediate global
normal form for an entire three-dimensional symbol.  Dani--Nogueira develop
a broad continued-fraction framework over $\mathbf C$ in which consecutive
Gaussian convergents are governed by unimodular relations
\cite{DaniNogueira2014}; this is naturally compatible with Farey-neighbour
edges, but the nonuniqueness of complex continued fractions cautions against
treating a chosen expansion as a canonical global reduction.

\begin{problem}[Reduced Gaussian Farey symbols]\label{prob:gaussian-normal-form}
Find a well-founded arithmetic complexity, or an equivalent finite reduction
procedure, on Picard Farey symbols expressed in Gaussian-rational
coordinates.  The reduction should preserve the subgroup, local stabilizers
and cusp data.  Compare any proposed rule with Hockman's geodesic reduction,
Dani--Nogueira type continued fractions, and Vorono\"i/Cremona reduction,
and distinguish a preferred arithmetic representative from a merely
terminating construction.
\end{problem}

\subsection{Intrinsic geometric modular shadows and compression}\label{subsec:open-level-shadow}
The existence question suggested by the modular slice is no longer open.
Proposition~\ref{prop:M-adapted-lifting} proves that whenever
$(H\cap M)\backslash M\to H\backslash G$ is bijective, a side-connected
Kulkarni polygon lifts to an $M$-adapted Picard fundamental polyhedron with
exactly that section; Corollary~\ref{cor:split-prime-geometric-shadow}
applies this to the standard split-prime congruence families.  What remains
is an intrinsic criterion for a \emph{given} Picard development, especially a
compressed one.

\begin{problem}[Intrinsic modular section and compression]\label{prob:modular-shadow}
Give necessary and sufficient conditions on a developed Picard Farey symbol
for its intersection with $\Pi$ to be a Kulkarni special polygon for
$H\cap M$.  Determine how such a section changes under octahedral compression
and under the elementary moves sought in Problem~\ref{prob:admissibility}.
When a special section exists, recover its classical side-pairing matrices,
widths and marked widths directly from the Gaussian data.  Determine whether
every split-prime congruence subgroup admits a representative which is both
$M$-adapted and reduced in one of the senses of
Problem~\ref{prob:min-boundary}.
\end{problem}

The exact example $\Gamma_1^{\rm Pic}(2+i)$ is the basic warning: the
one-octahedron domain and the twelve-cell $M$-adapted domain encode the same
subgroup, but only the latter has the full $\Gamma_1(5)$ special polygon as
section.  A related coordinate question is to express membership of Gaussian
face-pairing transformations in $\Gamma_0(\mathfrak a)$,
$\Gamma_1(\mathfrak a)$ or $\Gamma(\mathfrak a)$ directly in terms of
Gaussian numerator/denominator data, in the spirit of Kulkarni's balanced and
semibalanced congruences.

\subsection{The other Euclidean Bianchi fields}
Cremona and Yasaki describe the Euclidean cells for the other imaginary
quadratic Euclidean rings, while Nakada--Natsui--Thuswaldner develop
corresponding Farey structures
\cite{Cremona1984,Yasaki2010,NakadaNatsuiThuswaldner2026}.  The Picard
analysis suggests that, before attempting a uniform theorem, one should
determine both the ordinary cell stabilizers and the effect of the
determinant-unit and reflection extensions in each field.

\begin{problem}[A second Euclidean field]\label{prob:second-field}
For $K=\QQ(\sqrt{-2})$ or $\QQ(\sqrt{-3})$, fix the standard Farey cell or
finite set of cell types, compute their stabilizers and subdivisions by
standard Bianchi cells, determine the relevant extended symmetries and face
orbits, and prove the analogue of Theorem~\ref{thm:farey-reconstruction}.
Include at least one subgroup with more than one Farey-cell orbit.
\end{problem}

\subsection{Farey reduction of cusp paths and Hecke operators}\label{subsec:open-path-reduction}
Proposition~\ref{prop:steinberg-presentation} gives the integral Gaussian
relation module.  A further problem is to reduce arbitrary cusp pairs to its
finite edge basis and bound path length, coefficient growth, matrix sparsity,
number of orbit representatives and running time against existing methods.

\begin{problem}[Farey reduction of cusp paths]\label{prob:path-reduction}
Given $[\beta]-[\alpha]\in\Delta_0(\QQ(i))$, construct a direct reduction to
the finite oriented edge orbits determined by a Picard Farey symbol and prove
correctness.  Interpret the reduction geometrically as an edge path in the
Gaussian Farey tessellation, compare it with the continued-fraction viewpoint
of Kulkarni's Section~8 and with complex continued fractions such as
Dani--Nogueira's, and measure path length, coefficient growth and running time
against Cremona/Vorono\"i reduction.
\end{problem}

\begin{problem}[Hecke compatibility]\label{prob:hecke}
For a Hecke correspondence, reduce the resulting sum of arbitrary cusp pairs
back to the finite Steinberg edge basis.  Determine whether the resulting
algorithm has competitive reduction length, coefficient growth, matrix
sparsity and running time compared with existing Cremona/Gunnells methods.
\end{problem}

\subsection{Overconvergent refinements}
Only after an effective Hecke-compatible reduction is available does it
become reasonable to ask whether the same finite presentation is convenient
for distribution-valued symbols.  A Bianchi analogue of the
Pollack--Stevens overconvergent control theorem would require new analytic
input and is a separate project rather than a formal consequence of the
geometry developed here.

\appendix
\section{Exact Picard examples}\label{app:examples}

The appendix collects the exact low-index calculations that support and test
the general statements in the main text.  The ambient Picard cell, Lee's
index-$12$ constructions and the classical side-pairing machinery come from
the cited literature; the tables are rewritten in the common positive
Gaussian normalization used in this paper, and the level-$(2+i)$ and
boundary-complexity computations are the finite checks used above.  In
particular, the Whitehead-link calculation in
Subsection~\ref{app:octahedra} records the topological information recovered
from one marked octahedron without turning the main text into a separate
survey of arithmetic links.  Taken together, the examples separate four
phenomena used repeatedly in the paper: local octahedral isotropy can reduce
the number of Picard cells over one occurrence; internal faces of a block
decomposition may disappear from the exposed boundary; the same geometric
octahedron can support different subgroup structures through different
ordered face pairings; and more than one octahedral occurrence can be joined
by genuine inter-occurrence face classes, as in \(\Gamma_0(3)\).

\subsection{The Picard group}\label{app:picard}
Put \(\pi=1+i\) and retain the matrices \(S,R,T,U,X,Y\) from Subsection~\ref{subsec:bianchi}.  A standard Picard fundamental cell has vertices
\[
0,\quad \infty,\quad
P_1=(1/2,0,\sqrt3/2),\quad
P_2=(1/2,1/2,1/\sqrt2),\quad
P_3=(0,1/2,\sqrt3/2).
\]
Its principal boundary faces are paired by
\[
\begin{array}{c|c}
\text{face(s)}&\text{pairing}\\ \hline
(0,P_3,\infty)&S\\
(0,P_1,\infty)&R\\
(0,P_1,P_2)\leftrightarrow(\infty,P_1,P_2)&X\\
(0,P_2,P_3)\leftrightarrow(\infty,P_2,P_3)&Y.
\end{array}
\]
Let \(j=(0,0,1)\) on the edge \(0\infty\).  Subdivision through \(j\) removes the self-face inversions and gives the reduced quotient square
\[
j\xrightarrow{R}P_1\xrightarrow{X}P_2\xrightarrow{Y}P_3\xrightarrow{S}j.
\]
The local groups are
\[
\begin{array}{c|c}
\text{cell}&\text{stabilizer}\\ \hline
j&C_2\times C_2\\
P_1&S_3\\
P_2&A_4\\
P_3&S_3\\
jP_1&C_2\\
P_1P_2&C_3\\
P_2P_3&C_3\\
P_3j&C_2.
\end{array}
\]
These give the presentation \eqref{eq:picard-presentation}.  The six raw edge cycles may be represented by
\[
RS,\quad RX,\quad X,\quad X^{-1}Y,\quad Y,\quad SY,
\]
with stabilizers \(C_2,C_2,C_3,C_2,C_3,C_2\), respectively.  The cusp stabilizer is
\[
G_\infty=\langle T=XS,U=YR,L=RS\mid [T,U]=1,\ L^2=1,\ LTL=T^{-1},\ LUL=U^{-1}\rangle
\cong\ZZ^2\rtimes C_2,
\]
so the cusp cross-section is the pillowcase \(S^2(2,2,2,2)\).

\subsection{The subgroup \texorpdfstring{$\Gamma_0(1+i)$}{Gamma0(1+i)}}\label{app:gamma0pi}
Let
\[
H=\Gamma_0(\pi)=\left\{\begin{pmatrix}a&b\\c&d\end{pmatrix}\in G:c\equiv0\pmod\pi\right\},
\qquad \pi=1+i.
\]
Reduction modulo \(\pi\) gives \([G:H]=3\).  Identify the left-coset set
\(H\backslash G\), equipped with its right \(G\)-action, with
\[
\PP^1(\mathbf F_2)=\{q_0=(0:1),q_\infty=(1:0),q_1=(1:1)\}
\]
and choose representatives \(1,X,X^2\).  Then
\[
S\equiv R=(q_0\ q_\infty),\qquad
X\equiv Y=(q_0\ q_\infty\ q_1).
\]
Thus the Picard quotient square lifts to three faces.  The \(R\)- and
\(S\)-edges each split into one free lift and one folded \(C_2\)-lift, while
the \(X\)- and \(Y\)-edges have one free lift.  The quotient has
\[
f_2=3,\qquad f_1=6,\qquad f_0=5.
\]
At the Farey-octahedron level there is one occurrence with local group
\[
J=\{1,ac^2,c^2a,aca\},\qquad |J|=4,
\]
so \([K:J]=3\).  The three cosets of \(J\) expand to the three Picard cells
above.  This is the smallest example in which one ambient cell type has both
singular and free lifts.

The stabilizer-labelled quotient may be summarized as follows:
\begin{center}\small
\begin{tabular}{lll}
\toprule
cell over the Picard square & number of lifts & stabilizers\\
\midrule
square & 3 & all trivial\\
\(R\)-edge & 2 & \(1\), \(C_2\)\\
\(X\)-edge & 1 & trivial\\
\(Y\)-edge & 1 & trivial\\
\(S\)-edge & 2 & \(1\), \(C_2\)\\
\(j\) & 2 & \(C_2\), \(C_2\times C_2\)\\
\(P_1\) & 1 & \(C_2\)\\
\(P_2\) & 1 & \(C_2\times C_2\)\\
\(P_3\) & 1 & \(C_2\)\\
\bottomrule
\end{tabular}
\end{center}
The lifted raw edge cycles consist of eight order-two cycles and four free
cycles.  There are two cusp classes; their translational lattices have
indices \(1\) and \(2\) in \(\ZZ[i]\), and both retain the order-two
rotational part, so both cross-sections are pillowcases.

A connected three-cell realization of the same subgroup is
\[
\mathcal P_7=P_0\cup XP_0\cup X^2P_0.
\]
Put
\[
P_2=\left(\frac12,\frac12,\frac1{\sqrt2}\right),\qquad
P_3=\left(0,\frac12,\frac{\sqrt3}{2}\right),
\]
\[
XP_3=\left(1,\frac12,\frac{\sqrt3}{2}\right),\qquad
X^2P_3=\left(\frac12,\frac14,\frac{\sqrt3}{4}\right).
\]
Its seven exposed geometric faces are
\[
\begin{gathered}
(\infty,0,1),\quad
(\infty,0,P_3),\quad(\infty,1,XP_3),\quad
(0,1,X^2P_3),\\
(\infty,P_3,P_2,XP_3),\quad
(1,XP_3,P_2,X^2P_3),\quad
(0,P_3,P_2,X^2P_3).
\end{gathered}
\]
The last three are genuine quadrilateral faces; the diagonals used to list
their triangular halves are not exposed boundary edges.  Introduce
\[
\begin{aligned}
\alpha&=\begin{pmatrix}-i&i\\0&i\end{pmatrix},&
\tau&=\begin{pmatrix}1&1\\0&1\end{pmatrix},&
\gamma&=\begin{pmatrix}-i&-1+i\\0&i\end{pmatrix},\\
\eta&=\begin{pmatrix}-1&1+i\\-1+i&1\end{pmatrix},&
\delta&=\begin{pmatrix}-i&0\\-1-i&i\end{pmatrix},&
\jmath&=\begin{pmatrix}-1&1\\-2&1\end{pmatrix}.
\end{aligned}
\]
The complete ordered source-to-target table is
\begin{center}\scriptsize
\begin{tabular}{ccl}
\toprule
pairing & ordered source & ordered image\\
\midrule
\(\alpha\) & \((\infty,0,1)\) & \((\infty,1,0)\)\\
\(\tau\) & \((\infty,0,P_3)\) & \((\infty,1,XP_3)\)\\
\(\gamma\) & \((\infty,P_3,P_2)\) & \((\infty,XP_3,P_2)\)\\
\(\eta\) & \((1,XP_3,P_2)\) & \((1,X^2P_3,P_2)\)\\
\(\delta\) & \((0,P_3,P_2)\) & \((0,X^2P_3,P_2)\)\\
\(\jmath\) & \((0,1,X^2P_3)\) & \((1,0,X^2P_3)\)\\
\bottomrule
\end{tabular}
\end{center}
The upper-right entry of \(\eta\) is \(1+i\); the old printed entry \(2i\)
has nonunit determinant.  With this correction the Poincar\'e relations may
be written
\[
\alpha^2=\eta^2=\gamma^2=\delta^2=\jmath^2
=(\alpha\tau)^2=(\gamma\tau)^2=(\delta\eta\gamma)^2
=\tau\delta\jmath\eta=(\alpha\jmath)^2=1.
\]

\begin{proposition}[A face-minimal tetrahedral domain for $\Gamma_0(1+i)$]
\label{prop:gamma0pi-tetrahedron}
Besides \(\mathcal P_7\), the subgroup \(\Gamma_0(1+i)\) has the ideal
tetrahedral fundamental polyhedron
\[
\mathcal T_{1+i}=(\infty,0,1,i)
\]
with four self-paired faces
\[
\begin{array}{c|c|c}
\text{face}&\text{pairing}&\text{ordered image}\\ \hline
(\infty,0,1)&\alpha=\begin{psmallmatrix}-i&i\\0&i\end{psmallmatrix}&(\infty,1,0)\\[2mm]
(\infty,0,i)&\beta=\begin{psmallmatrix}-i&-1\\0&i\end{psmallmatrix}&(\infty,i,0)\\[2mm]
(\infty,i,1)&\gamma=\begin{psmallmatrix}-i&-1+i\\0&i\end{psmallmatrix}&(\infty,1,i)\\[2mm]
(0,1,i)&\delta=\begin{psmallmatrix}-i&0\\-1-i&i\end{psmallmatrix}&(0,i,1).
\end{array}
\]
All four matrices lie in \(\Gamma_0(1+i)\) and are involutions in
\(\PSL_2(\ZZ[i])\).  The three edge cycles have words
\(\gamma\beta\alpha\), \(\delta\beta\delta\alpha\), and \(\gamma\delta\),
each of order two and developed angle \(2\pi\).  There are two cusp orbits,
\(\{\infty\}\) and \(\{0,1,i\}\), with complete Euclidean pillowcase
links.  Consequently
\[
\Gamma_0(1+i)=
\langle \alpha,\beta,\gamma,\delta\mid
\alpha^2=\beta^2=\gamma^2=\delta^2
=(\gamma\beta\alpha)^2=(\delta\beta\delta\alpha)^2
=(\gamma\delta)^2=1\rangle.
\]
Since a nondegenerate three-dimensional polyhedron has at least four
geometric faces, this domain attains the absolute face minimum.
\end{proposition}

\begin{figure}[htbp]
\centering
\includegraphics[width=.88\textwidth]{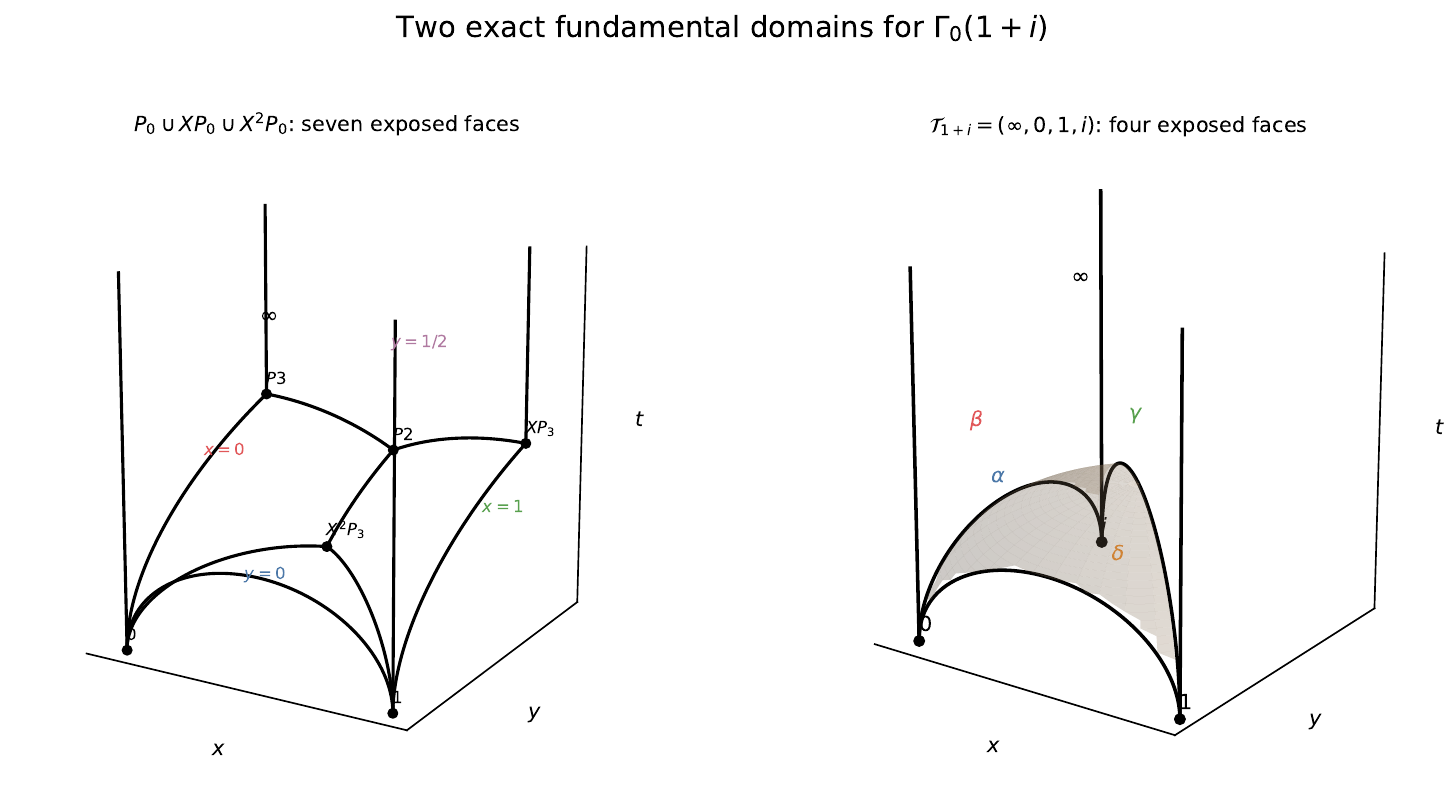}
\caption{Two exact domains for \(\Gamma_0(1+i)\): the seven-face development \(P_0\cup XP_0\cup X^2P_0\) and the four-face ideal tetrahedron \(\mathcal T_{1+i}\).  Upward rays represent the ideal direction \(\infty\).}
\label{fig:gamma0pi-two-domains}
\end{figure}

\begin{remark}[Comparison of the two side-pairing systems]
\label{rem:gamma0pi-seven-four}
The common generators are \(\alpha,\gamma,\delta\), and the exact dictionary is
\[
\tau=\gamma\beta,\qquad
\eta=\gamma\delta,\qquad
\jmath=\delta\beta\delta,\qquad
\beta=\gamma\tau
\]
in \(\PSL_2(\ZZ[i])\).  This is an exact dictionary between the two
side-pairing generating systems; it is not a sequence of legal geometric
cuts, translations, or face mergers from the seven-face domain to
\(\mathcal T_{1+i}\).  No such reduction sequence is presently proved.
\end{remark}

\subsection{Two index-six square pyramids}\label{app:index6-pyramids}
The next two examples restore two small domains from the original notes in the
same normalization as \(\mathcal T_{1+i}\).  They are useful for a different
reason from the reconstruction examples: they show explicitly how internal
faces of a block decomposition disappear from the exposed boundary.

\begin{proposition}[The principal level-$(1+i)$ square pyramid]
\label{prop:gamma-principal-pi-pyramid}
The union
\[
\mathcal T_{1+i}\cup S\mathcal T_{1+i}
\]
is an ideal square-pyramid fundamental polyhedron for \(\Gamma(1+i)\), with
vertices \(\infty,-1,0,1,i\).  The common face \((\infty,0,i)\) is internal,
and the five exposed self-paired faces are
\[
\begin{array}{c|c|c}
\text{pairing}&\text{exposed face}&\text{ordered image}\\ \hline
u_C=\begin{psmallmatrix}-i&-1+i\\0&i\end{psmallmatrix}&(\infty,i,1)&(\infty,1,i)\\[2mm]
u_D=\begin{psmallmatrix}-i&0\\-1-i&i\end{psmallmatrix}&(0,1,i)&(0,i,1)\\[2mm]
u_E=\begin{psmallmatrix}-i&0\\0&i\end{psmallmatrix}&(\infty,-1,0,1)&(\infty,1,0,-1)\\[2mm]
u_F=\begin{psmallmatrix}-i&0\\-1+i&i\end{psmallmatrix}&(0,-1,i)&(0,i,-1)\\[2mm]
u_G=\begin{psmallmatrix}i&1+i\\0&-i\end{psmallmatrix}&(\infty,-1,i)&(\infty,i,-1).
\end{array}
\]
The four edge-cycle words are
\[
u_Gu_Cu_E,\qquad u_Fu_Du_E,\qquad u_Cu_D,\qquad u_Fu_G,
\]
each of order two and developed angle \(2\pi\).  The cusp orbits are
\[
\{\infty\},\qquad\{0\},\qquad\{-1,1,i\},
\]
and, after normalization at a cusp, every translation lattice is
\((1+i)\ZZ[i]\); each stabilizer also contains a half-turn.  Thus all three
links are complete Euclidean pillowcases, and
\[
\begin{aligned}
\Gamma(1+i)=\langle u_C,u_D,u_E,u_F,u_G\mid {}&
 u_C^2=u_D^2=u_E^2=u_F^2=u_G^2=1,\\
 & (u_Gu_Cu_E)^2=(u_Fu_Du_E)^2=(u_Cu_D)^2=(u_Fu_G)^2=1\rangle.
\end{aligned}
\]
\end{proposition}

\begin{proof}
The two tetrahedra meet exactly in \((\infty,0,i)\).  The displayed matrices
have determinant one, lie in \(\Gamma(1+i)\), reverse the corresponding
face coorientations and satisfy the stated ordered maps.  The four complete
edge cycles have total developed angle \(2\pi\), and the cusp links above are
complete.  The Poincar\'e polyhedron theorem therefore applies.  The union has
six Picard-cell volumes, equal to \([G:\Gamma(1+i)]\), so the side-pairing
group is the whole principal subgroup.
\end{proof}

\begin{proposition}[The level-two square pyramid]
\label{prop:gamma0-two-pyramid}
Put
\[
\eta_2=\begin{pmatrix}-i&0\\-1-i&i\end{pmatrix}.
\]
The union
\[
\mathcal T_{1+i}\cup\eta_2\mathcal T_{1+i}
\]
is an ideal square-pyramid fundamental polyhedron for \(\Gamma_0(2)\), with
vertices \(\infty,0,1,i,w\), where \(w=(1+i)/2\).  The common face
\((0,1,i)\) lies on the internal hemisphere
\[
(x-\tfrac12)^2+(y-\tfrac12)^2+t^2=\tfrac12,
\]
and is not an exposed sixth face.  The five exposed self-paired faces are
\[
\begin{array}{c|c|c}
\text{pairing}&\text{exposed face}&\text{ordered image}\\ \hline
v_A=\begin{psmallmatrix}-i&i\\0&i\end{psmallmatrix}&(\infty,0,1)&(\infty,1,0)\\[2mm]
v_B=\begin{psmallmatrix}-i&-1\\0&i\end{psmallmatrix}&(\infty,0,i)&(\infty,i,0)\\[2mm]
v_C=\begin{psmallmatrix}-i&-1+i\\0&i\end{psmallmatrix}&(\infty,i,w,1)&(\infty,1,w,i)\\[2mm]
v_D=\begin{psmallmatrix}-1&i\\2i&1\end{psmallmatrix}&(0,i,w)&(i,0,w)\\[2mm]
v_E=\begin{psmallmatrix}1&-1\\2&-1\end{psmallmatrix}&(0,1,w)&(1,0,w).
\end{array}
\]
The four edge-cycle words are
\[
v_Av_Cv_B,\qquad v_Cv_Dv_E,\qquad v_Av_E,\qquad v_Bv_D,
\]
each of order two and developed angle \(2\pi\).  The cusp representatives
\(\infty,w,0\) have translation lattices
\(\ZZ[i],\ZZ[i],2\ZZ[i]\), respectively, and each cusp stabilizer also
contains a half-turn.  Hence all three links are complete Euclidean
pillowcases, and
\[
\begin{aligned}
\Gamma_0(2)=\langle v_A,v_B,v_C,v_D,v_E\mid {}&
 v_A^2=v_B^2=v_C^2=v_D^2=v_E^2=1,\\
 & (v_Av_Cv_B)^2=(v_Cv_Dv_E)^2=(v_Av_E)^2=(v_Bv_D)^2=1\rangle.
\end{aligned}
\]
\end{proposition}

\begin{proof}
Exact reduction modulo \((2)\) gives \([G:\Gamma_0(2)]=6\).  The five
matrices have determinant one, lie in \(\Gamma_0(2)\), reverse the stated
face coorientations and give the displayed ordered maps.  The four edge
cycles and the three complete cusp links satisfy the Poincar\'e conditions.
The polyhedron has index-six volume, so its side-pairing group is all of
\(\Gamma_0(2)\).
\end{proof}

\begin{figure}[htbp]
\centering
\includegraphics[width=.96\textwidth]{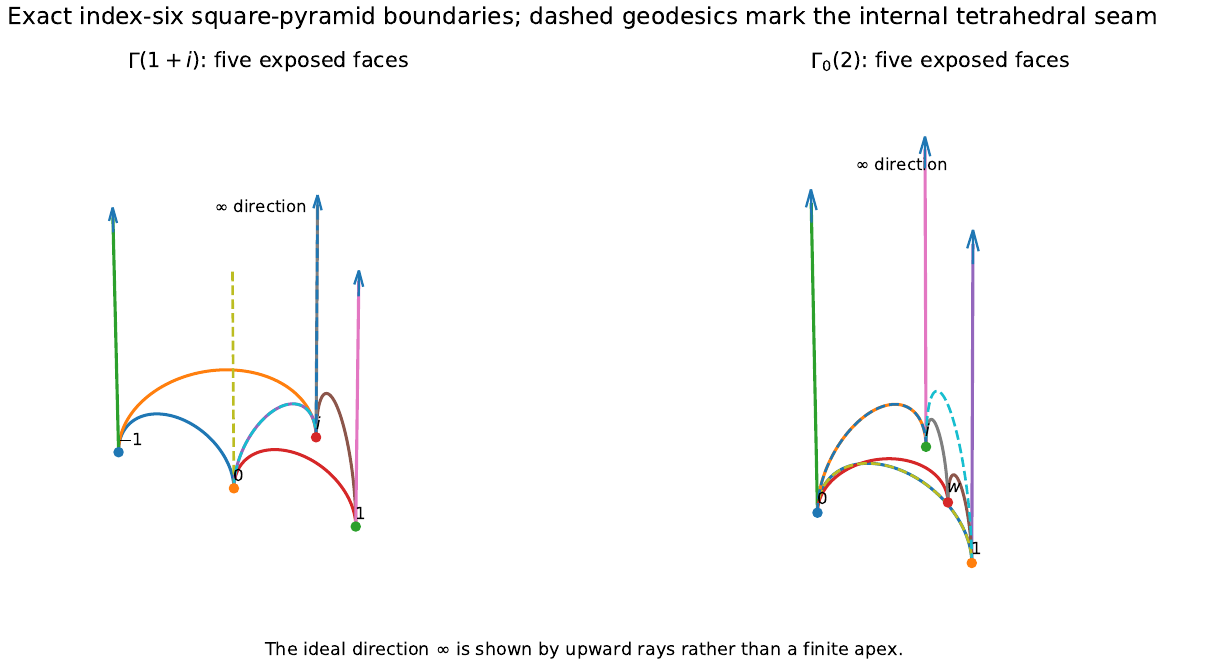}
\caption{The two exact index-six square-pyramid domains.  Dashed geodesics mark the internal tetrahedral seam: \((\infty,0,i)\) for \(\Gamma(1+i)\), and the face \((0,1,i)\) on \((x-1/2)^2+(y-1/2)^2+t^2=1/2\) for \(\Gamma_0(2)\).}
\label{fig:index6-square-pyramids}
\end{figure}

\subsection{The three one-octahedron index-twelve groups}\label{app:octahedra}
Lee's two torsion-free index-twelve groups and Hockman's torsion index-twelve group use the same regular ideal octahedron
\[
\OO=\left\{\infty,0,1,i,1+i,\frac{1+i}{2}\right\},
\]
with all dihedral angles \(\pi/2\).  For all three groups the octahedral quotient has one occurrence and \(J=1\), so Theorem~\ref{thm:farey-reconstruction} expands it to all twelve Picard cells.  The groups are distinguished by the ordered face pairings.

\begin{center}\small
\begin{tabular}{lccc}
\toprule
example & octahedral local group & torsion & cusp information\\
\midrule
Lee I & $J=1$ & torsion-free & two torus cusps\\
Lee--Whitehead & $J=1$ & torsion-free & two torus cusps; $H_1\cong\ZZ^2$\\
Hockman & $J=1$ & torsion present & three cusp orbits\\
\bottomrule
\end{tabular}
\end{center}
The common undecorated octahedron therefore does not distinguish the three
subgroups; their ordered face pairings and lower-dimensional cycle data do.

Lee's first class has four side-pairing generators \(\alpha,\beta,\gamma,\delta\) and three four-step identity edge-cycle relations
\[
\beta\alpha^{-1}\beta^{-1}\alpha=1,
\qquad
\delta^{-1}\gamma\delta^{-1}\alpha=1,
\qquad
\gamma^{-1}\delta^{-1}\gamma^{-1}\beta=1.
\]
Its two cusp links are tori.  Lee's second class, identified with the Whitehead-link group, may be written
\[
\langle\alpha,\beta,\gamma,\delta\mid
\delta\alpha^{-1}\beta^{-1}\alpha,
\ \beta^{-1}\gamma\delta^{-1}\alpha,
\ \gamma^{-1}\delta^{-1}\gamma\beta\rangle,
\]
and also has two torus cusps \cite{Lee1984}.

\begin{proposition}[The Lee--Whitehead octahedron in the positive normalization]
\label{prop:lee-whitehead-exact}
Put \(w=(1+i)/2\), and retain the outward oriented faces
\[
\begin{array}{llll}
F_1=(\infty,0,1),&F_2=(\infty,1,1+i),
 &F_3=(\infty,1+i,i),&F_4=(\infty,i,0),\\
F_5=(0,w,1),&F_6=(1,w,1+i),
 &F_7=(1+i,w,i),&F_8=(i,w,0).
\end{array}
\]
For Lee's second torsion-free index-twelve class, source-to-target face
pairings in the convention of Definition~\ref{def:farey-symbol} are
\[
\begin{array}{c|c|c}
q&F\longrightarrow F'&q(F)\\ \hline
\displaystyle
\alpha=\begin{pmatrix}1&-1-i\\1-2i&-2+i\end{pmatrix}
 &F_6\longrightarrow F_8&(w,i,0)=\overline{F_8}\\[8pt]
\displaystyle
\beta=\begin{pmatrix}1&0\\-1&1\end{pmatrix}
 &F_5\longrightarrow F_4&(0,i,\infty)=\overline{F_4}\\[8pt]
\displaystyle
\gamma=\begin{pmatrix}-i&1+i\\-1&1\end{pmatrix}
 &F_1\longrightarrow F_3&(i,1+i,\infty)=\overline{F_3}\\[8pt]
\displaystyle
\delta=\begin{pmatrix}-i&2i\\-1&2+i\end{pmatrix}
 &F_2\longrightarrow F_7&(i,w,1+i)=\overline{F_7}.
\end{array}
\]
The reverse face pairings are represented by the literal inverses.  The
octahedral quotient has three four-edge cycles, two torus cusps, and
\[
H_1(H;\ZZ)\cong\ZZ^2.
\]
More precisely, at cusp representatives \(1\) and \(0\), respectively, the
complete peripheral translation lattices are
\[
 \Lambda_1=\langle1+i,4\rangle_{\ZZ},
 \qquad
 \Lambda_0=\langle1,2i\rangle_{\ZZ}.
\]
Under Lee's identification this face-paired regular ideal octahedron is the
standard one-octahedron decomposition of the Whitehead-link complement.
\end{proposition}

\begin{proof}
Lee uses the translated octahedron
\(
T^{-1}\OO=\{\infty,0,i,-1,-1+i,(-1+i)/2\}
\), where \(T(z)=z+1\).  Conjugating the four ordered pairings in
Lee's Figure~10 by \(T\) gives the displayed matrices.  Direct substitution
gives the four ordered images above, so each map reverses the stored outward
face orientation exactly as required.

The three unoriented edge classes are
\[
\begin{aligned}
&\{01,\ \infty(1+i),\ (1+i)i,\ \infty0\},\\
&\{0w,\ 0i,\ (1+i)w,\ 1(1+i)\},\\
&\{1w,\ \infty i,\ \infty1,\ wi\}.
\end{aligned}
\]
Their cycle transformations are the identity in
\(\PSL_2(\ZZ[i])\); since each cycle has four angles \(\pi/2\), every
angle sum is \(2\pi\).  The six ideal vertices have two equivalence classes
\[
 \{\infty,1,i,w\},\qquad \{0,1+i\},
\]
and Lee's horospherical gluing gives a torus for each class.

Lee records the peripheral subgroups as
\[
 \langle\beta^{-1}\gamma,
 \beta^{-1}\delta^{-1}\beta\delta\rangle,
 \qquad
 \langle\beta,\alpha\gamma\rangle.
\]
Conjugating the first cusp by
\(R_1=\left(\begin{smallmatrix}0&-1\\1&-1\end{smallmatrix}\right)\)
and the second by
\(R_0=\left(\begin{smallmatrix}0&-1\\1&0\end{smallmatrix}\right)\)
gives
\[
\begin{aligned}
R_1(\beta^{-1}\gamma)R_1^{-1}&=\begin{pmatrix}1&1+i\\0&1\end{pmatrix},&
R_1(\beta^{-1}\delta^{-1}\beta\delta)R_1^{-1}
 &=\pm\begin{pmatrix}1&4\\0&1\end{pmatrix},\\
R_0\beta R_0^{-1}&=\begin{pmatrix}1&1\\0&1\end{pmatrix},&
R_0(\alpha\gamma)R_0^{-1}&=\begin{pmatrix}1&2i\\0&1\end{pmatrix}.
\end{aligned}
\]
This proves the lattice assertions.

Finally, abelianizing Lee's presentation gives
\[
 [\delta]=[\beta],\qquad [\alpha]+[\gamma]=2[\beta].
\]
The Smith invariants are \(1,1\), hence \(H_1(H;\ZZ)=\ZZ^2\), with basis
\(
m_1=[\beta^{-1}\gamma],\ m_0=[\beta]
\).
The first peripheral longitude
\(
\beta^{-1}\delta^{-1}\beta\delta
\)
has zero class.  On the second cusp the primitive element
\(
(\alpha\gamma)\beta^{-2}
\)
also has zero class and, together with \(\beta\), is a lattice basis.
Thus the two primitive meridians give the two homology generators and the
primitive longitudes vanish, as expected for the two-component
Whitehead-link complement.  Lee's identification is
\cite[\S3.2, p.~188]{Lee1984}.
\end{proof}

\begin{remark}
The marked Picard Farey symbol determines the face-paired ideal cell complex,
edge cycles, cusp-link tori, peripheral subgroups and lattices, the Poincar\'e
presentation, and its homology.  Naming the resulting manifold as the
Whitehead-link complement uses Lee's identification; a canonical embedding in
\(S^3\) and canonical meridian--longitude labels are not part of the unaugmented
symbol.
\end{remark}

\begin{figure}[htbp]
\centering
\includegraphics[width=.90\textwidth]{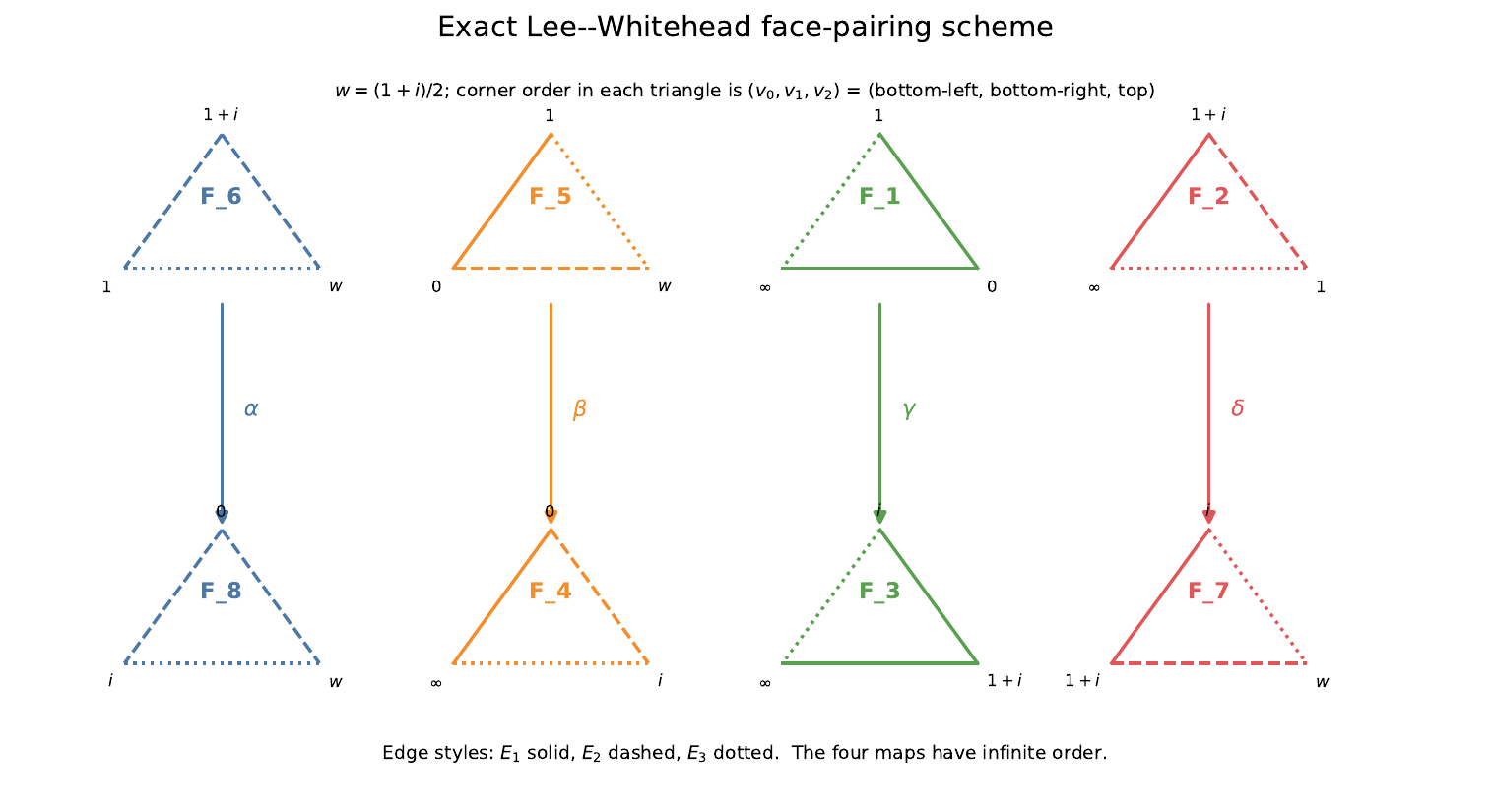}
\caption{Exact Lee--Whitehead face pairings in the positive normalization, with \(w=(1+i)/2\).  Solid, dashed and dotted edges record the three four-edge cycles; the corner order in each triangle is the ordered outward face triple.}
\label{fig:lee-whitehead-pairings}
\end{figure}

Hockman's subgroup is
\[
H_{\rm Hoc}=\langle\tau_1,\tau_i,\delta_1,\delta_i\rangle.
\]
Writing \(w=(1+i)/2\), the four unoriented octahedral edge orbits may be represented by
\[
\begin{array}{c|c|c}
\text{orbit}&\text{ordered traversal}&\text{cycle transformation}\\
\hline
\text{vertical}&[1,\infty]\to[1+i,\infty]\to[i,\infty]\to[0,\infty]&1\\
\text{equatorial I}&[0,1]\leftrightarrow[i,1+i]&(\tau_{-i}\delta_1)^2\\
\text{equatorial II}&[0,i]\leftrightarrow[1,1+i]&(\tau_{-1}\delta_i)^2\\
\text{floor}&[i,w]\to[1+i,w]\to[1,w]\to[0,w]&1.
\end{array}
\]
Thus there are two equatorial order-two cycles and two four-edge identity cycles.  The opposite based edges give conjugate order-two words.  Hockman's group therefore has torsion even though its octahedral local group \(J\) is trivial.  Its cusp action has three orbits.  This example is the simplest warning that top-dimensional local groups do not determine lower-dimensional torsion.

\subsection{\texorpdfstring{The level-\((2+i)\) congruence family}{The level-(2+i) congruence family}}\label{app:level5}
Let \(\mathfrak p=(2+i)\).  Reduction modulo \(\mathfrak p\) sends \(i\) to \(3\in\mathbf F_5\) and identifies the finite quotient with \(\PSL_2(\mathbf F_5)\) of order sixty.  The Borel subgroup, the upper-unipotent subgroup, and the identity subgroup give \(\Gamma_0(\mathfrak p),\Gamma_1(\mathfrak p),\Gamma(\mathfrak p)\), with indices six, twelve and sixty.

For \(\Gamma_0(\mathfrak p)\), one octahedral occurrence has
\[
J=\{1,c^2a\}=\langle c^2a\rangle,
\qquad
c^2a=\begin{pmatrix}-i&-1+i\\0&i\end{pmatrix}.
\]
A connected transversal for \(J\backslash K\) is
\[
1,c,c^2,ca,aca,a.
\]
The four face classes are those of \eqref{eq:g0-faceclasses}.  The exact finite-edge calculation detects order-two torsion.

For \(\Gamma_1(\mathfrak p)\), \(J=1\), and the four unoriented pairings are \eqref{eq:g1-facepairs}.  For \(\Gamma(\mathfrak p)\), five representatives are
\[
I,\ S,\ R,\ X^{-1}S=\begin{pmatrix}1&0\\1&1\end{pmatrix},\
Y^{-1}R=\begin{pmatrix}1&0\\-i&1\end{pmatrix}.
\]
Every local group is trivial.  The complete occurrence/face table is summarized by the following adjacency table, in which the entry \(O_j:Q\) means that the indicated source face is paired with face \(Q\) of \(O_j\):
\[
\begin{array}{c|cccccccc}
 &A&B&C&D&E&F&G&H\\\hline
O_0&O_2:A&O_4:G&O_3:F&O_1:D&O_3:D&O_1:F&O_2:G&O_4:A\\
O_1&O_4:H&O_2:B&O_3:A&O_0:D&O_3:G&O_0:F&O_4:B&O_2:H\\
O_2&O_0:A&O_1:B&O_3:H&O_4:E&O_3:B&O_4:C&O_0:G&O_1:H\\
O_3&O_1:C&O_2:E&O_4:F&O_0:E&O_4:D&O_0:C&O_1:E&O_2:C\\
O_4&O_0:H&O_1:G&O_2:F&O_3:E&O_2:D&O_3:C&O_0:B&O_1:A
\end{array}
\]
Each pair of distinct rows is connected by two unoriented face classes, giving the multigraph \(2K_5\) of equation~\eqref{eq:2k5}.  Expansion by the fixed twelve-cell table gives \(6,12,60\) Picard cells in the three cases, and all \(312\) transitions under \(S,R,X,Y\) agree with the direct finite-quotient actions.

\begin{proposition}[A $30$-triangle development for $\Gamma(2+i)$]
\label{prop:gamma2i-thirty}
For $H=\Gamma(2+i)$ the five octahedra
\[
\mathcal O,\quad \mathcal O-1,\quad \mathcal O-i,\quad
\mathcal O-1-i,\quad \mathcal O+i
\]
represent the five distinct elements of $H\backslash G/K$ and form a
connected whole-octahedron fundamental domain.  Their adjacency graph is the
four-cycle
\[
\mathcal O-(\mathcal O-1)-(\mathcal O-1-i)-(\mathcal O-i)-\mathcal O
\]
together with the additional edge $\mathcal O-(\mathcal O+i)$.  Hence there
are five internal triangular face-pairs and
\[
b_\triangle=8\cdot5-2\cdot5=30,
\]
improving the generic spanning-tree value $32$.

An exact exhaustive search through the $1640$ connected rooted
five-octahedron transversals generated from $\mathcal O$ finds $1520$ with
four internal face-pairs and $120$ with five, and none with six or more.
Thus $30$ is minimal in this searched complete-octahedron class.  No global
minimum is asserted after allowing non-octahedral cuts or face mergers.
\end{proposition}

\subsection{The two-octahedron example \texorpdfstring{$\Gamma_0(3)$}{Gamma0(3)}}\label{app:gamma03}
Since \(3\) is inert in \(\ZZ[i]\), reduction identifies the left-coset set
\[
\Gamma_0(3)\backslash G\cong\PP^1(\mathbf F_9),\qquad |\PP^1(\mathbf F_9)|=10,
\]
with the following right \(G\)-action convention:
\[
[u:v]\begin{pmatrix}a&b\\c&d\end{pmatrix}=[ua+vc:ub+vd].
\]
The action of \(K\) has two orbits:
\[
A=\{0,\infty,-1,i,-1+i,1-i\},\qquad
B=\{1,1+i,-1-i,-i\}.
\]
Take
\[
g_A=1,\qquad g_B=\begin{pmatrix}1&0\\1&1\end{pmatrix}.
\]
Thus the two entries of the arithmetic Picard Farey symbol are the Gaussian-rational octahedra
\[
   \OO_A=\OO,
   \qquad
   \OO_B=g_B\OO
   =\left(1,0,\frac12,\frac{1+i}{2},\frac{3+i}{5},\frac{2+i}{5}\right),
\tag{A.1}\label{eq:gamma03-rational-octahedra}
\]
with the order inherited from \((\infty,0,1,i,1+i,w)\).  This makes visible the arithmetic content that is suppressed when the two occurrences are denoted only by \(A\) and \(B\).
Then
\[
J_A=\{1,c^2a\}=\langle c^2a\rangle,\qquad |J_A|=2,
\]
\[
J_B=\{1,c,c^2\}=\langle c\rangle,\qquad |J_B|=3.
\]
Connected transversals may be chosen as
\[
1,c,c^2,ca,aca,a
\]
for \(J_A\backslash K\), and
\[
1,a^2c^2,ca,a^2
\]
for \(J_B\backslash K\).

The occurrence/fibre diagram appears with the reconstruction examples in
Figure~\ref{fig:gamma03-two-octahedra}.  We now record the exact local face
maps used by that diagram.

The stored global return elements $h\in\Gamma_0(3)$ determine the normalized
local source-to-target maps by
\[
q=g_y^{-1}h^{-1}g_x.
\]
For the eight face classes the resulting $q$-matrices are as follows, with
$w=(1+i)/2$:
\[
\begin{array}{c|c|c}
\text{source}&\text{target}&\text{normalized local map }q\\ \hline
A:(0,1,\infty)&A:(0,1,\infty)&\begin{psmallmatrix} i&-i\\0&-i\end{psmallmatrix}\\[1mm]
A:(1,1+i,\infty)&A:(i,0,\infty)&\begin{psmallmatrix}-1&1\\0&-1\end{psmallmatrix}\\[1mm]
A:(0,w,1)&B:(i,0,\infty)&\begin{psmallmatrix}-1&0\\1&-1\end{psmallmatrix}\\[1mm]
A:(1,w,1+i)&B:(i,w,0)&\begin{psmallmatrix}1-i&-1\\-1-2i&2i\end{psmallmatrix}\\[1mm]
B:(0,1,\infty)&B:(0,1,\infty)&\begin{psmallmatrix}0&-i\\-i&0\end{psmallmatrix}\\[1mm]
B:(1+i,i,\infty)&A:(i,w,0)&\begin{psmallmatrix}-i&-1\\-1-i&-1+2i\end{psmallmatrix}\\[1mm]
B:(i,0,\infty)&A:(0,w,1)&\begin{psmallmatrix}-1&0\\-1&-1\end{psmallmatrix}\\[1mm]
B:(1+i,w,i)&B:(1+i,w,i)&\begin{psmallmatrix}1+2i&1-2i\\2&-1-2i\end{psmallmatrix}
\end{array}
\]
Applying each displayed $q$ to the ordered source gives a reverse cyclic
ordering of the stated target, which is the fixed outward-coorientation
convention.  The matrices $q$ are local chart maps and need not themselves
lie in $\Gamma_0(3)$; membership in the subgroup belongs to the global
return elements $h$.  Every transition reverses the induced boundary
coorientation.  After expansion there are ten Picard cells, eighteen internal facet pairs and twenty-two exposed oriented face pairings.  All forty facet pairs are involutive at complete-flag level.  The resulting permutations are
\[
\begin{aligned}
S={}&(A[1]\ A[c])(A[c^2]\ B[1])(A[aca]\ B[ca])(A[a]\ B[a^2c^2]),\\
R={}&(A[1]\ A[c])(A[ca]\ B[a^2])(A[aca]\ B[a^2c^2])(A[a]\ B[ca]),\\
X={}&(A[1]\ A[c]\ A[c^2])(A[ca]\ A[aca]\ A[a])(B[a^2c^2]\ B[a^2]\ B[ca]),\\
Y={}&(A[1]\ A[c]\ A[ca])(A[c^2]\ A[aca]\ A[a])(B[1]\ B[a^2c^2]\ B[ca]),
\end{aligned}
\]
with omitted points fixed.  These agree entry-by-entry with the direct congruence action on \(\PP^1(\mathbf F_9)\).

\end{document}